\documentclass[11pt,a4paper,twoside]{article}
\usepackage{amssymb}
\usepackage{amsfonts}
\usepackage{geometry}
\usepackage{color}
\usepackage{mathtools}
\usepackage{amsmath,amsthm,amssymb,amsfonts,enumitem}
\usepackage{CJK}
\usepackage{latexsym}
\usepackage{graphicx}
\usepackage{pgfpages}
\usepackage{mathrsfs}
\usepackage{ytableau}
\usepackage{enumerate}
\usepackage{cite}
\usepackage[hyphens]{url}
\usepackage[bookmarks=false,hidelinks]{hyperref}
\usepackage[affil-it]{authblk}
\usepackage{appendix}
\usepackage{amsmath}
\allowdisplaybreaks

\newtheorem{definition}{Definition}[section]
\newtheorem{theorem}[definition]{Theorem}
\newtheorem{lemma}[definition]{Lemma}
\newtheorem{proposition}[definition]{Proposition}
\newtheorem{corollary}[definition]{Corollary}

\newtheorem{condition}{Assumption}

\newtheorem{example}[definition]{Example}

\numberwithin{equation}{section}

\begin{document}
\title{Large deviations principle for invariant measures of stochastic Burgers equations
}
	\author[a]{Rui Bai}
	\author [a] {Chunrong Feng}
	\author[a,b]{Huaizhong Zhao}
	\affil[a]{Department of Mathematical Sciences, Durham University, DH1 3LE, UK}
	\affil[b] {Research Centre for Mathematics and Interdisciplinary Sciences, Shandong University, Qingdao 266237, China}
	
	\affil[ ]{rui.bai@durham.ac.uk, chunrong.feng@durham.ac.uk,  huaizhong.zhao@durham.ac.uk}
	\date{}
	
\maketitle

\begin{abstract}
	 We study the small noise asymptotic for stochastic Burgers equations on $(0,1)$ with the Dirichlet boundary condition. We consider the case in which the noise is more singular than space-time white noise. We let the noise magnitude $\sqrt{\epsilon} \rightarrow 0$ and the covariance operator $Q_\epsilon$ converge to $(-\Delta)^{\frac 1 2}$ and prove the large deviations principle for solutions, uniformly with respect to the initial value of the equation. Furthermore, we set $Q_\epsilon$ to be a trace class operator and converge to $(-\Delta)^{\frac{\alpha}{2}}$ with $\alpha<1$ in a suitable way such that the invariant measures exist. Then, we prove the large deviations principle for the invariant measures of stochastic Burgers equations.
    \medskip
	
	\noindent

	{\bf MSC2020 subject classifications:} Primary 60H10, 60B10; secondary 37A50.

\end{abstract}

\tableofcontents

\section{Introduction}

In this paper, we consider the small noise large deviations principle (LDP) for stochastic Burgers equation (SBE)
\begin{equation}
	\label{sbe1}
\partial_t u_\epsilon (t,\xi ) = \Delta u_\epsilon (t,\xi) + \frac{1}{2}D(u^2_\epsilon (t,\xi )) + \sqrt{\epsilon Q_{\epsilon}}\partial_t W(t,\xi), \ t \geq 0 , \ \xi \in [0,1]
\end{equation}
with Dirichlet boundary condition
\begin{equation*}
	u_\epsilon(t,0) = u_\epsilon(t,1)= 0,\  t \geq 0.
\end{equation*}
The function $u_\epsilon(t,\xi)$ takes value in $\mathbb{R}$. The stochastic perturbation $\partial_t W(t,\xi)$ is a space-time white noise. The operator $Q_{\epsilon}^{\frac{1}{2} }$ depends on the noise magnitude $\epsilon$ and will be defined later. In the present paper, we are interested in the behavior of Eq. (\ref{sbe1}) as $\epsilon$ goes to 0.

The above equation plays an important role in fluid dynamics as a simple model for turbulence. This stochastic model has been intensively studied for decades, see, e.g., Da Prato, Debussche and Temam\cite{Da_Prato_1994} for the well-posedness of mild solution; and Goldys and Maslowski\cite{goldys2005exponential}, Gourcy\cite{gourcy2007large} for invariant measures. About the large deviations for the SBE driven by space-time white noise, small noise asymptotic for the law of solutions were investigated by Cardon-Weber \cite{CARDONWEBER199953}. In Mattingly, Romito and Su\cite{Mattingly_Romito_Su_2022}, the authors investigated the equivalence between the law of SBE solution and the law of stochastic Heat equation when $\sqrt{Q} \approx (-\Delta)^{\frac{\alpha }{2}}$, where $\alpha <1$. It is well known that when $\alpha < \frac{1}{2}$, the corresponding standard energy estimates guarantee the existence of a unique global solution. When $\alpha \geq \frac{1}{2}$ , the solution is no longer a function, but a distribution. However, if the convergence of $\sqrt{Q_\epsilon }$ to $(-\Delta)^{\frac{1}{4}}$ is sufficiently slow, it may be possible to show that $u_\epsilon $ converges to the deterministic solution of the equation without noise since the noise magnitude converges to zero faster.

Motivated by the scope described above, with a suitable choice of $\sqrt{Q_\epsilon }$ that converges to $(-\Delta)^{\frac{1}{4}}$, we prove a uniformly large deviations principle for the law of solutions to Eq. (\ref{sbe1}). Here, $\sqrt{Q_\epsilon }$ can be non-trace class operator. Furthermore, the noise operator is chosen to be $\sqrt{Q_\epsilon}=(-\Delta)^{\alpha/2}\sqrt{Q_{\delta(\epsilon)}}$ for $\alpha < 1/2$, where $\sqrt{Q_{\delta(\epsilon)}}\partial_tW(t,\xi)$ formally converges to the space-time white noise in a smooth enough way such that there exists a unique invariant measure $\nu _\epsilon $ for each $\epsilon $ (see \cite{goldys2005exponential}). It can be shown that if the noise magnitude converges to zero fast enough, these measures converge weakly to the Dirac measure at $0$. In this connection, the aim is to prove a LDP for $\nu _\epsilon $.

The small noise LDP and ULDP for the law of solutions have been established for a variety of SPDEs, including stochastic Navier-Stokes equations (Bessaih and Millet\cite{bessaih2009large}, Chang\cite{CHANG199665}, Sritharan and Sundar\cite{SRITHARAN20061636}), stochastic reaction-diffusion equations (Cerrai and Debussche\cite{cerrai2019largea}, Sowers\cite{sowers1992largea}) and nonlinear Schrödinger equation (Gautier\cite{GAUTIER20051904}). The ULDP for a general class SPDEs has also been investigated deeply by many works (see e.g. Salins\cite{salins2019uniform}, Salin, Budhiraja and Dupuis \cite{salins2019equivalences}). For equations with variable covariance noise, the weak convergence method developed in Budhiraja and Dupuis\cite{budhiraja2000variational}, Budhiraja, Dupuis and Maroulas \cite{budhiraja2008large}, Duluis and Ellis \cite{dupuis2011weak} is widely used. In Cerrai and Debussche \cite{cerrai2019large}, the authors used a weak convergence method to prove the LDP for the 2D stochastic Navier-Stokes equation with noise that converges to space-time white noise. The study of LDP for invariant measures of infinite-dimensional stochastic systems began in the work of Sowers \cite{sowers1992large}. In Cerrai and R{\"o}ckner \cite{cerrai2005large}, the authors proved the LDP for invariant measures of stochastic reaction-diffusion equations with multiplicative noise under suitable assumptions. In Martirosyan \cite{martirosyan2017large}, the author studied the LDP for stationary measures of stochastic nonlinear wave equations. The LDP for invariant measures of singular SPDEs has also been studied recently. In Klose and Mayorcas \cite{klose2024large}, the authors combined the LDP techniques developed in \cite{sowers1992large,cerrai2005large} with the regularity structures theory to prove the LDP for the $\Phi^4_3$ measure. 

In recent years, the LDP for invariant measures of the 2D stochastic Navier-Stokes equations was established by Brze{\'z}niak and Cerrai in \cite{brzezniak2017large,brzezniak2015quasipotential}. In \cite{brzezniak2015quasipotential}, they showed that for the fixed trace class noise operator $\sqrt{Q_\delta }$, the quasi-potential $U_\delta(x)$ of the stochastic Navier-Stokes equation has compact level sets, which means that $U_\delta (x)$ can be a good rate function of LDP. They also proved that if the problem is considered on the torus, the quasi-potential $U_\delta (x)$ converges pointwise to $U(x)= \| x\Vert _{[H^1(\mathbb{T} ^2)]^2}$ when $\sqrt{Q_\delta }$ converges to $\sqrt{Q_0}= I $. In \cite{brzezniak2017large}, it was proven that the invariant measures satisfy a LDP in $[L^2(\mathbb{T} ^2)]^2$ with a good rate function $U_\delta $. The paper \cite{cerrai2022large} bridged the results of \cite{brzezniak2015quasipotential,brzezniak2017large} with the result of \cite{cerrai2019large}. In that paper, the authors took the noise magnitude converging to $0$ and $\sqrt{Q_\epsilon }$ to the identity operator simultaneously. They showed that the invariant measures satisfy the LDP with a good rate function $U(x)$. The first step in proving the LDP for invariant measures is to prove an ULDP for the law of the solution. In \cite{cerrai2022large}, Cerrai and Paskal achieved that by first proving the LDP for a linearized problem using a weak convergence approach. After that, the ULDP is followed from the uniform contraction principle. In order to apply the uniform contraction principle, exponential tightness for the law of stochastic convolution is needed in their method.

For the problem of LDP for the law of solutions of SBEs, 
there have been a small number of works for space-time white noise in the literature (see e.g. \cite{CARDONWEBER199953}).
Our noise is allowed to be more singular than space-time white noise and our result is about ULDP. For any fixed $\epsilon>0$, the covariance operator $Q_\epsilon$ is of order $(-\Delta)^{\alpha_\epsilon/2}$, $\alpha_\epsilon < 1$. But when $\epsilon \rightarrow 0$, roughly speaking, $Q_\epsilon$ converges to an operator of order $(-\Delta)^{1/2}$. This is the best result one could achieve as long as the large deviations principle on a function space is studied, since if $Q_\epsilon$ is of order $(-\Delta)^{\alpha/2} $, $\alpha \geq 1$, the solution of SBE is no longer a function, but a distribution. 
To obtain this optimal result, the method in Cerrai and Paskal \cite{cerrai2022large} is hard to work here. This is because Itô's formula is no longer applicable since $Q_\epsilon$ could be of a non-trace class. Hence, it is difficult to obtain the exponential estimate for the law of stochastic convolution that is needed in their method. So we first prove the LDP for the stochastic convolution. Then, we use this result to prove the ULDP for the law of solutions of SBEs by two different methods. The first is the equicontinuous uniform Laplace principle approach introduced by Salins in \cite{salins2019equivalences}. The second is a modification of the uniform contraction principle as we put in Section \ref{sec: ULDP}. 
An exponential estimate is not necessary for both methods. 

In the second part of the paper, we investigate the LDP for invariant measures of SBEs.  
To the best of our knowledge, this is the first paper in the study of the large deviations for invariant measures of this type of equations.  For this, we first choose the covariance operator to be a trace class operator such that the SBE has an invariant measure $\nu_\epsilon$ for any fixed $\epsilon>0$. But $Q_\epsilon =(-\Delta)^{\alpha/2}[I + \delta(\epsilon)(-\Delta)^\beta]^{-1/2} \rightarrow (-\Delta)^{\alpha/2} $ in $D((-\Delta)^{\alpha/2})$, $0\leq \alpha< 1/2$, $1/2 <\beta -\alpha<1$. In this case, $\nu_\epsilon \rightarrow \delta_0$ as $\epsilon \rightarrow 0$, where $\delta_0$ is the Dirac measure. We further obtain the LDP for the invariant measures with the quasi-potential $U^\alpha$ defined by (\ref{quasi potential 1}) as its good rate function.

The difficulty in this part is to show an exponential estimate on a compact subset for the family of invariant measures $v_\epsilon $. In \cite{cerrai2022large}, the compact subset has been chosen to be the bounded set of $D(A^{1/2})$. Then, the desired exponential estimates are based on the property of the 2D SNS equation on torus that 
$$\langle B(u),Au\rangle_H = 0, \ u \in D(A), $$
where $A$ is the Stokes operator and $B$ is the nonlinear operator of the equation. However, such a property does not hold for stochastic Burgers equation. We need a new approach to obtain the exponential estimates. In this paper, we choose the compact subset to be the closed ball of $D((-\Delta)^{2\sigma})$ for some small $\sigma>0$ in $L^2(0,1)$ following the compact Sobolev embedding. Then, we develop a new method in Section \ref{sec:exp esti} to prove the exponential estimates on this subset for the family of invariant measures. This part is highly nontrivial.


The paper is arranged as follows. In the next section, we give the notation of this paper and we recall the definitions of the LDP and the mild solution. In Section \ref{sec: ULDP}, we prove the ULDP for the path of stochastic Burgers equation through two different approaches. The assumption in this section does not need $\sqrt{Q_\epsilon}$ to be of trace class. In Section \ref{sec: Rate Fun}, we study the rate function of LDP for invariant measures. Since $\sqrt{Q_\epsilon} = (-\Delta)^{\alpha/2}\sqrt{Q_\delta(\epsilon)}$ is convergent to $(-\Delta)^{\alpha/2}$ as $\epsilon \rightarrow 0$ for $\alpha<\frac 1 2$, we need to deal with the skeleton equation in a negative Sobolev space. We prove that the quasi-potential $U^\alpha$ defined by (\ref{quasi potential 1}) is a good rate function. In Section \ref{sec:LDP for invariant measure}, we first prove the lower bound of LDP. After that, we prove an exponential estimate for the invariant measure on the closed ball of $D((-\Delta)^{2\sigma})$. This then leads to the LDP upper bound by using the ideas of \cite{sowers1992large,brzezniak2017large,cerrai2022large}. 
\section{Preliminaries}
We define $H := \{u \in L^2(0,1): u(0) = u(1) =0\}$ endowed with the standard $L^2(0,1)$ inner product $\langle \cdot ,\cdot \rangle_H $. The family $\left\{e_k\right\}_{k \in \mathbb{N} }$ defined by
\begin{equation*}
	e_k(\xi) = \sqrt{2}\sin (k \pi \xi) , \ k \in \mathbb{N} , \ \xi \in [0,1],
\end{equation*}
form a complete orthonomal system in $H$. We define the operator $A$ by 
\begin{equation*}
	A := -\Delta, \ D(A) = \{ u \in W^{2,2}(0,1): u(0) = u(1) = 0\} = W^{2,2}(0,1) \bigcap W^{1,2}_0(0,1).
\end{equation*}
It is clear that for each $e_k$
\begin{equation*}
	Ae_k =  (k\pi)^2e_k,
\end{equation*}
where
$A$ is a positive, self-adjoint operator in $H$. For any $r \in \mathbb{R} $ we define its fractional power $A^{\frac{r}{2}}$ with domain $D(A^{\frac{r}{2}})=: H^r = H^r(0,1)$ which is the closure of ${\rm span}_{k\in \mathbb{N}} e_k$ with respect to the norm $\| \cdot\Vert_{H^r} $ given by
\begin{equation*}
	 \| u\Vert _{H^r}^2 := \| u\Vert _{D(A^{\frac{r}{2}})}^2 :=\langle A^{\frac{r}{2}}u, A^{\frac{r}{2}}u\rangle_H = \Sigma _{ k \in \mathbb{N}} (k\pi)^{2r} \langle u,e_k \rangle^2 .
\end{equation*}
We will use the notation 
$L^p := L^p(0,1)$, $p \geq 1$.
Now we can rewrite Eq. (\ref{sbe1}) as a stochastic evolution equation.
\begin{equation}
	\label{sbe2}
	\begin{cases}
		du_\epsilon (t) + Au_\epsilon (t)dt = \frac{1}{2} D(u^2_\epsilon (t)) + \sqrt{\epsilon Q_{\epsilon}} dW(t),\\
		u_\epsilon(0) = x.
	\end{cases}
\end{equation}
Here, $W(t), t \in \mathbb{R} ^+$ is a cylindrical Wiener process in $H$ and has the expansion
\begin{equation*}
	W(t,\xi) = \sum_{k \in \mathbb{N} } e_k(\xi)\beta_k(t),\ t \geq  0, \ \xi \in (0,1),
\end{equation*}
where $\{\beta_k\}_{k \in \mathbb{N} }$ are independent, real-valued Brownian motions on a fixed probability space $(\Omega , \mathcal{F} , \mathbb{P} )$ adapted to a filtration $ \{\mathcal{F}_t\}_{t \geq 0}$. The covariance operator $Q_\epsilon$ is defined by $$\sqrt{Q_\epsilon}e_k = \sigma _{\epsilon ,k} e_k,\ k \in \mathbb{N}. $$ Then the driving noise has the expression 
\begin{equation*}
	\sqrt{ Q_{\epsilon}} W(t,\xi) =  \sum_{k \in \mathbb{N} } \sigma _{\epsilon , k}e_k(\xi)\beta _k(t).
\end{equation*}
Denote by $S(t) = e^{-At},\ t \geq 0$, the semigroup on $H$ generated by $-A$. We define $Z_\epsilon $ to be the mild solution of the linear equation 
\begin{equation}
	\label{OU1}
\begin{cases}
	dZ_\epsilon  + AZ_\epsilon dt = \sqrt{\epsilon Q_{\epsilon}} dW(t),\\
	Z_\epsilon (0)=0.
\end{cases}
\end{equation}
It is well-known that $Z_\epsilon $ has the expression 
\begin{equation}
	\label{OU}
	Z_\epsilon (t) = \int_{0}^{t} S(t-s) \sqrt{\epsilon Q_{\epsilon}} \,dW(s).
\end{equation}
The difference $Y^x_\epsilon (t) = u_\epsilon ^x(t) - Z_\epsilon(t) $ satisfies formally the equation
\begin{equation}
	\label{Y}
	\begin{cases}
		\frac{d }{d t}Y^x_\epsilon (t) + AY^x_\epsilon  = \frac{1}{2}D\left[(Y^x_\epsilon(t) + Z_\epsilon(t))^2 \right],\\
		Y^x_\epsilon (0) = x.
	\end{cases}
\end{equation}
We define $Y_\epsilon ^x$ to be the solution of deterministic integral equation
\begin{equation}
	\label{YM}
	Y^x_\epsilon(t) = S(t)x + \frac{1}{2}\int_{0}^{t} S(t-s)D\left[(Y^x_\epsilon(t) + Z_\epsilon(t))^2 \right] \,ds,\ t \geq 0 .
\end{equation}
\begin{definition}
	A process $u^x_\epsilon$ is said to be a mild solution of Eq. (\ref{sbe2}) if and only if $Y^x_\epsilon = u^x_\epsilon - Z_\epsilon$ is a solution of Eq. (\ref{YM}).
\end{definition}

The aim of this paper is to study the large deviation behavior as $\epsilon $ goes to 0. We recall Freidlin-Wentzell's formulation.
\begin{definition}\label{LDP}
	Let $E$ be a Banach space. Suppose that $\{\mu _\epsilon \}_{\epsilon > 0}$ is a family of probability measures on $E$ and the action function $I:E\rightarrow [0,+\infty]$ is a good rate function, which means that for each $r \geq  0$, the level set $\Phi(r) := \{h \in E: I(h) \leq r\}$ is a compact subset of $E$. The family $\{\mu _\epsilon \}_{\epsilon > 0}$ is said to satisfy a large deviation principle (LDP) in $E$, with the rate function $I$, if the following two conditions hold.
	\item[(1)](Lower bound) For every $r \geq  0$, $\delta > 0$ and $\gamma > 0$, there exists $\epsilon _0 > 0$ such that
	\begin{equation*}
		\mu _\epsilon \left(B_E(\varphi, \delta ) \right) \geq  \exp \left(-\frac{I(\varphi) + \gamma }{\epsilon }\right),
	\end{equation*}
	for any $\epsilon \leq \epsilon _0$ and $\varphi  \in \Phi(r)$, where $B_E(\varphi , \delta) := \left\{h \in E: \| h -\varphi \Vert_E < \delta \right\}$. 
	\item[(2)] (Upper bound) For every $r_0 \geq 0$, $\delta > 0$ and $\gamma > 0$, there exists $\epsilon _0 > 0$ such that 
	\begin{equation*}
		\mu _\epsilon \left(B^c_E(\Phi(r), \delta ) \right) \leq \exp \left(-\frac{r-\gamma }{\epsilon }\right),
	\end{equation*}
	for any $\epsilon \leq  \epsilon _0$ and $r \leq r_0$, where $B^c_E(\Phi(r), \delta ) := \left\{h \in E: dist_E(h,\Phi(r)) \geq \delta \right\}$.
	 
\end{definition}
\section{ULDP for paths}\label{sec:ULDP}
\subsection{Stochastic Burgers equations}
\begin{proposition}
	\label{wellposedness of SBE}
	Assume $\sum_{k \in \mathbb{N}}\frac{\sigma^2 _{\epsilon ,k}}{k^2} < \infty$, then for any $x \in H$, there exists a unique mild solution $u^x_\epsilon $ of Eq. (\ref{sbe2}) and for all $T > 0$, $u_\epsilon ^x \in L^\infty([0,T]; H)$.

\end{proposition}
\begin{proof}
	First we prove the existence and uniqueness of the local solution. We fix $T > 0$ and any $1 \leq p < \infty$. By Burkholder-Davis-Gundy type inequality for stochastic convolutions (see  Proposition 7.3, \cite{da2014stochastic}), we have 
	\begin{align*}
		\mathbb{E} \sup_{0 \leq t \leq T}\| Z_\epsilon (t)\Vert _{L^p}^p &= \epsilon ^{p/2}\mathbb{E} \sup_{t \in [0,T]}\left\| \int_{0}^{t} S(t-s)\sqrt{Q_\epsilon } \,dW(s) \right \Vert_{L^p}^p \\ &\leq c_p\epsilon ^{p/2} \mathbb{E} \left| \int_{0}^{T}\left\| S(t-s)\sqrt{Q_\epsilon }\right\Vert_{HS}^2\, ds \right|^{p/2} \\ &\leq c_p\epsilon ^{p/2} \left| \sum_{k \in \mathbb{N} } \int_{0}^{T} e^{-2k^2\pi^2s}\sigma _{\epsilon, k}^2\,ds \right|^{p/2} \\ &\leq c_p\epsilon ^{p/2}\left(\sum_{k \in \mathbb{N}}\frac{\sigma^2 _{\epsilon ,k}}{k^2}\right)^{p/2} < \infty.
	\end{align*}
	Then we infer that $Z_\epsilon \in C([0,T];L^p)$ almost surely. 
	
	Thanks to Lemma 14.2.1 in \cite{da1996ergodicity} , we know that $F(u)(\cdot) =  \frac{1}{2}\int_{0}^{\cdot} S(\cdot-s)D\left(u(s) \right) \,ds$ is a bounded linear mapping from $C([0,T];L^1)$ into $C([0,T];H)$. Define $F_2(u) = F(u^2)$, then for some $u_\epsilon  \in C([0,T];H)$ and $C >0$	
 \begin{align*}
		\sup_{0 \leq t \leq T} \left\| F_2(u_\epsilon ^x)(t) - F_2(\bar{u}_\epsilon ^x) \right \Vert _{L^2} & \leq \sup_{0 \leq t \leq T} C t^{1/4} \sup_{0 \leq s \leq t} \left\| (u_\epsilon ^x)^2(s) - (\bar{u}_\epsilon ^x )^2(s)\right \Vert _{L^1} \\ & \leq C T^{1/4} \sup_{0 \leq t \leq T} \|u_\epsilon ^x(t) - \bar{u}_\epsilon ^x(t)\Vert _{L^2} \|  u_\epsilon ^x (t)+ \bar{u}_\epsilon ^x(t)\Vert _{L^2} \\ & \leq C T^{1/4} \left\| u_\epsilon ^x - \bar{u}_\epsilon ^x \right\Vert _{L^\infty([0,T];H)}\left(\| u_\epsilon ^x\Vert_{L^\infty([0,T];H)}  + \| \bar{u}_\epsilon ^x\Vert_{L^\infty([0,T];H)} \right) .
	\end{align*}

	We select $\alpha^\prime > 4\| x\Vert _H$. For $\| u_\epsilon ^x\Vert_{L^\infty([0,T];H)}, \| \bar{u}_\epsilon ^x\Vert_{L^\infty([0,T];H)} \leq \alpha' $, we can choose $T(\omega )$ such that 
	\begin{equation*}
		2CT(\omega )^{1/4}\alpha' \leq 1/2 , \  \sup_{0 \leq t \leq T(\omega )}\|S(t)x + Z_\epsilon (t) \Vert_H \leq \frac{1}{2}\alpha' .
	\end{equation*}
	Since \begin{equation}
		\label{uM}
		u_\epsilon ^x(t) = S(t)x + F_2(u_\epsilon ^x)(t) + Z_\epsilon (t),
	\end{equation}
	by Lemma 15.2.4 in \cite{da1996ergodicity}, using a contraction argument on the Banach space $L^\infty([0,T(\omega )];H)$, it is not difficult to prove the existence and uniqueness of a mild solution of Eq. (\ref{sbe2}) in a small interval $[0, T(\omega )]$ depending on $\omega \in \Omega $.

	Next, we will show that this solution is actually global in time. We apply the analogous argument as Lemma 3.1 of \cite{Da_Prato_1994}. We consider the Eq. (\ref{YM}) with $Z_\epsilon $ being replaced by a regular function $\phi \in C([0,T];H^1)$ and $x \in D(A)$
\begin{equation}
	\label{Y1}
	\begin{cases}
		\partial _tY(t)+ AY(t) = \frac{1}{2}D\left[(Y(t) + \phi(t))^2 \right],\\
		Y(0) = x.
	\end{cases}
\end{equation}
It is not difficult to see Eq. (\ref{Y1}) can be solved by a fixed point argument in $C([0,T];H^1_0(0,1))$.

Now we find a prior estimate on the solution $Y$ which does not depend on the regularity assumptions on $\phi $ and $x$.

Multiplying both sides of Eq. (\ref{Y1}) by $Y$ and integrating with respect to $\xi $ in $[0,1]$ and by integration by parts we get
\begin{align*}
	\frac{1}{2}\frac{d}{dt}\int_{0}^{1} |Y(t)|^2 \,d\xi + \int_{0}^{1} Y_\xi^2 (t) \,d\xi  &=  -  \frac{1}{2}\int_{0}^{1} \left(Y(t)+\phi (t)\right)^2 Y_\xi (t) \,d\xi \\
	&= - \frac{1}{2}\int_{0}^{1} Y^2(t)Y_\xi  \,d\xi - \int_{0}^{1} \phi (t)Y(t)Y_\xi (t) \,d\xi - \frac{1}{2}\int_{0}^{1} \phi ^2(t)Y_\xi (t) \,d\xi     
\end{align*}
Notice that 
\begin{equation*}
	\int_{0}^{1} Y^2(t)Y_\xi (t) \,d\xi  = \frac{1}{3}\left[Y^3(t)\right]_0^1 =0.
\end{equation*}
We get 
\begin{equation*}
	\frac{1}{2}\frac{d}{dt}\| Y(t)\Vert _H^2 + \| Y(t)\Vert _{H^1}^2 \leq \| \phi \Vert _{L^{4}}\| Y_\xi (t)\Vert _{L^2}\| Y(t)\Vert _{L^{4}} + \frac{1}{2}\| \phi \Vert_{L^{4}}^2\| Y_\xi(t)\Vert _{H} . 
\end{equation*}
Applying the following interpolation inequality 
\begin{equation*}
	\| Y(t)\Vert _{L^{4}} \leq C_1\| Y(t)\Vert _H^{3/4}\| Y(t)\Vert _{H^1}^{1/4},
\end{equation*}
we have that
\begin{align*}
	\frac{1}{2}\frac{d}{dt}\| Y(t)\Vert _H^2 + \| Y(t)\Vert _{H^1}^2 &\leq C_1\| \phi(t) \Vert _{L^{4}} \| Y(t)\Vert _H^{3/4}\| Y(t)\Vert _{H^1}^{5/4} + \frac{1}{2}\| \phi(t) \Vert _{L^{4}}^2\| Y(t)\Vert_{H^1}\\ &\leq  C_2\| \phi(t) \Vert _{L^{4}}^{8/3} \| Y(t)\Vert _H^2 + \frac{1}{4} \| Y(t)\Vert_{H^1}^2 + \frac{1}{4}\| \phi(t) \Vert _{L^{4}}^4 + \frac{1}{4} \| Y(t)\Vert_{H^1}^2.
\end{align*}
Then 
\begin{equation*}
	\frac{d}{dt}\| Y(t)\Vert _H^2 + \| Y(t)\Vert _{H^1}^2 \leq 2C_2\| \phi(t) \Vert _{L^4}^{8/3} \| Y(t)\Vert _H^2 + \frac{1}{2}\| \phi(t) \Vert _{L^4}^4.
\end{equation*}
By using the Gronwall inequality, we have 
\begin{equation}
	\label{Y estimate 1}
	\| Y(t)\Vert _H^2 \leq \| x\Vert _H^2 \exp\left(2C_2\int_{0}^{t} \| \phi(s) \Vert _{L^4}^{8/3} \,ds \right) + \frac{1}{2} \exp\left(2C_2\int_{0}^{t} \| \phi(s) \Vert _{L^4}^{8/3} \,ds \right) \int_{0}^{t} \| \phi(s) \Vert _{L^4}^4 \,ds
\end{equation}
and \begin{equation}
	\label{Y estimate 2}
	\int_{0}^{t} \| Y(s)\Vert _{H^1}^2 \,ds \leq \int_{0}^{t} \left[2C_2\| \phi(s) \Vert _{L^4}^{8/3} \| Y(s)\Vert _H^2 + \frac{1}{2}\| \phi(s) \Vert _{L^4}^4 \,\right]ds + \|x\Vert _H^2.
\end{equation}
Note estimate (\ref{Y estimate 1}) and (\ref{Y estimate 2}) only depend on $C([0,T];L^4 )$ norm of $\phi$ rather that its $C([0,T];H_0^1(0,1) )$ norm though we assume $\phi \in C([0,T];H_0^1(0,1) ) $ at the beginning. Moreover, $C([0,T];H^1_0(0,1))$ is dense in $C([0,T];L^4)$, thus by approximation method, it is not difficult to see that (\ref{Y estimate 1}), (\ref{Y estimate 2}) hold for $\phi \in C([0,T];L^4 ) $. For the same reason, they hold for $x \in H$. Moreover, we know that for any $T > 0$, the path of $Z_\epsilon $ is in $C([0,T];L^4)$. The estimates (\ref{Y estimate 1}) and (\ref{Y estimate 2}) still holds for any local solution $Y_\epsilon ^x$  of the equation with $\phi $ replaced by $Z_\epsilon $. This shows that local solutions cannot explode in any finite time and we prove the global existence of solutions to Eq. (\ref{sbe2}).

Again by (\ref{Y estimate 1}) and (\ref{Y estimate 2})
\begin{equation}
	\label{Y space}
	Y_\epsilon ^x \in L^\infty ([0,T];H) \bigcap L^2([0,T];H^1)
\end{equation}
which implies 
\begin{equation*}
	u_\epsilon ^x \in L^\infty ([0,T];H).
\end{equation*}
\end{proof}


Next, we prove a lemma which will be used in the later part. We will show that the solution $Y_\epsilon ^x$ to the Eq. (\ref{YM}) is continuous with respect to the OU process $Z_\epsilon $.

For every $x \in H$, let $\mathcal{G} _x : L^4([0,T]; L^4) \rightarrow L^\infty([0,T];H)$ be the mapping that maps any $\phi  \in L^4([0,T];L^4)$ to the solution of the equation 
\begin{equation*}
	Y^x(t) = S(t)x + \frac{1}{2}\int_{0}^{t} S(t-s)D\left[(Y^x(s)+\phi (s))^2\right] \,ds,\ t \in [0,T] .
\end{equation*}
By the proof of Proposition \ref{wellposedness of SBE}, this mapping is well-defined. We have the following result.
\begin{lemma}
	\label{continuous lemma}
	For every fixed $T > 0$, the mappings $\mathcal{G} _x : L^4([0,T]; L^4) \rightarrow L^\infty([0,T];H)$ are Lipschitz continuous on any bounded balls, uniformly in x in any bounded sets of $H$. That is, for any $r > 0$ and $R > 0$, there exists a constant $C_{r,R} > 0$ such that 
	\begin{equation}
		\label{continuous1}
		\sup_{x \in B_H(r)\, \phi ,\varphi  \in B_{L^4([0,T]; L^4)}(R)}\left\| \mathcal{G} _x(\phi ) - \mathcal{G} _x(\varphi ) \right\Vert _{L^\infty([0,T];H)} \leq C_{r,R} \left\| \phi -\varphi \right\Vert _{L^4([0,T]; L^4)}.
	\end{equation}
\end{lemma}

\begin{proof}
	For every $\phi ,\varphi  \in B_{L^4([0,T]; L^4)}(R)$, define $X := \mathcal{G} _x(\phi ) - \mathcal{G} _x(\varphi )$. Equation (\ref{Y1}) implies
	\begin{align*}
		\frac{d}{dt} X(t) &= \frac{d}{dt} \mathcal{G} _x(\phi )(t) -\frac{d}{dt}\mathcal{G} _x(\varphi)(t)\\ &= -AX(t) +  D\left[(\mathcal{G} _x(\phi )(t) + \phi (t))^2 - (\mathcal{G} _x(\varphi  )(t) + \varphi  (t))^2\right]\\ &= -AX(t) + D\left[(\mathcal{G} _x(\phi)(t) + \mathcal{G} _x(\varphi ) (t) + \phi (t) + \varphi (t))(\mathcal{G} _x(\phi)(t) - \mathcal{G} _x(\varphi ) (t) + \phi (t) - \varphi (t))\right]\\ &= -AX(t) + D\left[(\mathcal{G} _x(\phi)(t) + \mathcal{G} _x(\varphi ) (t) + \phi (t) + \varphi (t))(X(t) + \phi (t) - \varphi (t))\right].
	\end{align*}

	Define $\psi (t) = \mathcal{G} _x(\phi)(t) + \mathcal{G} _x(\varphi ) (t) + \phi (t) + \varphi (t)$. Multiplying by $X(t)$ and integrating by $\xi $, we have 
	\begin{align*}
		\frac{1}{2}\frac{d}{dt}\int_{0}^{1} |X(t)|^2 \,d\xi + \|X(t)\Vert _{H^1}^2 &= \int_{0}^{1} D\left[\psi (t)(X(t)+\phi (t)-\varphi (t))\right]X(t) \,d\xi \\ &= -\int_{0}^{1} \psi (t)(X(t)+\phi (t)-\varphi (t))DX(t) \,d\xi \\ &\leq \| \psi (t)\Vert _{L^4}\| X(t)\Vert _{L^4}\| X(t)\Vert _{H^1} + \| \psi(t) \Vert _{L^4}\| \phi (t) - \varphi (t)\Vert _{L^4}\| X(t)\Vert _{H^1}.
	\end{align*}
	It follows from interpolation inequality
	\begin{align*}
		\frac{1}{2}\frac{d}{dt}\| X(t)\Vert _H^2 + \|X(t)\Vert _{H^1}^2 &\leq C_1\| \psi (t)\Vert _{L^4}\| X(t)\Vert _H^{3/4}\| X(t)\Vert _{H^1}^{5/4} + \| \psi(t) \Vert _{L^4}\| \phi (t) - \varphi (t)\Vert _{L^4}\| X(t)\Vert _{H^1} \\ &\leq C_2\| \psi(t) \Vert _{L^4}^{8/3}\| X(t)\Vert _H^2 + \frac{1}{4}\| X(t)\Vert _{H^1}^2 + \| \psi(t) \Vert _{L^4}^2\| \phi (t) - \varphi (t)\Vert _{L^4}^2 + \frac{1}{4}\| X(t)\Vert _{H^1}^2.
	\end{align*}
	Then \begin{equation*}
		\frac{d}{dt}\| X(t)\Vert _H^2 + \|X(t)\Vert _{H^1}^2 \leq 2C_2\| \psi(t) \Vert _{L^4}^{8/3}\| X(t)\Vert _H^2 + 2\| \psi(t) \Vert _{L^4}^2\| \phi (t) - \varphi (t)\Vert _{L^4}^2.
	\end{equation*}
	Since $X(0) = x-x = 0$, it follows from the Gronwall inequality
	\begin{align*}
		\| X(t)\Vert _H^2 &\leq 2\exp\left(2C_2\int_{0}^{t} \| \psi(s) \Vert _{L^4}^{8/3} \,ds \right)\int_{0}^{t} \| \psi (s)\Vert _{L^4}^2\| \phi (s)-\varphi (s)\Vert_{L^4}^2  \,ds \\ &\leq 2\exp\left(2C_2\int_{0}^{t} \| \psi(s) \Vert _{L^4}^{8/3} \,ds \right)\| \psi \Vert _{L^4([0,t]; L^4)}^2\| \phi -\varphi \Vert _{L^4([0,t]; L^4)}^2.
	\end{align*}
	Then 
	\begin{equation}
		\label{continuity estimate}
		\left\| \mathcal{G} _x(\phi ) - \mathcal{G} _x(\varphi ) \right\Vert _{L^\infty([0,T];H)} \leq \sqrt{2}\exp\left(C_2\int_{0}^{T} \| \psi(s) \Vert _{L^4}^{8/3} \,ds \right)\| \psi \Vert _{L^4([0,T]; L^4)}\| \phi -\varphi \Vert _{L^4([0,T]; L^4)}.
	\end{equation}
	By (\ref{Y space}), we know that $\mathcal{G} _x(\phi ),\mathcal{G} _x(\varphi ) \in L^\infty ([0,T];H) \bigcap L^2([0,T];H^1) \subset L^4([0,T];L^4)$ which implies that for every $x \in B_H(r)$ and $\phi ,\varphi  \in B_{L^4([0,T]; L^4)}(R)$\begin{equation*}
		\| \psi \Vert _{L^4([0,T]; L^4)} \leq L_{r,R},
	\end{equation*}
	where $L_{r,R}$ is a constant depending on $r$ and $R$. Combining this with (\ref{continuity estimate}), we obtain (\ref{continuous1}).
\end{proof}

\subsection{ULDP for the stochastic Burgers equations}\label{sec: ULDP}
In this section, we prove that the family $\left\{\mathcal{L} (u_\epsilon )\right\}_{\epsilon > 0}$ satisfies the uniform large deviations principle in $L^\infty ([0,T];H)$. We first recall the Freidlin-Wentzell formulation of uniformly large deviations principle. This definition can also be found in \cite{freidlin2012random} and \cite{salins2019equivalences}.
\begin{definition}
   \label{ULDPFW}
   Let $E$ be a Banach space and let $D$ be some non-empty set. Suppose that for each $x \in D$, $\{\mu^x _\epsilon \}_{\epsilon > 0}$ is a family of probability measures on $E$ and the action function $I^x:E\rightarrow [0,+\infty]$ is a good rate function. The family $\{\mu^x_\epsilon \}_{\epsilon > 0}$ is said to satisfy a Freidlin-Wentzell uniform large deviation principle (FWULDP) in $E$, with rate function $I^x$, uniformly with respect to $x \in D$, if the following hold.
   \item[(1)](Lower bound) For every $r \geq  0$, $\delta > 0$ and $\gamma > 0$, there exists $\epsilon _0 > 0$ such that
   \begin{equation*}
	   \inf_{x \in D}\left(\mu^x _\epsilon \left(B_E(\varphi, \delta ) \right) -  \exp \left(-\frac{I^x(\varphi) + \gamma }{\epsilon }\right) \right) \geq 0,
   \end{equation*}
   for any $\epsilon \leq \epsilon _0$ and $\varphi  \in \Phi^x(r)$, where $\Phi^x(r) := \{h \in E: I^x(h) \leq r\}$.
   \item[(2)] (Upper bound) For every $r_0 \geq 0$, $\delta > 0$ and $\gamma > 0$, there exists $\epsilon _0 > 0$ such that 
   \begin{equation*}
	   \sup_{x \in D}\mu^x _\epsilon \left(B^c_E(\Phi^x(r), \delta ) \right) \leq \exp \left(-\frac{r-\gamma }{\epsilon }\right),
   \end{equation*}
   for any $\epsilon \leq  \epsilon _0$ and $r \leq r_0$, where $B^c_E(\Phi^x(r), \delta ) := \left\{h \in E: dist_E(h,\Phi^x(r)) \geq \delta \right\}$.
\end{definition}
In \cite{cerrai2022large}, the Cerrai-Paskal proved the FWULDP for stochastic Navier-Stokes equations. They achieved that by first using the weak convergence approach developed in Budhiraja, Dupuis and Maroulas\cite{budhiraja2008large} to prove the large deviations principle for the stochastic convolution $Z_\epsilon $. Then, they used a uniform contraction principle to prove the Freidlin-Wentzell uniformly large deviations principle in a bounded subset. In order to apply the uniform contraction principle, an exponential estimate to the law of stochastic convolution is needed. But this is difficult to obtain in our case since the noise operator $Q_{\epsilon }$ could be non-trace class and Itô's formula is not applicable. Here, we use a different approach to prove FWULDP by proving the Equicontinuous Uniform Laplace Principle (EULP) (see Theorem 2.10 of \cite{salins2019equivalences}).

Let $\mathcal{P} _2$ be the collection of $\mathcal{F} _t$ adapted $H$-valued processes $f(t)$ with $ \mathbb{P} \left(\| f\Vert _{L^2([0,T];H)} < +\infty\right) =1$. Let $\mathcal{S}^N = \left\{f \in L^2([0,T];H):\| f\Vert _{L^2([0,T];H)}^2 \leq N\right\}$. Let $\mathcal{P} _2^N$ be the collection of $\mathcal{F} _t$-adapted $H$-valued processes such that $\mathbb{P} \left(f \in \mathcal{S} ^N\right)=1$. For $f \in \mathcal{P} _2$, we denote by $u_{\epsilon ,f}^x$ the solution to the equation 
\begin{equation*}
	\begin{cases}
		du_{\epsilon ,f}^x(t) + Au_{\epsilon ,f}^x(t)dt = \frac{1}{2}D\left[u_{\epsilon ,f}^x(t)^2\right] dt+\sqrt{\epsilon Q_{\epsilon}}dW(t) + \sqrt{ Q_{\epsilon}}f(t)dt,\\
		u_{\epsilon ,f}^x(0) = x,
	\end{cases}
\end{equation*}
and we denote $u_f$ to be the solution of skeleton equation
\begin{equation}
	\label{skt eq1}
	\begin{cases}
	du_f^x(t) + Au_f^x(t)dt = \frac{1}{2}D\left(u_f^x(t)^2\right) dt +\sqrt{Q_0}f(t)dt, \\
	u_f^x(0) = x,
	\end{cases}	
\end{equation}
where $\sqrt{Q_0}$ is defined as $$\sqrt{Q_0}e_k = \sigma _{0,k}e_k, \ k \in \mathbb{N}. $$
Define $\mathcal{F} ^\epsilon_x  $ to be the measurable mapping that maps $W(\cdot)$ to the mild solution of the equation
\begin{equation*}
	du(t) + Au(t)dt = \frac{1}{2}D\left(u^2(t)\right)dt + \sqrt{Q_\epsilon }dW(t),\ u(0) = x.
\end{equation*}
Then $u_{\epsilon ,f}^x$ can be written as
\begin{equation*}
	\mathcal{F}^\epsilon_x \left(\sqrt{\epsilon }\left(W + \frac{1}{\sqrt{\epsilon }}\int_{0}^{\cdot} f(s) \,ds \right)\right) = \left(I + \mathcal{G} _x\right)\left(Z_{\epsilon ,f}\right),
\end{equation*}
where $Z_{\epsilon ,f}$ is the solution the the equation 
\begin{equation*}
	\begin{cases}
		dZ_{\epsilon ,f}(t) + AZ_{\epsilon ,f}(t)dt = \sqrt{\epsilon Q_{\epsilon}}dW(t) + \sqrt{ Q_{\epsilon}}f(t)dt,\\
		Z_{\epsilon ,f}(0) = 0.
	\end{cases}
 \end{equation*}
Moreover, $u_f^x$ can be written as 
$$\mathcal{F}^0_x \left(\int_{0}^{\cdot} f(s) \,ds \right) = \left(I + \mathcal{G} _x\right)\left(Z_f\right), $$
where $Z_f$ is the solution the the equation 
\begin{equation*}
	\begin{cases}
		dZ_f(t) + AZ_f(t)dt = \sqrt{ Q_{0}}f(t)dt,\\
		Z_f(0) = 0.
	\end{cases}
 \end{equation*}
For any operator $Q$ satisfies $\sqrt{Q}e_k = \sigma _k e_k ,\ k \in \mathbb{N} $, we define a mapping from $L^2([0,T];H)$ to $C([0,T];L^4)$ by $${\cal R}(f)(t): = \int_{0}^{t} S(t-s)\sqrt{Q}f(s) \,ds.$$ We have the following proposition.
\begin{proposition}
	\label{compact1}
	Assume $\sup_{k \in \mathbb{N} }\frac{\sigma^2 _k}{k^\theta }< \infty$ for some $\theta < \frac{3}{2}$, then the map ${\cal R}$ defined above is a compact operator from $L^2([0,T];H)$ into $C([0,T];L^4)$.
\end{proposition}
\begin{proof}
	To prove the compactness of ${\cal R}$, we need to show that for all $M \in [0,\infty)$ and $f$ satisfying
	\begin{equation}
		\label{M bound}
		\frac{1}{2}\int_{0}^{T} \| f(t)\Vert _H^2 \,dt \leq M
	\end{equation}
	the set
	$$K_Q(M) := \{{\cal R}(f)\}$$ is a compact set in $C([0,T];L^4)$.

	Using a factorization formula, we rewrite $Z_f(t) = \int_{0}^{t} S(t-s)\sqrt{Q}f(s) \,ds$ as $Z_f = R_\alpha\left(\mathcal{Y} _\alpha (f)\right),\ \alpha  \in (0,1)$ where
	$${\cal R}_\alpha (\mathcal{Y} )(t):= C_\alpha \int_{0}^{t} (t-s)^{\alpha -1}S(t-s)\mathcal{Y} (s) \,ds,$$ $C_\alpha = \frac{\sin\alpha \pi}{\pi} $ and 
	$$\mathcal{Y} _\alpha (f)(s):= \int_{0}^{s} (s-r)^{-\alpha }S(s-r)\sqrt{Q}f(r) \,dr.$$ 
	It is well-known that ${\cal R}_\alpha$ is a linear bounded operator from $L^p([0,T];H)$ into $C^{\alpha -\gamma -\frac{1}{p}}([0,T];D(A^{\gamma }))$, where $\alpha-\gamma -\frac{1}{p} >0$ (see Proposition A.1.1, \cite{da1996ergodicity}). Choosing $\alpha = 1/2 - \theta /4 > 1/8$, we next show that $\mathcal{Y} _\alpha (f) \in L^\infty([0,T];H)$.

	By (\ref{M bound}) we have $$\left[\int_{0}^{T} \| f(t)\Vert _H^2 \,dt\right]^{1/2} = \left[\sum_{k \in \mathbb{N} }\int_{0}^{T} f_k^2(t) \,dt \right]^{1/2} \leq \sqrt{2M},$$ where $f_k(t)$ denote the $k$-th coordinate of $f(t)$, that is $f_k(t)=\langle f(t), e_k\rangle $. Then for every $s \in [0,T]$,
	\begin{align*}
		&\ \ \ \ \| \mathcal{Y} _\alpha (f)(s)\Vert _H \\&= \left \langle \sum_{k \in \mathbb{N} }\sigma _{k}e_k\int_{0}^{s} (s-r)^{-\alpha }e^{-\pi^2k^2(s-r)}f_k(r) \,dr, \sum_{k \in \mathbb{N} }\sigma _{k}e_k\int_{0}^{s} (s-r)^{-\alpha }e^{-\pi^2k^2(s-r)}f_k(r) \,dr\right\rangle^{1/2} \\&= \left[\sum_{k \in \mathbb{N} }\sigma _{k}^2\left(\int_{0}^{s} r^{-\alpha }e^{-\pi^2k^2r}f_k(s-r) \,dr \right)^2\right]^{1/2}\\ &\leq \left[\sum_{k \in \mathbb{N} }\sigma _{k}^2\int_{0}^{s} r^{-2\alpha }e^{-2\pi^2k^2r} \,dr \int_{0}^{s} f_k^2(r) \,dr  \right]^{1/2} \\ &\leq \left[C\sum_{k \in \mathbb{N} }\frac{\sigma _{k}^2}{k^{2(1-2\alpha) }}\int_{0}^{s} f_k^2(r) \,dr\right]^{1/2} \\ &= \left[C\sum_{k \in \mathbb{N}}\frac{\sigma^2 _{k}}{k^\theta }\int_{0}^{s} f_k^2(r) \,dr\right]^{1/2} \leq \sqrt{2C_1M},
	\end{align*}
 where $ \ C_1 = C\sup_{k \in \mathbb{N} }\frac{\sigma^2 _{k}}{k^\theta }< \infty $.
	Hence $\mathcal{Y} _\alpha (f) \in L^\infty([0,T];H)$ for $\alpha = 1/2 - \theta /4 > 1/8$. We choose $\gamma>1/8$ and $p$ large enough such that $\alpha -\gamma -1/p=\delta >0$. Then \begin{equation}\label{inequality: Z}
	    \| Z_f\Vert _{C^\delta ([0,T];H^{2\gamma })} \leq C_\delta \sqrt{2C_1M},	\end{equation} 
     thus $\mathcal{R}$ is a bounded linear operator from $L^2([0,T];H)$ into $C^\delta ([0,T];H^{2\gamma })$. According to Rellich-Kondrachov theorem, $H^s$ is compact embedded into $L^4$ for $s>1/4$. Therefore, $K_Q(M)$ is a bounded closed set in $C^\delta ([0,T];H^{2\gamma })$ and thus a compact set in $C([0,T];L^4)$.
\end{proof}

To state the following results, we give some assumptions about the noise.
\begin{condition}
	\label{assumption 1}
	The eigenvalues of operator $\sqrt{Q_\epsilon}$ satisfy $$\text{for each} \ \epsilon > 0, \ \sum_{k \in \mathbb{N}}\frac{\sigma^2 _{\epsilon ,k}}{k^2} < \infty \ \text{and}\  \lim_{\epsilon \rightarrow 0}\epsilon \sum_{k \in \mathbb{N}}\frac{\sigma^2 _{\epsilon ,k}}{k^2} = 0.$$ Moreover, for any $N>0$, $$\lim_{\epsilon \rightarrow 0}\sqrt{Q_0}^{-1}\sqrt{Q_\epsilon }f= f \ \text{weakly in}\  L^2([0,T];H), \text{uniformly with respect to } \  \| f\Vert _{L^2([0,T];H)}^2 \leq N,$$ where $\sqrt{Q_0}$ satisfies $\sup_{k \in \mathbb{N} }\frac{\sigma^2 _{0,k}}{k^\theta }< \infty$ for some $\theta < \frac{3}{2}$.
\end{condition}
\begin{example}
	Suppose that $\sqrt{Q_0} \in L(H)$ and
	$$\lim_{\epsilon \rightarrow 0 }\sqrt{Q_\epsilon} = \sqrt{Q_0}$$ in $L(H)$, then Assumption \ref{assumption 1} is satisfied.
\end{example}

\begin{example}
	Set $\sqrt{Q_\epsilon} = A^{1/4-\delta (\epsilon )}$ where $\delta (\epsilon )>0$. We have $$\sigma _{\epsilon ,k} = (k\pi)^{1/2-2\delta (\epsilon )}, \ k \in \mathbb{N}. $$ Suppose $\lim_{\epsilon \rightarrow 0}\delta (\epsilon ) = 0$ and $\lim_{\epsilon \rightarrow 0}\frac{\epsilon }{\delta (\epsilon )} = 0$, then 
	$$\lim_{\epsilon \rightarrow 0}A^{-1/4}\sqrt{Q_\epsilon }= A^{-1/4}\sqrt{Q_0} \ \text{in}\  L(H).$$ Moreover, 
	\begin{align*}
		\lim_{\epsilon \rightarrow 0}\epsilon \sum_{k \in \mathbb{N}}\frac{\sigma^2 _{\epsilon ,k}}{k^2} &= \lim_{\epsilon \rightarrow 0}\epsilon \sum_{k \in \mathbb{N}}\frac{\pi^2}{(k\pi)^{1+ 4\delta  (\epsilon )}} \\&\leq \lim_{\epsilon \rightarrow 0}\epsilon\int_{1}^{+ \infty} \frac{\pi^2}{(\pi x)^{1+ 4\delta  (\epsilon )}} \,dx \\ & \leq \lim_{\epsilon \rightarrow 0}\frac{\pi\epsilon }{4\delta (\epsilon )} = 0.
	\end{align*}
	Thus, Assumption \ref{assumption 1} is satisfied.
\end{example}
\begin{lemma}
	\label{LDPZ}
	Under Assumption \ref{assumption 1}, for any $N >0$ and $\delta >0$
	$$\lim_{\epsilon  \rightarrow 0}\sup_{f \in \mathcal{P} _2^N}\mathbb{P} \left(\|Z_{\epsilon ,f} - Z_f\Vert _{C([0,T];L^4)} > \delta \right) =0.$$
\end{lemma}
\begin{proof}
	
	
	For every $t \geq 0$ and $f \in \mathcal{P}_2^N$ , we have 
	\begin{align*}
		Z_{\epsilon,f}(t) - Z_{f}(t) &= \sqrt{\epsilon }\int_{0}^{t} S(t-s)\sqrt{Q_\epsilon} \,dW(s) + \int_{0}^{t} S(t-s)\left[\sqrt{Q_\epsilon}f(s) - \sqrt{Q_0}f(s)\right] \,ds \\
		&=: I_1^\epsilon(t) + I_2^\epsilon(t).  
	\end{align*}
	Following from the first part of proof of the Proposition \ref{wellposedness of SBE}, we have 
	\begin{align*}
		\lim_{\epsilon\rightarrow 0}\mathbb{E} \| I_1^\epsilon\Vert ^4_{C([0,T];L^4)} &= \lim_{\epsilon\rightarrow 0}\mathbb{E} \sup_{0 \leq t \leq T}\| Z_\epsilon (t)\Vert _{L^4}^4\\ &\leq c_4\lim_{\epsilon\rightarrow 0}\left(\epsilon \sum_{k \in \mathbb{N}}\frac{\sigma^2 _{\epsilon ,k}}{(k\pi)^2}\right)^{2}\\ &= 0.
	\end{align*}
	
	For an estimate of $I_2^\epsilon(t)$, by Proposition \ref{compact1}, the operator ${\cal R}_{Q_0}(f)(t) = \int_{0}^{t} S(t-s)\sqrt{Q_0} f(s)\,ds$ is a compact operator from $L^2([0,T];H)$ into $C([0,T];L^4)$. it then follows from

	\begin{align*}
		I_2^\epsilon(t) &= \int_{0}^{t} S(t-s)\left[\sqrt{Q_\epsilon}f(s) - \sqrt{Q_0}f(s)\right] \,ds\\
		&=\int_{0}^{t} S(t-s)\sqrt{Q_0}\left[\sqrt{Q_0}^{-1}\sqrt{Q_\epsilon}f(s) - f(s)\right] \,ds \\&= \mathcal{R}_{Q_0}\left(\sqrt{Q_0}^{-1}\sqrt{Q_\epsilon}f - f\right)(t).
	\end{align*}
	and the fact that compact operators map weakly convergent sequences to strongly convergent sequences that
	$$\lim_{\epsilon\rightarrow 0} \left\| I_2^\epsilon\right\Vert _{C([0,T];L^4)} = 0, \ \mathbb{P} - a.s.$$ The result follows.
\end{proof}

Next, we study the rate function of stochastic Burgers equation. Assume $T >0$, given a function $u \in L^\infty([0,T];H)$, we say that
	$$Q_0^{-1/2}(\partial_tu + Au - \frac{1}{2}D(u^2)) \in L^2([0,T];H),$$ if there exists $f \in L^2([0,T];H)$ such that $u$ is the mild solution of Burgers equation
	\begin{equation}
		\begin{cases}
			\partial_tu + Au(t) - \frac{1}{2}D(u^2(t))) = \sqrt{Q_0}f,\\
		u(0) = x \in H.
		\end{cases}
	\end{equation}
Clearly, the corresponding function $f$ is unique and we will denote it by $Q_0^{-1/2}\mathcal{H} (u)$, where
$$\mathcal{H} (u):= \partial _tu + Au - \frac{1}{2}D(u^2).$$
For any $x \in H$ and $u \in L^{\infty}([0,T];H)$, we define the rate function
$$I^x(u) := \begin{cases}
	\frac{1}{2}\int_{0}^{T} \| Q_0^{-1/2}\mathcal{H} (u)(t)\Vert_H^2  \,dt & \text{if} \ Q_0^{-1/2}\mathcal{H}(u) \in L^2([0,T];H) \ \text{and} \ u(0) = x,\\
	+ \infty &\text{otherwise}.
\end{cases}$$
\begin{theorem}
	\label{LDPu}
	Under Assumption \ref{assumption 1}, if $u_\epsilon^x$ is the mild solution of Eq. (\ref{sbe2}), then the family $\left\{\mathcal{L} (u_\epsilon^x)\right\}_{\epsilon > 0}$ satisfies the Freidlin-Wentzell uniform large deviations principle in $L^\infty([0,T];H)$ with a good rate function $I^x$, uniformly with respect to $x \in B_H(R)$. 
\end{theorem}
\begin{proof}
	Since $u_{\epsilon ,f}^x = (I + \mathcal{G} _x)(Z_{\epsilon ,f})$ and $u_f^x = (I + \mathcal{G} _x)(Z_f) $, it follows that for any $N>0$ and $\delta >0$
	\begin{align*}
	&\ \ \sup_{x \in B_H(R)}\sup_{f \in \mathcal{P} _2^N}\mathbb{P} \left(\|u_{\epsilon ,f}^x - u_f^x\Vert _{L^\infty([0,T];H)} > \delta \right) \\ &\leq \sup_{f \in \mathcal{P} _2^N}\mathbb{P} \left(\|Z_{\epsilon ,f} - Z_f\Vert _{C([0,T];L^4)} > \delta/2 \right) \\ &\ \ + \sup_{x \in B_H(R)}\sup_{f \in \mathcal{P} _2^N}\mathbb{P} \left(\|\mathcal{G} _x(Z_{\epsilon ,f}) - \mathcal{G} _x(Z_f)\Vert _{L^\infty([0,T];H)} > \delta/2 \right).
    \end{align*}
    According to Lemma \ref{LDPZ}, the first term converges to zero. For the second term, since trivially $$\sqrt{2}\exp\left(C_2\int_{0}^{T} \| \psi(s) \Vert _{L^4}^{8/3} \,ds \right)\| \psi \Vert _{L^4([0,T]; L^4)} \leq \exp\left(C_T + \| \psi\Vert _{L^4([0,T];L^4)}^4 \right),$$ by (\ref{continuity estimate}), it follows that
	\begin{align*}
		\mathbb{P} \left(\|\mathcal{G} _x(Z_{\epsilon ,f}) - \mathcal{G} _x(Z_f)\Vert _{L^\infty([0,T];H)} > \delta/2 \right) &\leq \mathbb{P} \left(\| \psi^x \Vert _{L^4([0,T];L^4)}>\left(\ln M - C_T\right)^{1/4}\right) \\ &\ \ \ \ + \mathbb{P} \left(\| Z_{\epsilon, f}-Z_f\Vert _{C([0,T];L^4)} > \frac{\delta }{2M}\right),
	\end{align*}
	where $\psi^x = Z_{\epsilon ,f} + Z_f + \mathcal{G} _x(Z_{\epsilon ,f}) + \mathcal{G} _x(Z_f)$ and $M$ is some positive constant. By interpolation inequality and the Poincar\'e inequality, there exits a continuous function $C(\cdot,\cdot)$ on $[0,+ \infty) \times [0,+ \infty)$ such that
	\begin{align*}
		\| \mathcal{G} _x(Z_f)\Vert _{L^4([0,T];L^4)} &\leq C_1\left(\|\mathcal{G} _x(Z_f)\Vert _{L^\infty([0,T];H)}\cdot\| \mathcal{G} _x(Z_f)\Vert _{L^2([0,T];H^1)} \right)^{1/2} \\ &\leq C\left(\| Z_f\Vert _{C([0,T];L^4)}, \| x\Vert _H\right),
	\end{align*}
	where the second inequality is implied by (\ref{Y estimate 1}) and (\ref{Y estimate 2}). According to Proposition \ref{compact1}, there exits a positive constant $C^\prime$ such that $$\| Z_f\Vert_{L^4([0,T];L^4)} \leq T^{1/4}\| Z_f\Vert_{C([0,T];L^4)} \leq C^\prime \sqrt{N}, \ \mathbb{P} - a.s.$$ for every $f \in \mathcal{P} _2^N$. Thus, for every $f \in \mathcal{P} _2^N$ and $x \in B_H(R)$, $$\| \mathcal{G} _x(Z_f)\Vert _{L^4([0,T];L^4)} \leq C_{N,R},$$ where $C_{N,R}$ is some positive constant depending on $N$ and $R$. Similarly, we have 
 $$\|Z_{\epsilon,f}\Vert_{L^4([0,T;L^4]) } \leq  \|Z_{\epsilon,f}\Vert_{C([0,T;L^4]) } \leq C^\prime\sqrt{N} + \|Z_{\epsilon,f}-Z_f\Vert_{C([0,T];L^4) } $$ and $$\|\mathcal{G}_x(Z_{\epsilon,f})\Vert_{L^4([0,T];L^4)} \leq C\left(C^\prime \sqrt{N}+ \|Z_{\epsilon,f}-Z_f\Vert_{C([0,T];L^4) },R \right)$$ $\mathbb{P} - a.s.$ for every $f \in \mathcal{P} _2^N$. 
 Then, by Lemma \ref{LDPZ}, we can choose $M$ sufficiently large such that
	$$\lim_{\epsilon \rightarrow 0}\sup_{x \in B_H(R)}\sup_{f \in \mathcal{P} _2^N}\mathbb{P} \left(\| \psi^x \Vert _{L^4([0,T];L^4)}>\left(\ln M - C_T\right)^{1/4}\right) = 0$$ and 
	$$\lim_{\epsilon \rightarrow 0}\sup_{f \in \mathcal{P} _2^N} \mathbb{P} \left(\| Z_{\epsilon, f}-Z_f\Vert _{C([0,T];L^4)} > \frac{\delta }{2M}\right) = 0.$$ Therefore, by above discussions, it follows that
	$$ \lim_{\epsilon \rightarrow 0}\sup_{x \in B_H(R)}\sup_{f \in \mathcal{P} _2^N}\mathbb{P} \left(\|u_{\epsilon ,f}^x - u_f^x\Vert _{L^\infty([0,T];H)} > \delta \right) = 0.$$ According to Theorem 2.10 and Theorem 2.13 of \cite{salins2019equivalences}, this implies that the family $\left\{\mathcal{L} (u_\epsilon^x)\right\}_{\epsilon > 0}$ satisfies a Freidlin-Wentzell uniform large deviations principle in $L^\infty([0,T];H)$ with rate functions $I^x$, uniformly with respect to $x \in B_H(R)$. 

	Moreover, $$K^x(N) = \left\{u \in L^\infty([0,T];H): I^x(u) \leq N \right\} = \left\{(I + \mathcal{G}_x)(Z_f): f \in \mathcal{S}^{2N}  \right\}.$$ By Proposition \ref{compact1}, the set $\left\{Z_f: f \in \mathcal{S} ^{2N}\right\}$ is compact for any $N>0$. Lemma \ref{continuous lemma} implies that the mapping $$I  + \mathcal{G} _x: C([0,T]; L^4) \rightarrow L^\infty ([0,T];H),$$ is continuous for each fixed $x \in H$. Thus, $K^x(N)$ is compact which implies that the rate function $I^x$ is good. 

\end{proof}

We can also prove the FWULDP through another approach. We introduce the following modification of the Uniformly Contraction Principle \cite{cerrai2022large} that allows us to get a uniform large deviation principle.
\begin{theorem} \label{uniform contraction principle}
    Let $D$ be some nonempty set and $\{ \Lambda^x \}_{x \in D}$ be a family of mappings from a Banach space $F$ to a Banach space $E$. Assume \\ (i) the family of measures $\{\nu_\epsilon\}_{\epsilon>0}$ satisfies a large deviations principle on $F$ with good rate function $J: F \rightarrow [0,+\infty]$; \\
    (ii) for any fixed $s>0$, 
    $$\sup_{h \in \Phi(s)}\|h\Vert_F < \infty,$$ where $\Phi(s):=\left\{h \in F: J(h) \leq s \right\}$; \\
    (iii) there exists another Banach space $G$ such that $F$ is continuously embedded into $G$ 
    and for every $R>0$, there exists some $L_R>0$ such that 
    $$\sup_{x\in D}\sup_{\phi_1,\phi_2 \in B_G(R) \cap F} \frac{\|\Lambda^x(\phi_1)-\Lambda^x(\phi_2)\Vert_E}{\|\phi_1 -\phi_2\Vert_F} = L_R.$$ Then the family of push-forward measures $\{\mu_\epsilon^x\}_{\epsilon>0}$ defined by $\mu_\epsilon^x:=\nu_\epsilon \circ (\Lambda^x)^{-1}$ satisfies a Freidlin-Wentzell uniform large deviations principle in $E$ with rate functions $I^x$ uniformly with respect to $x \in D$, where $I^x$ is given by 
    $$I^x(\phi):= \inf\left\{J(\Psi): \Psi\in F, \phi = \Lambda^x(\Psi) \right\}.$$
\end{theorem}
\begin{proof}
     According to Theorem 3.3 of \cite{cerrai2022large}, we only need to prove that for every $s>0$ there exists $R_s>0$ and $\epsilon_s >0$ such that 
    $$\nu_\epsilon\left(B_G(R_s) \cap F \right) \geq 1 - \exp{\left(- \frac{s}{\epsilon}\right)},$$ for any $\epsilon \leq \epsilon_s$. Since $F$ is continuously embedded into $G$, there exists a positive constant $C$ such that $$\nu_\epsilon\left(B_G(R_s) \cap F \right) \geq \nu_\epsilon\left(B_F(R_s/C) \right).$$  By the LDP of $\nu_\epsilon$ in $F$, it follows that for every $s \geq 0 $, $\delta >0$ and $\gamma>0$, there exits $\epsilon_s >0$ such that 
    $$\nu_\epsilon\left(B_F^c(\Phi(s),\delta) \right) \leq \exp{\left(- \frac{s- \gamma}{\epsilon} \right)},$$ for any $\epsilon \leq \epsilon_s$. For every $s>0$, we set $\gamma$ and $\delta$ small and choose a positive constant $R_s$ such that $R_s \geq C\left(\sup_{h \in \Phi(s+\gamma)}\|h\Vert_F + \delta \right)$. Then, we have \begin{align*}
        \nu_\epsilon\left(B_G(R_s) \cap F \right) &\geq \nu_\epsilon\left(B_F(R_s/C) \right) \\ & \geq \nu_\epsilon\left(B_F(\Phi(s+ \gamma),\delta) \right) \\ &\geq 1- \exp{\left(- \frac{s}{\epsilon} \right)}.        \end{align*}Then the result of the theorem follows. 
\end{proof}
Thanks to the discussions in the proof of Proposition \ref{compact1} and (\ref{inequality: Z}), $$\sup_{h \in \Phi(s)}\|h\Vert_{C([0,T];L^4)} \leq C\sqrt{s}.$$ Moreover, by Lemma \ref{LDPZ} and Theorem 2.10, Theorem 2.13 in \cite{salins2019equivalences}, the family $\mathcal{L}(Z_\epsilon)$ satisfies LDP in $C([0,T];L^4)$. Then, the FWULDP for $\mathcal{L}(u_\epsilon^x) $ follows by Lemma \ref{continuous lemma} and Theorem \ref{uniform contraction principle}.

In this paper, we also need another definition of uniform large deviations principle which is called Dembo-Zeitouni uniform large deviations principle. According to \cite{salins2019equivalences}, these two definitions of ULDP are not equivalent. The following definition can be found in \cite{dembo2009large}.
\begin{definition}
	\label{DZULDP}
	Let $E$ be a Banach space and $D$ be some non-empty set. Suppose that for each $x \in D$, $\{\mu^x _\epsilon \}_{\epsilon > 0}$ is a family of probability measures on $E$ and an action function $I^x:E\rightarrow [0,+\infty]$ is a good rate function. The family $\{\mu^x_\epsilon \}_{\epsilon > 0}$ is said to satisfy a Dembo-Zeitouni uniform large deviation principle (DZULDP) in $E$, with rate function $I^x$, uniformly with respect to $x \in D$, if the following hold.
	\item[(1)](Lower bound) For any $\gamma > 0$ and open set $G \subset E$, there exists $\epsilon _0 > 0$ such that
	\begin{equation*}
		\inf_{x \in D}\mu^x _\epsilon (G)  \geq \exp\left(-\frac{1}{\epsilon }\left[\sup_{y \in D}\inf_{u \in G}I^y(u) + \gamma \right]\right),
	\end{equation*}
	for any $\epsilon \leq \epsilon _0$.
	\item[(2)] (Upper bound) For any $\gamma > 0$ and closed set $F \subset E$, there exists $\epsilon _0 > 0$ such that 
	\begin{equation*}
		\sup_{x \in D}\mu^x _\epsilon (F) \leq \exp \left(-\frac{1 }{\epsilon }\left[\inf_{y \in D } \inf_{u \in F} I^y(u) -\gamma \right]\right),
	\end{equation*}
	for any $\epsilon \leq  \epsilon _0$.
 \end{definition}

\begin{corollary}
	\label{corollary DZ}
	Let $D \subset H$ be a compact set. If $u_\epsilon^x$ is the mild solution of Eq. (\ref{sbe2}), then under Assumption \ref{assumption 1}, the family $\left\{\mathcal{L} (u_\epsilon^x)\right\}_{\epsilon > 0}$ satisfies a Dembo-Zeitouni uniform large deviations principle in $L^\infty([0,T];H)$ with good rate function $I^x$, uniformly with respect to $x \in D$. 
 \end{corollary}
\begin{proof}
	According to Theorem 2.7 of \cite{salins2019equivalences}, to prove equivalence of the two uniform large deviations principle over a compact subset of $H$, it suffices the show that the level sets are continuous in Hausdorff metric in the sense that for any $r \geq 0$,
	$$\lim _{n \rightarrow \infty}d_H(x_n, x) = 0 \ \text{implies }\ \lim_{n \rightarrow \infty}\lambda \left(\Phi ^x(r),\Phi^{x_n}(r)\right) = 0,$$ where $\lambda$ is the Hausdorff metric and $\Phi ^x(r) := \left\{u \in L^{\infty}([0,T];H): I^x(u) \leq r\right\}$ is the level set of $I^x$. That is, we must show that for any $x_n \rightarrow x \in H$,
	$$\lim_{n \rightarrow \infty} \max\left(\sup_{u \in \Phi ^{x_n}(r)} dist_{L^\infty([0,T];H)}(u,\Phi ^x(r)), \sup_{u \in \Phi ^x(r)} dist_{L^\infty([0,T];H)}(u,\Phi ^{x_n}(r))\right) = 0.$$
	Suppose $u \in \Phi ^x(r)$, then $f := Q_0^{-1/2}\mathcal{H} (u) \in L^2([0,T];H)$, $\frac{1}{2}\left\| f\right\Vert^2 _{L^2([0,T];H)} \leq r$ and $u$ satisfies
	$$u = (I+\mathcal{G} _x)(Z_f)$$ where $$Z_f(t) = \int_{0}^{t} S(t-s)\sqrt{Q_0}f(s) \,ds .$$ Define $v = (I+\mathcal{G} _y)(Z_f)$, then 
	\begin{align*}
		I^y (v)&= \frac{1}{2}\int_{0}^{T} \| Q_0^{-1/2}\mathcal{H} (v)(t)\Vert_H^2  \,dt \\ &=\frac{1}{2}\int_{0}^{T} \| Q_0^{-1/2}(\partial _tv(t) + Av(t) - \frac{1}{2}D(v^2(t))\Vert_H^2  \,dt\\ &= \frac{1}{2}\int_{0}^{T} \| f(t)\Vert_H^2  \,dt \leq r,
	\end{align*}
	which implies $v \in \Phi ^y(r).$
	Thanks to Equation (\ref{Y1}), 
	\begin{align*}
		\frac{d}{dt} (u(t)-v(t)) 
  &= -A(u(t)-v(t)) + D\left[(\mathcal{G} _x(Z_f)(t) + \mathcal{G} _y(Z_f) (t) + 2Z_f(t))(u(t)-v(t))\right].
	\end{align*}

	Denote $\psi _f := \mathcal{G} _x(Z_f)(t) + \mathcal{G} _y(Z_f) (t) + 2Z_f(t)$. By similar energy estimates as in the proof of Lemma \ref{continuous lemma} we have that
	\begin{align*}
		\frac{1}{2}\frac{d}{dt}\| u(t)-v(t)\Vert _H^2 + \|u(t)-v(t)\Vert _{H^1}^2 &\leq \int_{0}^{1} D\left[\psi_f (t)(u(t)-v(t))\right](u(t)-v(t)) \,d\xi \\ &= -\int_{0}^{1} \psi_f (t)(u(t)-v(t))D(u(t)-v(t)) \,d\xi \\ &\leq \| \psi_f (t)\Vert _{L^4}\| u(t)-v(t)\Vert _{L^4}\| u(t)-v(t)\Vert _{H^1} \\ &\leq C_1\| \psi_f (t)\Vert _{L^4}\| u(t)-v(t)\Vert _H^{3/4}\| u(t)-v(t)\Vert _{H^1}^{5/4}\\ &\leq  C_2\| \psi_f (t)\Vert _{L^4}^{8/3}\| u(t)-v(t)\Vert _H^2 + \frac{1}{2}\| u(t)-v(t)\Vert _{H^1}^2.
	\end{align*}
	Then \begin{equation*}
		\frac{d}{dt}\| u(t)-v(t)\Vert _H^2 + \|u(t)-v(t)\Vert _{H^1}^2 \leq 2C_2\| \psi_f(t) \Vert _{L^4}^{8/3}\| u(t)-v(t)\Vert _H^2.
	\end{equation*}
	Since $u(0)-v(0) = x-y$, it follows from the Gronwall inequality that
	\begin{equation*}
		\sup_{0 \leq t \leq T}\| u(t)-v(t)\Vert _H^2 \leq \| x-y\Vert _H^2\exp\left(2C_2\int_{0}^{T} \| \psi_f(s) \Vert _{L^4}^{8/3} \,ds \right).
	\end{equation*}
	Moreover, \begin{align*}
		\int_{0}^{T} \| \psi_f(s) \Vert _{L^4}^{8/3} \,ds &\leq C_T\| \psi_f \Vert_{L^4([0,T];L^4)}^{8/3}\\  &\leq C_T\left(\| \mathcal{G} _x(Z_f)\Vert _{ L^4([0,T];L^4)}+ \| \mathcal{G} _y(Z_f) \Vert _{ L^4([0,T];L^4)} + 2\| Z_f\Vert_{L^4([0,T];L^4)}\right)^{8/3}.
	\end{align*}
 Notice that by (\ref{Y estimate 1}), (\ref{Y estimate 2})
	$$	\| \mathcal{G} _x(Z_f)(t)\Vert _H^2 \leq \| x\Vert _H^2 \exp\left(2C_2\int_{0}^{t} \| Z_f(s) \Vert _{L^4}^{8/3} \,ds \right) + \frac{1}{2} \exp\left(2C_2\int_{0}^{t} \| Z_f(s) \Vert _{L^4}^{8/3} \,ds \right) \int_{0}^{t} \| Z_f(s) \Vert _{L^4}^4 \,ds,$$ and 
	$$\int_{0}^{t} \| \mathcal{G} _x(Z_f)(s)\Vert _{H^1}^2 \,ds \leq \int_{0}^{t} 2C_2\| Z_f(s) \Vert _{L^4}^{8/3} \| \mathcal{G} _x(Z_f)(s)\Vert _H^2 + \frac{1}{2}\| Z_f(s) \Vert _{L^4}^4 \,ds + \|x\Vert _H^2.$$
	Since 
	$$\| \mathcal{G} _x(Z_f)\Vert _{ L^4([0,T];L^4)} \leq \| \mathcal{G} _x(Z_f)\Vert _{ L^\infty([0,T];H)}^{3/4}\| \mathcal{G} _x(Z_f)\Vert _{ L^2([0,T];H^1)}^{1/4}$$ and since by Proposition \ref{compact1}, $\| Z_f\Vert _{ L^4([0,T];L^4)} \leq C \| f\Vert _{L^2([0,T];H)}$, we have 
	\begin{equation}
		\label{initail con }
		dist_{L^\infty([0,T];H)}(u,\Phi ^y(r)) \leq \| u-v\Vert _{L^\infty([0,T];H)} \leq C_r(1+ \| x\Vert _H + \|  y\Vert _H)\| x-y\Vert_H.
	\end{equation}
	for some continuous increasing function $C_r :[0,\infty) \rightarrow [0,\infty)$. Since this is true for arbitrary $u \in \Phi ^x(r)$, it follows that $$\sup_{u \in \Phi ^x(r)} dist_{L^\infty([0,T];H)}(u,\Phi ^{x_n}(r)) \leq C_r(1+ \| x\Vert _H + \|  x_n\Vert _H)\|x-x_n\Vert_H, $$ which,  together with a similar argument for $\sup_{u \in \Phi ^{x_n}(r)} dist_{L^\infty([0,T];H)}(u,\Phi ^{x}(r))$, implies the desired result.
\end{proof}

\subsection{Invariant measures in a trace-class case}\label{trace class}
In this section, we consider the cases that noises are in more specific form. We assume that the covariance operator $Q_\epsilon $ has the following form
$$Q_\epsilon =  A^{\alpha} Q_{\delta (\epsilon )}, \ 0 \leq \alpha < 1/2$$ where $Q_{\delta (\epsilon )}$ belongs to $L(H)$ and is defined as $$Q_{\delta (\epsilon )} := (I + \delta (\epsilon )A^\beta )^{-1}$$ for some $\beta >0 $, $\delta(\epsilon) >0$. Here we take $\lim_{\epsilon \rightarrow 0}\delta (\epsilon ) =0$. Notice that $$\lim_{\epsilon \rightarrow 0}\sqrt{Q_\epsilon }=  A^{\alpha/2} \ \text{in}\  L(H).$$ Then Assumption \ref{assumption 1} is fulfilled. It follows from Theorem \ref{LDPu} that the family $\left\{\mathcal{L} (u_\epsilon^x)\right\}_{\epsilon > 0}$ satisfies a Freidlin-Wentzell uniform large deviations principle in $L^\infty([0,T];H)$. If $\alpha =0 $, the space of ULDP is $C([0,T];H)$. Next, we will choose $\beta $ such that for any fixed $\epsilon>0$, $Q_\epsilon $ is a trace class operator. Since $$ A^{\alpha/2}\sqrt{Q_{\delta(\epsilon)}} e_k = (k\pi)^\alpha \sigma _{\delta(\epsilon) ,k}e_k := (k\pi)^\alpha(1+\delta (\epsilon )(k \pi)^{2\beta })^{-1/2}e_k,$$ we have 
\begin{align*}
	Tr(Q_\epsilon ) &= \sum_{k \in \mathbb{N}}(k\pi)^{2\alpha}\sigma^2 _{\delta(\epsilon) ,k}\\ &= \sum_{k \in \mathbb{N}}(k\pi)^{2\alpha}(1+\delta (\epsilon )(k \pi)^{2\beta })^{-1}\\ & \leq \int_{1}^{ \infty} \frac{(\pi x)^{2\alpha}}{1+ (\delta ^{1/2\beta}\pi x)^{2 \beta}} \,dx \\&= \delta ^{-\frac{1+2\alpha}{2\beta}}\int_{\delta ^{\frac{1}{2\beta}}}^{\infty} \frac{(\pi y)^{2\alpha}}{1+ (\pi y)^{2\beta}} \,dy \\ &\leq \delta ^{-\frac{1+2\alpha}{2\beta}}\left(\int_{\delta ^{\frac{1}{2\beta}}}^{1} (\pi y)^{2\alpha} \,dy + \int_{1}^{\infty} \frac{1}{(\pi y)^{2\beta-2\alpha}} \,dy \right)\\ &\leq C_\beta \delta (\epsilon )^{-\frac{1+2\alpha}{2\beta}} < \infty,
\end{align*}
given $\beta >  1/2 + \alpha$. Therefore, $Q_\epsilon $ is trace class if we assume $\beta > 1/2 + \alpha$. Since $\left\langle D(u^2(s)),u(s)\right\rangle_H =0$, by applying It\^{o}'s formula we get
\begin{equation}
	\label{ito 2}
	\mathbb{E} \| u_\epsilon ^x(t)\Vert _H^2 + 2\int_{0}^{t} \mathbb{E} \| u_\epsilon^x(s)\Vert _{H^1}^2 \,ds = \| x\Vert _H^2 + t\epsilon \text{Tr}Q_\epsilon \leq  \| x\Vert _H^2 + C_\beta t\epsilon \delta (\epsilon )^{-(1+2\alpha)/2\beta }.
\end{equation}

To ensure the existence and uniqueness of the invariant measure for Eq. (\ref{sbe2}), the condition  $\beta > 1/2 + \alpha $ is not enough. Notice that for each $\epsilon  > 0$, we have 
$$D(A^{(\beta-\alpha) /2}) \subset \text{Im}(Q_\epsilon ^{1/2}),$$ where $\text{Im}(Q_\epsilon ^{1/2})$ is the range of the operator $Q_\epsilon ^{1/2}$. According to  Goldys and Maslowski's work on the exponential ergodicity of stochastic Burgers equations \cite{goldys2005exponential} (see also Gourey \cite{gourcy2007large}), if $1/2 < \beta - \alpha < 1$, there exists a unique invariant measure on $H$ associated with the mild solution of Eq. (\ref{sbe2}). 

Now, let $\left\{\nu _\epsilon \right\}_{\epsilon >0}$ be the family of invariant measure. Each $\nu _\epsilon $ is ergodic in the sense that
$$\lim_{T \rightarrow \infty} \frac{1}{T}\int_{0}^{T} f(u_\epsilon^x(t)) \,dt = \int_H f(x) \,d\nu _\epsilon (x),$$ for all $x \in H$ and Borel-measurable function $f: H \rightarrow \mathbb{R} $. 

Implied by (\ref{ito 2}) and the invariance of $\nu _\epsilon $, for every $T >0$
\begin{align*}
	\int_H \| x\Vert _{H^1}^2 \,d\nu _\epsilon (x) &=  \int_H \mathbb{E} \| u_\epsilon^x(t)\Vert^2 _{H^1} \,d\nu _\epsilon (x)\\ &= \frac{1}{T}\int_{0}^{T}  \int_H \mathbb{E} \| u_\epsilon^x(t)\Vert^2 _{H^1} \,d\nu _\epsilon (x) \,dt \\ &= \frac{1}{T}\int_H\int_{0}^{T}   \mathbb{E} \| u_\epsilon^x(t)\Vert^2 _{H^1}\,dt \,d\nu _\epsilon (x) \\ &\leq \frac{1}{2T}\int_{H} \| x\Vert _H^2 \,d\nu _\epsilon (x ) + \frac{1}{2}C_\beta \epsilon \delta (\epsilon )^{-(1+2\alpha)/2\beta } \\ & \leq \frac{1}{2T}\int_{H} \| x\Vert _{H^1}^2 \,d\nu _\epsilon (x ) + \frac{1}{2}C_\beta \epsilon \delta (\epsilon )^{-(1+2\alpha)/2\beta }.
\end{align*}
If $\sup_{\epsilon >0}\epsilon \delta (\epsilon )^{-(1+2\alpha)/2\beta }< \infty$ and we choose $T>1$ then we get 
$$ \sup_{\epsilon >0}\int_H \| x\Vert _{H^1}^2 \,d\nu _\epsilon (x) < \infty,$$ which implies the tightness of $\left\{\nu _\epsilon \right\}_{\epsilon >0}$ in $H$. If we further assume 
$$\lim_{\epsilon  \rightarrow 0}\epsilon \delta (\epsilon )^{-(1+2\alpha)/2\beta } = 0,$$we get that
$$\nu _\epsilon  \rightharpoonup \delta _0$$ as $\epsilon  \rightarrow 0$. In the remaining part of this paper, we will prove the LDP for $\left\{\nu _\epsilon \right\}_{\epsilon >0}$ with the setting above.

\section{The rate function of LDP for invariant measure }\label{sec: Rate Fun}
As we have already discussed in Section \ref{trace class}, assume that 
\begin{equation}\label{Q epsilon }
	Q_\epsilon  = A^{\alpha}Q_{\delta (\epsilon )}:= A^{\alpha} (I + \delta (\epsilon )A^\beta )^{-1}.
\end{equation}
where $ 1/2 < \beta -\alpha < 1$, then for any fixed $\epsilon>0$, there exists a unique invariant measure on $H$ associated to the mild solution of Eq. (\ref{sbe2}). For any $T>0$, $0 \leq \alpha <1/2$ and $u \in C([0,T];H^{-\alpha})$ we will denote the rate function $I_T:C([0,T];H^{-\alpha}) \rightarrow [0,+\infty]$,
$$I_T(u):= \frac{1}{2}\int_{0}^{T} \| A^{-\frac{\alpha}{2}}\mathcal{H} (u)(t)\Vert ^2_H \,dt .$$
The quasi-potential $U^{\alpha}: H^{-\alpha} \rightarrow \mathbb{R} ^+$ is defined as 
\begin{equation}
	\label{quasi potential 1}
	U^\alpha(\phi ):= \inf\left\{I_T(u): T >0, u \in C([0,T];H^{-\alpha}), u(0) = 0, u(T)= \phi \right\}.
\end{equation}
Notice that the quasi-potential can also be written in a different form 
\begin{equation}
	\label{quasi potential 2}
	U^\alpha(\phi )= \inf\left\{I_{-T}(u): T >0, u \in C([-T,0];H^{-\alpha}), u(-T) = 0, u(0)= \phi \right\},
\end{equation}
where 
$$I_{-T}(u):= \frac{1}{2}\int_{-T}^{0} \| A^{-\frac{\alpha}{2}}\mathcal{H} (u)(t)\Vert ^2_H \,dt $$ for any $T>0$ and $u \in C([-T,0];H^{-\alpha})$.
In this section, we will study the quasi-potential $U(x)$. We will prove that the quasi-potential $U$ can be a good rate function of large deviations principle. To achieve that, we need some properties of the skeleton equation.
\subsection{The skeleton equation}
The skeleton equation associated to Eq. (\ref{sbe2}) is 
\begin{equation}
	\label{skt eq 2}
\begin{cases}
	\partial _t u = \Delta  u(t) + \frac{1}{2}D(u^2(t)) + A^{\alpha/2}f(t),\\
	u(0) = x .
\end{cases}
\end{equation}
As we already discussed above, given $ f \in L^2([0,T];H)$, the mild solution of Eq. (\ref{skt eq 2}) is well-defined. We denote $W^{1,2}([0,T]; H^{2-\alpha}, H^{-\alpha})$ the space of all $u \in L^2([0,T];H^{2-\alpha})$ which are weakly differentiable as functions from $[0,T]$ to $H^{-\alpha}$ and their time derivative belongs to $L^2([0,T];H^{-\alpha})$. That is $$W^{1,2}([0,T]; H^{2-\alpha}, H^{-\alpha}) := L^2([0,T];H^{2-\alpha}) \cap W^{1,2}([0,T];H^{-\alpha}).$$ Then, we can prove the following propositions.
\begin{proposition}
	\label{skt estimate}
	Suppose $T > 0$, $f \in L^2([0,T];H)$, $0 \leq \alpha< 1/2$ and $x \in H^{1-\alpha}$, then the mild solution $u$ to Eq. (\ref{skt eq 2}) satisfies 
	$$u \in L^\infty([0,T];H^{1-\alpha}) \ \text {and}\ u \in W^{1,2}([0,T]; H^{2-\alpha}, H^{-\alpha}).$$
	
\end{proposition}
\begin{proof}
	We apply standard energy estimates. Multiplying both sides of Eq. (\ref{skt eq 2}) by $u(t)$ and integrating with respect to $\xi \in [0,1]$, by Lemma \ref{lemma:A2}, we get 
	\begin{align*}
		\frac{1}{2}\frac{d}{d t}\| u(t)\Vert _H^2 + \| u(t)\Vert _{H^1}^2 &= \int_{0}^{1} A^{\alpha/2}f(t)u(t)\,d\xi + \frac{1}{2}\int_{0}^{1} D(u^2(t))u(t) \,d\xi  \\ & \leq \| A^{\alpha/2}f(t)\Vert _{H^{-\alpha}}\| u(t)\Vert_{H^\alpha} \\ & \leq \frac{1}{2} \| f(t)\Vert _H^2 + \frac{1}{2}\| u(t)\Vert _{H^\alpha}^2\\ & \leq \frac{1}{2} \| f(t)\Vert _H^2 + \frac{1}{2}\| u(t)\Vert _{H^1}^2,
	\end{align*}
	where $\int_{0}^{1} D(u^2(t))u(t)\,d\xi= 0$ due to the boundary condition.
	We have \begin{equation}\label{u Gronwall 1}
		\frac{d}{d t}\| u(t)\Vert _H^2 + \| u(t)\Vert _{H^1}^2 \leq \| f(t)\Vert _H^2,
	\end{equation}
	thus \begin{equation}\label{u Gronwall 1 b} \| u(t)\Vert _H^2 + \int_{0}^{t} \| u(s)\Vert _{H^1}^2 \,ds \leq \int_{0}^{t} \| f(s)\Vert _H^2 \,ds + \| x\Vert _H^2 < \infty. \end{equation}
	Next, multiplying both sides of Eq. (\ref{skt eq 2}) by $A^{-\alpha}\partial_t u$ and repeat the above procedure we get 
	\begin{align*}
		\| \partial_t u\Vert _{H^{-\alpha}}^2 + \frac{1}{2}\frac{d}{d t}\| u(t)\Vert _{H^{1-\alpha}}^2 &= \int_{0}^{1} Du(t)u(t)A^{-\alpha}\partial_tu \,d\xi  + \int_{0}^{1} A^{\alpha/2}f(t)A^{-\alpha}\partial_tu \,d\xi \\ &\leq \| Du(t) u(t)\Vert _{H^{-\alpha}}\| \partial_tu\Vert _{H^{-\alpha}} + \| \partial_tu\Vert _{H^{-\alpha}}\| f(t)\Vert _H \\ &\leq \frac{1}{4} \| \partial_tu\Vert _{H^{-\alpha}}^2 + 2\| Du(t) u(t)\Vert _{H^{-\alpha}}^2 + 2\| f(t)\Vert _H^2. 
	\end{align*}
	By Lemma \ref{lemma:A1} and \ref{lemma:A2}, we have for some positive constants $C_1,C_2$ and  $\varepsilon =\frac{1-2\alpha}{3} >0 $ ,
	$$\| Du(t) u(t)\Vert _{H^{-\alpha}}^2 \leq C_1\| Du(t) \Vert_{B_{4,2}^{-\alpha}}^2 \|u(t)\Vert _{B_{4,2}^{\alpha+\varepsilon} }^2 \leq C_2\|u(t)\Vert_{H^{5/4 -\alpha}}^2\|u(t)\Vert_{H^{1/4+\alpha+\varepsilon}}^2 $$ 
 Moreover, by  interpolation inequality 
 $$\|u(t)\Vert_{H^{5/4 -\alpha}}^2 \leq  \|u(t)\Vert_{H^{1 -\alpha}}^{3/2}\|u(t)\Vert_{H^{2 -\alpha}}^{1/2}$$ and 
 $$\|u(t)\Vert_{H^{1/4 +\alpha + \varepsilon}}^2 \leq  \|u(t)\Vert_{H^{\alpha +  \varepsilon}}^{3/2}\|u(t)\Vert_{H^{1+\alpha +  \varepsilon}}^{1/2}.$$ Since  $1+ \alpha +  \varepsilon \leq 2 - \alpha$, by the Poincar\'e inequality
	\begin{equation}\label{ineq:nonlinear term}
	    \| Du(t) u(t)\Vert _{H^{-\alpha}}^2 \leq C_2 \|u(t)\Vert _{H^{1-\alpha}}^{3/2}\| u(t)\Vert _{H^{\alpha+ \varepsilon}}^{3/2}\| u(t)\Vert _{H^{2-\alpha}}.
     \end{equation}
	Next, we need to estimate $ \| u(t)\Vert _{H^{2-\alpha}}$. Since $$ \Delta  u(t) =\partial _t u - \frac{1}{2}D(u^2(t)) - A^{\alpha/2}f(t),$$ we have 
	\begin{align*}
		\| u(t)\Vert _{H^{2-\alpha}} = \| \Delta u(t)\Vert _{H^{-\alpha}} &= \| \partial _t u - \frac{1}{2}D(u^2(t)) - A^{\alpha/2}f(t)\Vert _{H^{-\alpha}}\\ & \leq \| \partial _t u\Vert _{H^{-\alpha}} + \| Du(t)u(t)\Vert _{H^{-\alpha}} + \| f(t)\Vert _H\\  &\leq \| \partial _t u\Vert _{H^{-\alpha}} + \sqrt{C_2}\|u(t)\Vert _{H^{1-\alpha}}^{3/4}\| u(t)\Vert _{H^{\alpha+ \varepsilon}}^{3/4}\| u(t)\Vert _{H^{2-\alpha}}^{1/2} + \| f(t)\Vert _H\\ & \leq \| \partial _t u\Vert _{H^{-\alpha}} + \frac{1}{2}\| u(t)\Vert _{H^{2-\alpha}} + \frac{C_2}{2}\|u(t)\Vert _{H^{1-\alpha}}^{3/2}\| u(t)\Vert _{H^{\alpha+\varepsilon}}^{3/2} + \| f(t)\Vert _H.
	\end{align*}
	Thus, \begin{equation}\label{u H2 }
		\| u(t)\Vert _{H^{2-\alpha}} \leq C\left(\| \partial _t u\Vert _{H^{-\alpha}} + \|u(t)\Vert _{H^{1-\alpha}}^{3/2}\| u(t)\Vert _{H^{\alpha+\varepsilon}}^{3/2} + \| f(t)\Vert _H\right)
	\end{equation}
	and \begin{align*}
		\| Du(t) u(t)\Vert _{H^{-\alpha}}^2 &\leq C \|u(t)\Vert _{H^{1-\alpha}}^{3/2}\| u(t)\Vert _{H^{\alpha+\varepsilon}}^{3/2}\| u(t)\Vert _{H^{2-\alpha}}\\ &\leq C\Big(\| \partial _t u\Vert _{H^{-\alpha}} \|u(t)\Vert _{H^{1-\alpha}}^{3/2}\| u(t)\Vert _{H^{\alpha+ \varepsilon}}^{3/2} + \|u(t)\Vert _{H^{1-\alpha}}^3\| u(t)\Vert _{H^{\alpha+\varepsilon}}^3 \\&+ \| f(t)\Vert _H\|u(t)\Vert _{H^{1-\alpha}}^{3/2}\| u(t)\Vert _{H^{\alpha+ \varepsilon}}^{3/2}\Big) \\ &\leq \frac{1}{8} \| \partial _t u\Vert _{H^{-\alpha}}^2 + C\|u(t)\Vert _{H^{1-\alpha}}^3\| u(t)\Vert _{H^{\alpha+ \varepsilon}}^3 + \| f(t)\Vert _H^2.
	\end{align*}
	In conclusion we get 
	\begin{equation}\label{u skt estimate}
		\| \partial_t u\Vert _{H^{-\alpha}}^2 + \frac{d}{d t}\| u(t)\Vert _{H^{1-\alpha}}^2 \leq C\|u(t)\Vert _{H^{1-\alpha}}^3\| u(t)\Vert _{H^{\alpha+\varepsilon}}^3 + C\| f(t)\Vert _H^2.
	\end{equation}
 By interpolation inequality
 \begin{equation}\label{4.8a}
     \|u(t)\Vert _{H^{1-\alpha}}\|u(t)\Vert_{H^{\alpha+ \varepsilon}}^3 \leq \|u(t)\Vert_H^{3-2\alpha -3\varepsilon} \|u(t)\Vert_{H^1}^{1+2\alpha+ 3\varepsilon} = \|u(t)\Vert_H^{2} \|u(t)\Vert_{H^1}^2. 
     \end{equation}
		Thus, 
	\begin{equation}
		\label{u H1}
		\begin{split}
			\| u(t)\Vert _{H^{1-\alpha}}^2 &\leq C\exp\left(C\int_{0}^{t} \|u(s)\Vert _H^2\| u(s)\Vert _{H^1}^2 \,ds \right)\int_{0}^{t} \| f(s)\Vert _H^2 \,ds \\ &+ \| x\Vert _{H^{1-\alpha}}^2\exp\left(C\int_{0}^{t} \|u(s)\Vert _H^2\| u(s)\Vert _{H^1}^2 \,ds \right)
		\end{split}
	\end{equation}
	and 
	\begin{equation}
		\label{u_t L2}
		\begin{split}
			\int_{0}^{T} \| \partial_t u\Vert _{H^{-\alpha}}^2 \,dt  &\leq C\sup_{0 \leq s \leq T}\| u(s)\Vert _{H^{1-\alpha}}^2\int_{0}^{T} \|u(t)\Vert _H^2\| u(t)\Vert _{H^1}^2 \,dt \\ &+ C\int_{0}^{T} \| f(t)\Vert _H^2 \,dt + \| x\Vert _{H^{1-\alpha}}^2,
		\end{split}
	\end{equation}
	where $$\int_{0}^{t} \|u(s)\Vert _H^2\| u(s)\Vert _{H^1}^2 \,ds \leq \sup_{0 \leq r \leq t}\| u(r)\Vert _H^2\int_{0}^{t} \| u(s)\Vert _{H^1}^2 \,ds < \infty,$$ for every $t \in [0,T]$.

	Moreover, thanks to (\ref{u H2 }), we get 
\begin{align*}
	&\ \ \int_{0}^{T} \| u(t)\Vert _{H^{2-\alpha}}^2 \,dt \\ &\leq C \left(\int_{0}^{T} \| \partial_t u\Vert _{H^{-\alpha}}^2 \,dt + \int_{0}^{T} \| f(t)\Vert _H^2 \,dt + \sup_{0 \leq s \leq T}\| u(s)\Vert _{H^{1-\alpha}}^2\int_{0}^{T} \|u(t)\Vert _H^2\| u(t)\Vert _{H^1}^2 \,dt\right) \\ &< \infty.
\end{align*}
\end{proof}

\begin{proposition}\label{skt estimate 2}
	Suppose $T > 0$, $f \in L^2([0,T];H)$, $0 \leq \alpha < 1/2$ and $x \in H$, then the mild solution $u$ to the Eq. (\ref{skt eq 2}) satisfies 
	$$u \in W^{1,2}([t_1,T]; H^{2-\alpha}, H^{-\alpha}) \cap L^\infty([t_1,T];H^{1-\alpha})$$ for any $t_1 \in (0,T)$. 
\end{proposition}
\begin{proof}
	Let us fix $T>0$. By (\ref{u Gronwall 1 b}), we get 
	$$ \| u(t_1)\Vert _{H}^2 + \int _{0}^{t_1}\| u(s)\Vert _{H^{1}}^2 \,ds \leq \int_{0}^{t_1} \| f(s)\Vert _H^2 \,ds + \| x\Vert _H^2 < \infty$$  for any $t_1 \in (0,T)$. Then we can find $t_0 \in (0,t_1)$ such that $\| u(t_0)\Vert _{H^{1-\alpha}} \leq \| u(t_0)\Vert _{H^1} < \infty$. Again by Proposition \ref{skt estimate}, we have $$u \in L^\infty([t_0,T];H^{1-\alpha}) \ ,\ u \in W^{1,2}(t_0,T; H^{2-\alpha}, H^{-\alpha}).$$ Therefore, 
	$u \in W^{1,2}(t_1,T; H^{2-\alpha}, H^{-\alpha}) \cap L^\infty([t_1,T];H^{1-\alpha})$.
\end{proof}

\begin{proposition}\label{proposition x-y}
	Suppose $T > 0$, $f \in L^2([0,T];H)$, $0 \leq \alpha <1/2$ and $x,y \in H$. Denote $u^x$ as the solution of Eq. (\ref{skt eq 2}). Then, we have 
	$$\sup_{0 \leq t \leq T}\left\| u^x(t)-u^y(t)\right\Vert _H^2 \leq \| x-y\Vert_H^2 \exp\left(C\left[\| x\Vert _H^2 + \| y\Vert _H^2 + \left\| f\right\Vert _{L^2([0,T];H)}^2\right]\right) .$$
\end{proposition}
\begin{proof}
	Define $v(t):= u^x(t)-u^y(t)$, then it satisfies equation
	\begin{equation*}
		\begin{cases}
			\partial _t v = \Delta  v(t) + \frac{1}{2}D\left((u^x)^2(t) - (u^y)^2(t)\right),\\
			v(0) = x-y .
		\end{cases}
	\end{equation*}
	We have by a standard estimate and Sobolev inequality that
	\begin{align*}
		\frac{1}{2}\frac{d}{d t}\| v(t)\Vert _H^2 + \| v(t)\Vert _{H^1}^2 &= -\left\langle \left[u^x(t)+u^y(t)\right]\left[u^x(t)-u^y(t)\right], Dv(t)\right\rangle _H \\ &\leq \| v(t)\Vert _H\left\| u^x(t) +u^y(t)\right\Vert _{L^\infty }\| v(t)\Vert _{H^1} \\ &\leq \frac{1}{2} \| v(t)\Vert _H^2\left\| u^x(t) +u^y(t)\right\Vert_{H^1}^2  + \frac{1}{2} \| v(t)\Vert ^2_{H^1}.
	\end{align*}
	Then, $$\frac{d}{d t}\| v(t)\Vert _H^2 + \| v(t)\Vert _{H^1}^2 \leq \| v(t)\Vert _H^2\left\| u^x(t) +u^y(t)\right\Vert_{H^1}^2.$$ Thus, by the Gronwall inequality
	\begin{align*}
		\| v(t)\Vert _H^2 &\leq \| x-y\Vert _H^2\exp\left(\int_{0}^{t} \left\| u^x(s) +u^y(s)\right\Vert_{H^1}^2 \,ds \right) \\ &\leq \| x-y\Vert _H^2\exp\left(2\int_{0}^{t} \left(\| u^x(s)\Vert_{H^1}^2  +\| u^y(s)\Vert_{H^1}^2\right) \,ds \right).
	\end{align*}
Thanks to (\ref{u Gronwall 1 b}), the result follows.
\end{proof}
\subsection{The rate functions}
As we already discussed in Section \ref{trace class}, if $Q_\epsilon := A^\alpha Q_{\delta (\epsilon )} = A^{\alpha} (I + \delta (\epsilon )A^\beta )^{-1}$, $1/2 < \beta -\alpha < 1$, then $\left\{\mathcal{L} (u_\epsilon^x)\right\}_{\epsilon > 0}$ satisfies Freidlin-Wentzell uniform large deviations principle in $L^\infty([0,T];H)$ with rate functions $I^x$ where 
$$I^x(u) = \begin{cases}
	\frac{1}{2}\int_{0}^{T} \| A^{-\frac{\alpha}{2}}\mathcal{H} (u)(t)\Vert_H^2  \,dt & \text{if} \ A^{-\frac{\alpha}{2}} \mathcal{H}(u) \in L^2([0,T];H) \ \text{and} \ u(0) = x,\\
	+ \infty &\text{otherwise}.
\end{cases}$$
For $u \in C((-\infty, 0);H^{-\alpha})$ we define the function 
$$I_{-\infty }(u) := \frac{1}{2}\int_{-\infty}^{0} \|A^{-\frac{\alpha}{2}}\mathcal{H} (u)(t)\Vert ^2_H \,dt .$$ In this section, we will prove the function $I_{- \infty}$ can be a good rate function for LDP. That is to prove the compactness for level sets of $I_{- \infty}$. To state these results, we introduce the following spaces
$$
	\chi = \left\{u \in C((-\infty,0];H^{-\alpha}): \lim_{t \rightarrow -\infty}\| u(t)\Vert _H =0\right\},\  \chi ^{\phi } = \left\{u \in \chi : u(0) = \phi \right\}.
$$
We endow the space $\chi $ with the topology of uniform convergence on compact intervals. We state the following properties of $I_{- \infty}(u)$.
\begin{proposition}\label{u infty estimate}
	Assume that $u \in \chi $ is such that $I_{- \infty}(u) < \infty$ and $0 \leq \alpha < 1/2$. Then $\sup_{t\leq 0}\|u(t)\Vert_{H^{1-\alpha}} < \infty$ and there is a sequence $t_k \rightarrow - \infty$ as $k \rightarrow \infty$ such that $$\lim _{k \rightarrow \infty}\| u(t_k)\Vert _{H^1} = 0 \ \text{and}\ u \in W^{1,2}((-\infty ,0];H^{2-\alpha};H^{-\alpha}).$$
\end{proposition}
\begin{proof}
	Similar to the proof of Proposition \ref{skt estimate 2}, for any $T \geq S \geq 0$ we have that
	$$\| u(-S)\Vert_H^2 + \int_{-T}^{-S} \| u(t)\Vert _{H^1}^2 \,dt \leq \int_{-T}^{-S} \| f(t)\Vert _H^2 \,dt + \| u(-T)\Vert _H^2.$$  
	Since  $I_{-\infty}(u) = \frac{1}{2}\int_{-\infty }^{0} \| f(t)\Vert _H^2 \,dt < \infty$ and $\lim _{t \rightarrow -\infty }\| u(t)\Vert _H^2 = 0$, we can find constant $\bar{C} $ and $T^*$ large enough such that for any $T \geq T^*$ and any $0 \leq S \leq T$, 
 \begin{equation}\label{inequality: H^1}
     \| u(-S)\Vert _H^2 + \int_{-T}^{-S} \| u(t)\Vert _{H^1}^2 \,dt \leq \int_{-\infty }^{-S} \| f(t)\Vert _H^2 \,dt + \| u(-T)\Vert _H^2 \leq \bar{C}. \end{equation}
	 As $u(-S) \in H $, then $u(0)  \in H^{1-\alpha}$ is a direct consequence of Proposition \ref{skt estimate 2}. Moreover, set $S= T_1$ and $T = T_2= 2T_1$. Then we can find $t_1 \in (-T_2, -T_1)$ such that $\| u(t_1)\Vert _{H^1}^2 \leq \frac{\bar{C}}{T_1}$. Next, we set $S = T_2 $ and $T =T_3 = 2 T_2$, then we can find $t_2 \in (-T_3, -T_2)$ such that $\| u(t_2)\Vert _{H^1}^2 \leq \frac{\bar{C}}{T_2}$. Repeating by this procedure, we can find $t_k$ such that $\lim_{k \rightarrow \infty}t_k = -\infty$ and $\lim_{k \rightarrow \infty}\| u(t_k)\Vert _{H^{1-\alpha}} \leq \lim_{k \rightarrow \infty}\| u(t_k)\Vert _{H^1} = 0$.

	Thanks to (\ref{u skt estimate}) and (\ref{4.8a}), we have for $-\infty <s \leq t \leq 0$, 
	\begin{equation*}
		\begin{split}
			\| u(t)\Vert _{H^{1-\alpha}}^2 &\leq C\exp\left(C\int_{s}^{t} \|u(r)\Vert _H^2\| u(r)\Vert _{H^1}^2 \,dr \right)\int_{s}^{t} \| f(r)\Vert _H^2 \,dr \\ &\ \ \ \ + \| u(s)\Vert _{H^{1-\alpha}}^2\exp\left(C\int_{s}^{t} \|u(r)\Vert _H^2\| u(r)\Vert _{H^1}^2 \,dr \right).
		\end{split}
	\end{equation*}
and 
\begin{equation*}
	\begin{split}
		\int_{s}^{0} \| \partial_r u\Vert _{H^{-\alpha}}^2 \,dr  &\leq C\sup_{\rho  \leq 0}\| u(\rho )\Vert _{H^{1-\alpha}}^2\int_{s}^{0} \|u(r)\Vert _H^2\| u(r)\Vert _{H^1}^2 \,dr \\ &\ \ \ \ + C\int_{s}^{0} \| f(r)\Vert _H^2 \,dr + \| u(s)\Vert _{H^{1-\alpha}}^2.
	\end{split}
\end{equation*}
Setting $s = t_k$ and taking the limit as $k \rightarrow \infty$, we get$$\sup_{t \leq 0}\| u(t)\Vert _{H^{1-\alpha}}^2 < \infty$$ and $$\int_{-\infty }^{0} \| \partial_t u\Vert _{H^{-\alpha}}^2 \,dt < \infty.$$ Moreover, by (\ref{u H2 }) and (\ref{4.8a})
\begin{align*}
	\int_{-\infty}^{0} \| u(t)\Vert _{H^{2-\alpha}}^2 \,dt &\leq C \bigg(\int_{-\infty }^{0} \| \partial_t u\Vert _{H^{-\alpha}}^2 \,dt + \int_{-\infty}^{0} \| f(t)\Vert _H^2 \,dt\\ &\ \quad \quad \ \ + \sup_{s \leq 0 }\| u(s)\Vert _{H^{1-\alpha}}^2\int_{-\infty }^{0} \|u(t)\Vert _H^2\| u(t)\Vert _{H^1}^2 \,dt\bigg) < \infty.
\end{align*}

\end{proof}
Before we prove the compactness for the level sets of $I_{- \infty}$, we need to show that $I_{- \infty}$ is lower semi-continuous. The proof of following proposition is analogous to the proof of Proposition 5.4 in \cite{brzezniak2015quasipotential}.
\begin{proposition}\label{p:lsc}
	Suppose $0 \leq \alpha < 1/2$, then the function $I_{-\infty}$ is lower semi-continuous in $\chi $.
\end{proposition}
\begin{proof}
	In order to prove the lower semi-continuity of $I_{-\infty}$, it is sufficient to show that if a $\chi- $valued sequence $\{u_n\}_{n=1}^\infty $ is convergent in $\chi $ to a function $u \in \chi $, then
	\begin{equation}\label{lsc}
		\liminf_{n \rightarrow \infty } I_{-\infty}(u_n) \geq I_{-\infty}(u).
	\end{equation} 
	First, we assume $u \in \chi $ such that $I_{-\infty}(u)= \infty$. We will show that $$\liminf_{n \rightarrow \infty } I_{-\infty}(u_n) = +\infty.$$
	Suppose by contradiction that $\liminf_{n \rightarrow \infty} I_{-\infty}(u_n) < +\infty$, then we can choose a subsequence still denoted as $u_n$ such that for all $n \in \mathbb{N} $
	$$I_{-\infty}(u_n) = \frac{1}{2}\|A^{-\frac{\alpha}{2}} \mathcal{H} (u_n)\Vert^2 _{L^2((-\infty,0];H)} \leq C,$$ for some positive constant $C$. According to Proposition \ref{u infty estimate}, the sequence is bounded in space $W^{1,2}((-\infty ,0];H^{2-\alpha};H^{-\alpha}) $. Hence, we can find $\tilde{u} \in W^{1,2}((-\infty ,0];H^{2-\alpha};H^{-\alpha})$ such that after further extraction of a subsequence 
	$$u_n \rightarrow \tilde{u} , \text{as} \ n\rightarrow\infty, \ \text{weakly in} \ W^{1,2}((-\infty ,0];H^{2-\alpha};H^{-\alpha}).$$ By the uniqueness of the limit,  we know that $u = \tilde{u} $. Then we have $ u \in W^{1,2}((-\infty ,0];H^{2-\alpha};H^{-\alpha})$. By Proposition \ref{u infty estimate}, noting that $\sup_{t \leq 0}\|u_n(t)\Vert_{H^{1-\alpha}} $ only depends on $f$, so there exists a positive constant $\bar{C}$ such that for all $n \in \mathbb{N}$, $\sup_{t \leq 0}\|u_n(t)\Vert_{H^{1-\alpha}} \leq \bar{C}.$ In particular, for fixed $t \leq 0$, $u_n(t)$ is bounded in $H^{1-\alpha}$. Then, we can find a subsequence $u_{n_k^t}(t)$ such that $u_{n_k^t}(t)$ converges to $\bar{u}(t)$ weakly in $H^{1-\alpha}$ as $k \rightarrow \infty$. Since $u_n(t)$ converges to $u(t)$ strongly in $H^{-\alpha}$. By the uniqueness of the limit, we know that $u(t)=\bar{u}(t)$ and $\liminf_{k \rightarrow \infty}\|u_{n_k^t}(t)\Vert_{H^{1-\alpha}} \geq  \|u(t)\Vert_{H^{1-\alpha}}$. Then, $$\sup_{t \leq 0}\|u(t)\Vert_{H^{1-\alpha}} \leq \sup_{t \leq 0}\liminf_{k \rightarrow \infty}\|u_{n_k^t}(t)\Vert_{H^{1-\alpha}} \leq \sup_{t \leq 0}\sup_{n \in \mathbb{N}}\|u_{n}(t)\Vert_{H^{1-\alpha}} = \sup_{n \in \mathbb{N}} \sup_{t \leq 0}\|u_{n}(t)\Vert_{H^{1-\alpha}} \leq \bar{C}.$$
 Moreover, by (\ref{ineq:nonlinear term}) and the Poincar\'e inequality, 
 \begin{align*}
     \left\|\frac{1}{2}D(u^2)\right\Vert_{L^2((-\infty, 0];H^{-\alpha}))} &= \left(\int_{-\infty}^{0}\|Du(t)u(t)\Vert_{H^{-\alpha}}^2 \, dt \right)^{1/2} \\ &\leq C\left(\int_{-\infty}^{0}\|u(t)\Vert_{H^{1-\alpha}}^2\|u(t)\Vert_{H^{2-\alpha}}^2 \, dt \right)^{1/2} \\& \leq C\sup_{t\leq 0}\|u(t)\Vert_{H^{1-\alpha}} \left\|u\right\Vert_{L^2((-\infty,0];H^{2-\alpha})} < +\infty.   
     \end{align*}
     Therefore, 
 $$I_{-\infty}(u) = \frac{1}{2}\left\|A^{-\frac{\alpha}{2}}\left( \partial_tu - \Delta u - \frac{1}{2}D(u^2)\right)\right\Vert _{L^2((-\infty,0];H)} < +\infty,$$ which contradicts with our assumption.

	Thus, we can assume that $I_{-\infty}(u) < \infty$. By Proposition \ref{u infty estimate}, we have $u(0) \in H^{1-\alpha}$ and $u \in W^{1,2}((-\infty ,0];H^{2-\alpha};H^{-\alpha})$. If (\ref{lsc}) is not true, then we can choose some $\epsilon >0$ and a subsequence still denoted by $u_n$ such that
 \begin{equation} \label{4.13a}
     I_{-\infty}(u_n) < I_{-\infty}(u) -\epsilon, \ n \in \mathbb{N}.  \end{equation}
  Hence, we have $f_n := A^{-\frac{\alpha}{2}}\mathcal{H} (u_n)$ is bounded in $L^2((-\infty , 0];H)$. Proposition \ref{u infty estimate} implies that $\left\{u_n\right\}_{n \in \mathbb{N} }$ is bounded in $W^{1,2}((-\infty ,0];H^{2-\alpha};H^{-\alpha})$. Then we can find $\tilde{u} \in W^{1,2}((-\infty ,0];H^{2-\alpha};H^{-\alpha})$ and a subsequence $u_n$ such that 
	\begin{equation}
		\label{4.9a}
		u_n \rightarrow \tilde{u} , \text{as} \ n\rightarrow\infty, \ \text{weakly in} \ W^{1,2}((-\infty ,0];H^{2-\alpha};H^{-\alpha}).
	\end{equation}
	Similar as before, we have $u = \tilde{u} $. Since the sequence $\{f_n\}$ associated with the subsequence $u_n$ selected above is bounded in $L^2((- \infty ,0];H)$, we can find $\tilde{f}$ and a new subsequence $f_n$ such that
	$$f_n \rightarrow \tilde{f} , \text{as} \ n\rightarrow\infty, \ \text{weakly in} \ L^2((- \infty ,0];H).$$ Now associated with the subsequence $f_n$, the new subsequence of $u_n$,  still denoted by $u_n$, also satisfies (\ref{4.9a}). Next, we will show $\tilde{f} = f :=A^{-\frac{\alpha}{2}} \mathcal{H} (u)$.

	For any fixed $T>0$ and any $\varphi \in L^2([- T ,0];H^{1})$, by the weak convergence of $f_n$ to $\tilde{f} $ in $L^2((- \infty ,0];H)$, we get 
	\begin{equation}\label{f align 1}
		\begin{split}
			&\ \ \ \ \int_{-T}^{0} \langle \tilde{f}(t),\varphi (t)\rangle _H  \,dt \\&=  \lim_{n \rightarrow \infty}\int_{-T}^{0} \langle f_n(t),\varphi (t)\rangle _H  \,dt \\ &= \lim_{n \rightarrow \infty}\int_{-T}^{0} \langle A^{-\frac{\alpha}{2}}\partial_tu_n,\varphi (t)\rangle _H  \,dt - \lim_{n \rightarrow \infty}\int_{-T}^{0} \langle A^{-\frac{\alpha}{2}}\Delta u_n(t),\varphi (t)\rangle _H  \,dt \\&\ \ \ \ - \lim_{n \rightarrow \infty}\frac{1}{2}\int_{-T}^{0} \langle A^{-\frac{\alpha}{2}} D(u_n^2(t)),\varphi (t)\rangle _H  \,dt.
		\end{split}
	\end{equation}
Since $A^{\frac{\alpha}{2}} \varphi \in L^2([-T,0];H^{1-\alpha}) \subset L^2([-T,0]; H^{-\alpha})$, $A^{\frac{\alpha}{2}-1} \varphi \in L^2([-T,0];H^{3-\alpha}) \subset L^2([-T,0];\\H^{2-\alpha})$, by the weak convergence of $u_n$ to $u$ in $W^{1,2}((-\infty ,0];H^{2-\alpha};H^{-\alpha})$, we have 
\begin{equation}
\begin{split}\label{f align 2}
\lim_{n \rightarrow \infty}\int_{-T}^{0} \langle  A^{-\frac{\alpha}{2}} \partial_tu_n,\varphi (t)\rangle _H  \,dt &=\lim_{n \rightarrow \infty} \int_{-T}^{0} \langle \partial_tu_n, A^{\frac{\alpha}{2}} \varphi (t)\rangle _{H^{-\alpha}}  \,dt \\&= \int_{-T}^{0} \langle \partial_tu, A^{\frac{\alpha}{2}} \varphi (t)\rangle _{H^{-\alpha}}  \,dt
\end{split}
 \end{equation} 
and 
\begin{equation}\label{f align 3}\begin{split}
    	-\lim_{n \rightarrow \infty}\int_{-T}^{0} \langle A^{-\frac{\alpha}{2}} \Delta u_n(t),\varphi (t)\rangle _H  \,dt &= \lim_{n \rightarrow \infty}\int_{-T}^{0} \langle u_n(t), A^{\frac{\alpha}{2}-1} \varphi (t)\rangle _{H^{2-\alpha}}  \,dt \\&= \int_{-T}^{0} \langle u(t), A^{\frac{\alpha}{2}-1}  \varphi (t)\rangle _{H^{2-\alpha}}   \,dt
     \end{split}
 \end{equation}
Moreover, by integration by parts formula 
\begin{align*}
		&\ \ \ \ \lim_{n \rightarrow \infty}\left| \int_{-T}^{0} \langle A^{-\frac{\alpha}{2}}\left[D(u_n^2(t)) - D(u^2(t))\right],\varphi (t)\rangle _H  \,dt \right\vert\\
        &= \lim_{n \rightarrow \infty}\left| \int_{-T}^{0} \int_{0}^{1} A^{-\frac{\alpha}{2}}[(u_n(t) - u(t))(u_n(t)+u(t))]\left(D\varphi (t)\right) \,d\xi  \,dt\right\vert  \\ &\leq \liminf_{n \rightarrow \infty }\int_{-T}^{0} \left\|(u_n(t) -u(t))(u_n(t)+u(t))\right\Vert_{H^{-\alpha}} \|\varphi(t)\Vert_{H^1}\,dt \\ &\leq \liminf_{n \rightarrow \infty }\int_{-T}^{0}\|u_n(t)-u(t)\Vert_{H^{-\alpha}}\|u_n(t)+u(t)\Vert_{C^{\alpha+\varepsilon}}\|\varphi(t)\Vert_{H^1}\,dt\\& \leq \liminf_{n \rightarrow \infty }\|u_n-u\Vert _{C([-T,0];H^{-\alpha})}\| u_n+u\Vert _{L^2([-T,0];C^{\alpha+\varepsilon})}\| \varphi (t)\Vert _{L^2([- T ,0];H^1)}.
\end{align*}
Here, $\epsilon = \frac{1-2\alpha}{3}$ as in the proof of Proposition \ref{skt estimate}.
Thanks to that $u_n$ and $u$ are bounded in $W^{1,2}((-\infty ,0];H^{2-\alpha};H^{-\alpha})$, by Sobolev inequality and  the Poincar\'e inequality, we get $$\left\| u_n+u\right\Vert _{L^2([-T,0];C^{\alpha+\varepsilon})} \leq C\left\| u_n+u\right\Vert _{L^2([-T,0];H^{1/2+\alpha+ \varepsilon})} \leq C\left\| u_n+u\right\Vert _{L^2((-\infty,0];H^{2-\alpha})} < \infty.$$ By a standard compactness argument, $W^{1,2}([-T ,0];H^{2-\alpha};H^{-\alpha}) $ is compactly embedded into $C([-T ,0];H^{-\alpha})$. Hence, $u_n$ converges to $u$ in $C([-T ,0];H^{-\alpha})$. It then follows that
\begin{equation}\label{f align 4}
	\lim_{n \rightarrow \infty}\frac{1}{2}\int_{-T}^{0} \langle A^{-\frac{\alpha}{2}}D(u_n^2(t)),\varphi (t)\rangle _H  \,dt = \frac{1}{2}\int_{-T}^{0} \langle A^{-\frac{\alpha}{2}} D(u^2(t)),\varphi (t)\rangle _H  \,dt.
\end{equation}
 Combining (\ref{f align 1}), (\ref{f align 2}), (\ref{f align 3}) and (\ref{f align 4}) together, we get 
 \begin{align*}
	\int_{-T}^{0} \langle \tilde{f}(t),\varphi (t)\rangle _H  \,dt &= \int_{-T}^{0} \langle A^{-\frac{\alpha}{2}} \partial_tu,\varphi (t)\rangle _H  \,dt - \int_{-T}^{0} \langle A^{-\frac{\alpha}{2}} \Delta u(t),\varphi (t)\rangle _H  \,dt\\
     &\ \ \ \ - \frac{1}{2}\int_{-T}^{0} \langle A^{-\frac{\alpha}{2}} D(u^2(t)),\varphi (t)\rangle _H  \,dt \\ &= \int_{-T}^{0} \langle f(t),\varphi (t)\rangle _H  \,dt
 \end{align*}
 for all $\varphi \in L^2((- \infty ,0];H^1)$. Since $ L^2((- \infty ,0];H^1)$ is dense in $L^2((- \infty ,0];H)$, it follows that for any $T>0$, $f =\tilde{f} $ in $L^2([-T,0];H)$. Then, $\|f-\tilde{f}\Vert_{L^2((-\infty,0];H)} =0$. Thus,$$f_n \rightarrow f , \text{as} \ n\rightarrow\infty, \ \text{weakly in} \ L^2((- \infty ,0];H)$$ and 
 \begin{align*}
	 \liminf_{n \rightarrow \infty}I_{-\infty}(u_n) &= \liminf_{n \rightarrow \infty}\frac{1}{2}\left\| f_n\right\Vert _{L^2((-\infty,0];H)}\\ &\geq \frac{1}{2}\left\| f\right\Vert _{L^2((-\infty,0];H)} = I_{-\infty}(u),
 \end{align*}
 which contradicts with our assumption (\ref{4.13a}). The proof is completed.
\end{proof}

Now we are ready to prove the main result of this section.
\begin{theorem}\label{level infty}
	Suppose $0 \leq \alpha < 1/2$, then the operator $I_{-\infty}$ has compact level sets in $\chi $.
\end{theorem}
\begin{proof}
	We need to prove that every sequence $\{u_n\}$ in $\chi $ such that $I_{-\infty}(u_n) \leq r$, for any $n \in \mathbb{N} $, has a subsequence still denoted by $u_n$ convergent in $\chi $ to some $u \in \chi $ such that $I_{-\infty}(u)\leq r$. According to Proposition \ref{u infty estimate}, there exists $M >0$ such that $\left\| u_n\right\Vert _{W^{1,2}((-\infty ,0];H^{2-\alpha};H^{-\alpha})} \leq M$, $n \in \mathbb{N} $. Then, we can find $u \in W^{1,2}((-\infty ,0];H^{2-\alpha};H^{-\alpha})$ and a subsequence of $u_n$ that is weakly convergent to $u$ in $W^{1,2}((-\infty ,0];H^{2-\alpha};H^{-\alpha})$. 
 By a standard compactness argument, the space $W^{1,2}([-T ,0];H^{2-\alpha};H^{-\alpha})$ is compact embedded into $C([-T,0];H^{-\alpha})$ for $T>0$. Thus, for each $T>0$ we can extract a subsequence of $u_n$ which is strongly convergent in $C([-T,0];H^{-\alpha})$. By the uniqueness of the limit we know that the limit is equal to the restriction of $u$ to the interval $[-T,0]$. Define $f:=A^{-\frac{\alpha}{2}}\mathcal{H}(u)$ and $f_n:=A^{-\frac{\alpha}{2}}\mathcal{H}(u_n)$. After applying similar arguments as in Proposition \ref{p:lsc}, it follows that there exists a  positive constant $\bar{C}$, which does not depend on $T$, such that $$\sup_{t \in [-T,0]}\|u(t)\Vert_{H} \leq \sup_{t \in [-T,0]}\|u(t)\Vert_{H^{1-\alpha}} \leq \bar{C}.$$ In addition, we have $$f_n \rightarrow f , \text{as} \ n\rightarrow\infty, \ \text{weakly in} \ L^2((- \infty ,0];H)$$ and 
 \begin{equation*}
	\liminf_{n \rightarrow \infty}\left\| f_n\right\Vert _{L^2((-\infty,0];H)}\\ \geq \left\| f\right\Vert _{L^2((-\infty,0];H)}. 
 \end{equation*} 
 Thus, by (\ref{inequality: H^1}) and a similar argument as in Proposition \ref{u infty estimate}, there exists a sequence $t_k \rightarrow -\infty$ as $k \rightarrow \infty$ such that
 $$\lim_{k \rightarrow \infty}\|u(t_k)\Vert_H \leq \lim_{k \rightarrow \infty}\|u(t_k)\Vert_{H^1} =0.$$ For any $\delta > 0 $, there exits $S$ sufficiently large such that $\int_{-\infty}^{-S}\|f(s)\Vert_H^2 \,ds \leq \delta/2 $. Then, for any $t \leq -S$, by (\ref{inequality: H^1}), we get $$\|u(t)\Vert_H^2 \leq \int_{-\infty}^{-S}\|f(s)\Vert_H^2 \,ds + \|u(-T)\Vert_H^2 \leq \delta,$$ where $T \in \{t_k \}_{k \in \mathbb{N}}$ is chosen to be large enough such that $-T < t$ and $\|u(-T)\Vert_H^2 \leq \delta/2.$ Hence, we have $u \in \chi $. Moreover, by employing the Helly's diagonal procedure, we can find a subsequence of the sequence after extraction which is convergent to $u$ in $\chi $. As $I_{-\infty }$ is lower semi-continuous, we have 
	$$I_{-\infty }(u) \leq \liminf_{n \rightarrow \infty }I_{-\infty }(u_n) \leq r.$$ The proof is completed.
\end{proof}

\subsection{The quasi-potentials}
In this section, we will prove that the function $U$ can be a good rate function for LDP in $H^{-\alpha}$ that is the level sets of $U$ is compact in $H^{-\alpha}$. Some steps of the proof in this section are analogous to those used in Theorem 6.2 of \cite{brzezniak2015quasipotential}. 
\begin{proposition}\label{U phi H1}
	Suppose $0 \leq \alpha < 1/2$, then $$U^\alpha(\phi )<\infty {\rm \ if \  and \ only \ if\ } \phi \in H^{1-\alpha}.$$
\end{proposition}
\begin{proof}
	Assume $U^\alpha(\phi ) < \infty $. According to the definition of quasi-potential (\ref{quasi potential 1}), for any $\varepsilon^\star>0$,  we can find $T>0$ and $u \in C([0,T];H^{-\alpha})$ such that $u(0) = 0$, $u(T)=\phi $, 
	$f \in L^2([0,T];H)$ with
	$$\partial _t u = \Delta  u(t) + \frac{1}{2}D(u^2(t)) + A^{\frac{\alpha}{2}}f(t)$$ and $I_T(u) - \varepsilon^\star < U^\alpha(\phi)$. By Proposition \ref{skt estimate}, we get $u \in L^\infty([0,T];H^{1-\alpha})$. Hence, $u(T)=\phi \in H^{1-\alpha}$.

	Conversely, assume $\phi \in H^{1-\alpha}$. For any $T>0$ , we find $u_1$ satisfying
	\begin{equation*}
	\begin{cases}
		\partial _t u_1(t) = \Delta  u_1(t) + \frac{1}{2}D(u_1^2(t)), \\
		u_1(0) = \phi  \in H^{1-\alpha},
	\end{cases}
	\end{equation*}
	then by Proposition \ref{skt estimate}, we get $u_1 \in W^{1,2}([0,T]; H^{2-\alpha}, H^{-\alpha})$. Hence, there exists $t_0 \in (0,T)$ such that $u_1(t_0) \in H^{2-\alpha}$.

	Next, we define $$u_2(t):= \frac{t}{T-t_0}u_1(t_0), t \in [0,T-t_0].$$ Notice that $u_2(0)=0$, $u_2(T-t_0)= u_1(t_0) \in H^{2-\alpha}$ and 
	$$\mathcal{H} (u_2(t))= \frac{1}{T-t_0}u_1(t_0)-\frac{t}{T-t_0}\Delta u_1(t_0) - \frac{t^2}{(T-t_0)^2}D\left(u_1(t_0)\right)u_1(t_0).$$ Therefore, we have
	\begin{align*}
		&\ \ \ \ \frac{1}{2}\int_{0}^{T-t_0} \left\| A^{-\frac{\alpha}{2}}\mathcal{H} \left(u_2(t)\right)\right\Vert _H^2 \,dt \\ &\leq \frac{3}{2}\int_{0}^{T-t_0} \Bigg[ \frac{1}{(T-t_0)^2} \| u_1(t_0)\Vert _{H^{-\alpha}}^2 + \frac{t^2}{(T-t_0)^2}\| u_1(t_0)\Vert _{H^{2-\alpha}}^2 \\ &\ \ \ \ + \frac{t^4}{(T-t_0)^4}\| Du_1(t_0) u_1(t_0)\Vert_{H^{-\alpha}}^2 \Bigg] \,dt \\ &\leq \frac{3}{2(T-t_0)}\| u_1(t_0)\Vert _{H^{-\alpha}}^2 + \frac{T-t_0}{2}\| u_1(t_0)\Vert _{H^{2-\alpha}}^2 + \frac{3(T-t_0)}{10}\| u_1(t_0)\Vert _{H^{2-\alpha}}^4\\ & < \infty,
	\end{align*}
where the second inequality is due to (\ref{ineq:nonlinear term}) and the Poincar\'e inequality.
	Finally, we define 
	$$u(t):=\begin{cases}
		u_2(t), &t \in [0,T-t_0], \\ 
		u_1(T-t), &t \in [T-t_0,T].
	\end{cases}$$
	It is easy to verify $u(0)= 0$, $u(T)= \phi $ and $u \in C([0,T];H^{-\alpha})$. Moreover, since $u_1 \in W^{1,2}(0,T; H^{2-\alpha}, H^{-\alpha})$, we have 
	\begin{align*}
	\frac{1}{2}\int_{T-t_0}^{T} \left\| A^{-\frac{\alpha}{2}}\mathcal{H} \left(u_1(T-t)\right)\right\Vert _H^2 \,dt &= \frac{1}{2}\int_{0}^{t_0} \left\| -\partial_t u _1 (t) - \Delta u_1(t) - \frac{1}{2}D\left(u_1^2(t)\right) \right\Vert _{H^{-\alpha}}^2 \,dt \\ &= \frac{1}{2}\int_{0}^{t_0} \left\| -2\partial_t u _1 (t) \right\Vert _{H^{-\alpha}}^2 \,dt\\ &< \infty.
	\end{align*}
	 Then, \begin{align*}
		I_T(u) &= \frac{1}{2}\int_{0}^{T-t_0} \left\|A^{-\frac{\alpha}{2}} \mathcal{H} \left(u_2(t)\right)\right\Vert _H^2 \,dt + \frac{1}{2}\int_{T-t_0}^{T} \left\|A^{-\frac{\alpha}{2}} \mathcal{H} \left(u_1(T-t)\right)\right\Vert _H^2 \,dt < \infty.
	\end{align*}
	Thus, by the definition of $U^\alpha(\phi )$, it is easy to see that
	$$ U^\alpha(\phi ) \leq I_T(u) < \infty.$$
\end{proof}

In the following, we prove two lemmas that will be used later. 

\begin{lemma}\label{lemma 2}
	Suppose $0 \leq \alpha < 1/2$, then for every $T>0$ and $\varepsilon >0$, there exists $\eta >0$, such that for any $y \in H^{1-\alpha}$ and $\| y\Vert _{H^{1-\alpha}}\leq\eta $, we can find 
	$$v \in W^{1,2}([0 ,T];H^{2-\alpha};H^{-\alpha}),$$ with $I_T(v)<\varepsilon $, $v(0)=0$ and $v(T)=y$.
\end{lemma}
\begin{proof}
	We first set 
	\begin{equation*}
		\begin{cases}
			\partial _t v_1 = \Delta  v_1(t) + \frac{1}{2}D(v_1^2(t)), \\
			v_1(0) = y  \in H^{1-\alpha}.
		\end{cases}
		\end{equation*}
		By (\ref{u Gronwall 1}), we get that for any $t \geq 0$,
		$$\int_{0}^{t} \| v_1(s)\Vert _{H^1}^2 \,ds \leq \| v_1(0)\Vert _H^2 \leq \eta ^2$$ and $$\| v_1(t)\Vert _H^2 \leq \| v_1(0)\Vert _H^2 \leq \eta ^2.$$ So by (\ref{u H1}), it follows that
		$$\| v_1(t)\Vert _{H^{1-\alpha}}^2 \leq \| y\Vert _{H^{1-\alpha}}^2\exp\left(C\sup_{0 \leq s \leq t}\| v_1(s)\Vert _H^2\int_{0}^{t} \| v_1(s)\Vert _{H^1}^2 \,ds \right) \leq \eta ^2\exp(C\eta ^4)$$ and $$\sup_{0 \leq r \leq t}\| v_1(r)\Vert _{H^{1-\alpha}}^2 \int_{0}^{t} \| v_1(s)\Vert _{H^1}^2 \| v_1(s)\Vert _H^2 \,ds  \leq \eta ^6\exp(C\eta ^4).$$ Then, by (\ref{u_t L2}), we have that
  \begin{align*}
      \int_{0}^{t} \| \partial_sv_1(s)\Vert _{H^{-\alpha}}^2 \,ds &\leq \| v_1(0)\Vert _{H^{1-\alpha}}^2 + C \sup_{0 \leq r \leq t}\| v_1(r)\Vert _{H^{1-\alpha}}^2 \int_{0}^{t} \| v_1(s)\Vert _{H^1}^2\| v_1(s)\Vert _H^2 \,ds \\ &\leq \eta ^2 + C\eta ^6\exp(C\eta ^4).
    \end{align*}
	 Thanks to the above inequalities and (\ref{u H2 }) we get
  \begin{align*}
      \int_{0}^{t} \| v_1(s)\Vert _{H^{2-\alpha}}^2 \,ds &\leq C\left(\int_{0}^{t} \| v_1^\prime(s)\Vert _{H^{-\alpha}}^2 \,ds + \sup_{0 \leq r \leq t}\| v_1(r)\Vert _{H^{1-\alpha}}^2 \int_{0}^{t} \| v_1(s)\Vert _{H^1}^2\| v_1(s)\Vert _H^2 \,ds \right)\\ &\leq C\eta ^2 + C\eta ^6\exp(C\eta ^4).
      \end{align*}
		Now we fix $T>0$ and choose $t = T/2$. Then, there exists $t_0 \in (0, t)$ such that $$\| v_1(t_0)\Vert _{H^{2-\alpha}}^2 \leq \frac{C\eta ^2 + C\eta ^6\exp(C\eta ^4)}{t} = \frac{2\left(C\eta ^2 + C\eta ^6\exp(C\eta ^4)\right)}{T}.$$

		Next, we set $v_2(t) := \frac{t}{T-t_0}v_1(t_0)$, $t \in [0,T-t_0]$. We have $v_2(0) = 0$, $v_2(T-t_0) = v_1(t_0)$ and 
		$$\mathcal{H} (v_2(t))= \frac{1}{T-t_0}v_1(t_0)-\frac{t}{T-t_0}D^2v_1(t_0) - \frac{t^2}{(T-t_0)^2}D\left(v_1(t_0)\right)v_1(t_0).$$ Similar to the computations in the proof of Proposition \ref{U phi H1}, we have
		\begin{align*}
			\frac{1}{2}\int_{0}^{T-t_0} \left\| A^{-\frac{\alpha}{2}} \mathcal{H} \left(v_2(t)\right)\right\Vert _H^2 \,dt &\leq \frac{3}{2}\int_{0}^{T-t_0} \Bigg[\frac{1}{(T-t_0)^2} \| v_1(t_0)\Vert _{H^{-\alpha}}^2 + \frac{t^2}{(T-t_0)^2}\| v_1(t_0)\Vert _{H^{2-\alpha}}^2 \\ &\ \ \ \ + \frac{t^4}{(T-t_0)^4}\| Du_1(t_0) u_1(t_0)\Vert_{H^{-\alpha}}^2\Bigg]\,dt \\ &\leq \frac{3}{2(T-t_0)}\| v_1(t_0)\Vert _H^2 + \frac{T-t_0}{2}\| v_1(t_0)\Vert _{H^{2-\alpha}}^2 \\ &\ \ \ \ + \frac{3(T-t_0)}{10}\| v_1(t_0)\Vert _{H^{1-\alpha}}^2 \| v_1(t_0)\Vert _{H^{2-\alpha}}^2 \\  &\leq C(1+1/T)\eta ^2 + C(\eta ^6 + \eta ^8 + \eta^{12})\exp(C\eta ^4),
		\end{align*}
		where $C >0$ is a constant. The last inequality is derived from $t_0 < T/2$.

		Finally, we define $$v(t):=\begin{cases}
			v_2(t), &t \in [0,T-t_0], \\ 
			v_1(T-t), &t \in [T-t_0,T].
		\end{cases}$$
		We can check that $v(0)=0$ and $v(T)= y$. Similar to the computations in Proposition \ref{U phi H1}, we get 
		\begin{align*}
			\frac{1}{2}\int_{T-t_0}^{T} \left\| A^{-\frac{\alpha}{2}}\mathcal{H} \left(v_1(T-t)\right)\right\Vert _H^2 \,dt &= \frac{1}{2}\int_{0}^{t_0} \left\| -v^\prime _1 (t) - \Delta v_1(t) - \frac{1}{2}D\left(v_1^2(t)\right) \right\Vert _{H^{-\alpha}}^2 \,dt \\ &= \frac{1}{2}\int_{0}^{t_0} \left\| -2v^\prime _1 (t) \right\Vert _{H^{-\alpha}}^2 \,dt\\ &\leq  2\eta ^2 + 2C\eta ^6\exp(C\eta ^4).
		\end{align*}
		Finally we have \begin{align*}
				I_T(v) &= \frac{1}{2}\int_{0}^{T-t_0} \left\| A^{-\frac{\alpha}{2}}  \mathcal{H} \left(v_2(t)\right)\right\Vert _H^2 \,dt + \frac{1}{2}\int_{T-t_0}^{T} \left\| A^{-\frac{\alpha}{2}}  \mathcal{H} \left(v_1(T-t)\right)\right\Vert _H^2 \,dt\\ & \leq C(1+1/T)\eta ^2 + C(\eta ^6 + \eta ^8 + \eta^{12})\exp(C\eta ^4)
		\end{align*}
		for some constant positive constant $C$.
		Therefore, for every $T,\varepsilon >0$, we can choose $\eta $ sufficiently small such that 
		$$I_T(v) \leq  C(1+1/T)\eta ^2 + C(\eta ^6 + \eta ^8 + \eta^{12})\exp(C\eta ^4) <\varepsilon .$$
\end{proof}

\begin{lemma}\label{lemma 3}
	Assume that $u \in \chi $ and $0 \leq \alpha < 1/2$. Then for each $\varepsilon >0$, we can find $T_\varepsilon >0$ and $v_\varepsilon  \in C([-T_\varepsilon,0];H^{-\alpha})$ such that $v_\varepsilon (-T_\varepsilon ) =0$, $v_\varepsilon (0) =u(0)$ and 
	$$I_{-T_\varepsilon } (v_\varepsilon )\leq I_{-\infty} (u) + \varepsilon .$$
\end{lemma}
\begin{proof}
	For $u \in \chi $, we denote $u(0)= \phi $. If $I_{-\infty}(u) = +\infty$, the result is obvious. We assume in the following that $I_{-\infty}(u)< \infty$. Then by Proposition \ref{u infty estimate}, we get $\phi \in H^{1-\alpha}$, $u \in W^{1,2}((-\infty ,0];H^{2-\alpha};H^{-\alpha})$. Moreover there exists a sequence $t_k$ such that $t_k \rightarrow -\infty $ as $k \rightarrow \infty $ and $\lim_{k \rightarrow \infty }\| u(t_k)\Vert _{H^{1-\alpha}} =0$. Then, fix $\varepsilon >0$ and choose $T_\varepsilon -1  = - t_k >0$ for $k$ sufficiently large such that $$\| u(-T_\varepsilon+ 1 )\Vert _{H^{1-\alpha}} =\| u(t_k)\Vert _{H^{1-\alpha}} <\eta,$$ where we choose $\eta >0$ as in Lemma \ref{lemma 2}, for $T=1$ and the $\varepsilon $ here. Then by Lemma \ref{lemma 2}, we can find $v \in W^{1,2}([-T_\varepsilon  ,-T_\varepsilon+1] ;H^{2-\alpha};H^{-\alpha})$ such that 
	$$I_{-T_\varepsilon ,-T_\varepsilon +1}(v) = I_{-T_\varepsilon}(v) - I_{-T_\varepsilon +1}(v)<\varepsilon ,\ v(-T_\varepsilon )=0 \ \text{and}\ v(-T_\varepsilon +1)= u(-T_\varepsilon+1 ) .$$

	Finally, we define 
	$$v_\varepsilon (t):=\begin{cases}
		v(t), &t \in [-T_\varepsilon ,-T_\varepsilon +1] ,\\ 
		u(t), &t \in [-T_\varepsilon +1,0].
	\end{cases}$$
	Notice that $v_\varepsilon (-T_\varepsilon )=0$, $v_\varepsilon (0) = u(0)$ and $v_\varepsilon \in W^{1,2}([-T_\varepsilon  ,0] ;H^{2-\alpha};H^{-\alpha})$. Moreover, we have \begin{align*}
		I_{-T_\varepsilon }(v_\varepsilon ) &= I_{-T_\varepsilon ,-T_\varepsilon +1}(v) + I_{-T_\varepsilon+1 }(u)\\ &< \varepsilon  + \left(I_{-\infty}(u)-I_{-\infty, -T_\varepsilon +1}(u)\right) \\ &< \varepsilon  + I_{-\infty}(u).
	\end{align*}
\end{proof}
Now we can prove the following essential characterization of the function $U^\alpha$.
	
\begin{theorem}\label{characterization of U}
	Suppose $0 \leq \alpha < 1/2$, then for any $\phi  \in H^{1-\alpha}$, we have 
	$$U^\alpha(\phi ) = \min \left\{I_{-\infty}(u): u \in \chi ^\phi \right\}.$$
\end{theorem}
\begin{proof}
	We use the formulation (\ref{quasi potential 2}) of $U^\alpha(\phi )$. Let $T>0$ be fixed and $u \in C([-T,0];H^{-\alpha})$ such that $u(-T)=0$, $u(0)=\phi $ and $I_{-T}(u) < \infty$. 

	Set 
	$$\bar{u}  (t):=\begin{cases}
		u(t), &t \in [-T ,0], \\ 
		0, &t \in (-\infty,-T].
	\end{cases}$$
	Obviously, $\bar{u} \in \chi ^\phi $ and $I_{-\infty }(\bar{u} ) = I_{-T}(u)$. In particular, this implies that
	$$\inf\left\{I_{-\infty}(v): v \in \chi ^\phi \right\} \leq I_{-T}(u).$$
	Taking the infimum over all $u \in C([-T,0];H^{-\alpha})$ and $T>0$, we get
	$$\inf\left\{I_{-\infty}(u): u \in \chi ^\phi \right\} \leq U^\alpha(\phi ).$$

	By Lemma \ref{lemma 3}, for any $u \in \chi ^\phi $ and $\varepsilon >0$ there exists $v_\varepsilon $ such that $v_\varepsilon (-T_\epsilon )=0$, $v_\varepsilon (0)=\phi $ and 
	$$I_{-T_\varepsilon } (v_\varepsilon )\leq I_{-\infty} (u) + \varepsilon .$$ Notice that for any $\varepsilon >0$
	$$ \inf\left\{I_{-T}(v):T>0, v\in C([-T,0];H^{-\alpha}), v(-T)=0, v(0)=\phi \right\} \leq I_{-T_\varepsilon } (v_\varepsilon ).$$ This implies $U^\alpha(\phi ) \leq I_{-\infty}(u) + \varepsilon$. Thus, by taking the infimum over $\varepsilon >0$ and $u \in \chi ^\phi $, we get $$U^\alpha(\phi) \leq  \inf\left\{I_{-\infty}(u): u \in \chi ^\phi \right\}.$$ Therefore, $U^\alpha(\phi) =  \inf\left\{I_{-\infty}(u): u \in \chi ^\phi \right\}$.

	Finally, by Proposition \ref{level infty}, the level sets of $I_{-\infty}$ are compact. Hence, the infimum is in fact a minimum.
	$$U^\alpha(\phi) =  \min\left\{I_{-\infty}(u): u \in \chi ^\phi \right\}.$$
\end{proof}

We conclude this section by proving that $U$ has compact level sets.
\begin{theorem}\label{U compactness}
	Suppose $0 \leq \alpha < 1/2$, then for any $r>0$, the set $\Phi^\alpha (r) := \left\{\phi \in H : U^\alpha(\phi )\leq r\right\}$ is compact in H. In particular, $U^\alpha$ is lower semi-continuous in $H$.
\end{theorem}
\begin{proof}
	Let $\left\{\phi _n\right\}$ be a sequence in $\Phi^\alpha (r)$. By Theorem \ref{characterization of U}, for any $n \in \mathbb{N} $ there exists $u_n \in \chi ^{\phi _n}$ such that $$I_{-\infty}(u_n) = U^\alpha(\phi _n)\leq r.$$ Hence, 
	$$\left\{u_n\right\} \subset \left\{I_{-\infty}\leq r\right\}.$$ By the compactness of the level sets of $I_{-\infty}$ proved in Theorem \ref{level infty}, we can find a subsequence $\left\{u_{n_k}\right\} \subset \left\{u_n\right\}$ and $\bar{u} \in C((-\infty,0];H^{-\alpha})$ such that 
	$$\lim_{k \rightarrow \infty} u_{n_k}=\bar{u} \ \text{in} \ C((-\infty,0];H^{-\alpha}).$$ 
 In particular, $\lim_{k \rightarrow \infty} \phi _{n_k} = \lim_{k \rightarrow \infty}u_{n_k}(0)=\bar{u}(0)$ in $H^{-\alpha}$. Moreover, by Proposition \ref{u infty estimate}, there exists a positive constant $C$ such that $\|u_{n_k}(0)\Vert_{H^{1-\alpha}} \leq C$ for  all $k \in \mathbb{N}$. Thus, we can find a subsequence $\left\{u_{n^0_k}(0)\right\} \subset \left\{u_{n_k}(0)\right\}$ such that $u_{n_k^0}(0)$ is convergent to $u(0)$ weakly in $H^{1-\alpha}$. By the uniqueness of limit, we have $u(0) = \bar{u}(0)$. Since $H^{1-\alpha}$ is compact embedded into $H$, $\phi_{n_k^0} = u_{n_k^0}(0)$ is convergent to $\bar{u}(0)$ in $H$. Thanks to the lower semi-continuity of $I_{-\infty}$ proved in Proposition \ref{p:lsc}, we have
	$$I_{-\infty}(\bar{u} ) \leq \liminf_{k \rightarrow \infty} I_{-\infty}(u_{n_k}) \leq r.$$ Moreover, by Theorem \ref{characterization of U}, $U^\alpha\left(\bar{u} (0)\right) \leq I_{-\infty}(\bar{u} )$. In particular, $U^\alpha\left(\bar{u} (0)\right) \leq r$. Hence, $\bar{u} (0) \in \Phi^\alpha (r)$, the compactness of $\Phi^\alpha (r)$ in $H$ follows.
\end{proof}

\section{LDP for the invariant measures}\label{sec:LDP for invariant measure}
 In this section, we send $\epsilon $ and $\delta (\epsilon )$ to $0$ simultaneously and consider the large deviations principle for the invariant measures of Eq. (\ref{sbe2}). The main result of this section is the following theorem. 
 \begin{condition}\label{assumption 2}
     Assume that $Q_\epsilon $ has the form given in (\ref{Q epsilon }) for some $1/2 < \beta -\alpha < 1$, $0 \leq \alpha < 1/2$ and 
	$$\lim_{\epsilon  \rightarrow 0}\delta (\epsilon ) = 0, \  \lim_{\epsilon  \rightarrow 0}\epsilon \delta (\epsilon )^{-\frac{1+2\alpha}{2\beta} }= 0.$$  \end{condition}
\begin{theorem}\label{LDP IM}
	Under Assumption \ref{assumption 2}, the family of invariant measures $\left\{\nu _\epsilon \right\}_{\epsilon >0}$ of Eq. (\ref{sbe2}) satisfies a large deviations principle in $H$ with a good rate function. The rate function is the quasi-potential given in (\ref{quasi potential 1}).

\end{theorem}
 In Theorem \ref{U compactness}, we have known that $U$ has compact level sets. Thus, in order to prove the LDP, we only need to prove the lower bound and the upper bound in Definition \ref{LDP}.

 \subsection{Lower bound}
 In this section, we will prove the following lower bound result. The following proof is a standard argument used by many people, e.g. see Theorem 6.1 of \cite{brzezniak2017large} and Proposition 4.1 of \cite{cerrai2022large}. We include it here for completeness.
\begin{proposition}
	Under Assumption \ref{assumption 2}, the family of the invariant measures $\left\{\nu _\epsilon \right\}_{\epsilon >0}$ of Eq. (\ref{sbe2}) satisfies a large deviations principle lower bound in $H$ with the rate function $U$. That is, for any $\phi \in H$ and $\delta , \gamma  >0$, there exists $\epsilon _0>0$ such that 
$$\nu _\epsilon \left(B_H(\phi ,\delta )\right) \geq \exp\left(-\frac{U^\alpha(\phi )+\gamma }{\epsilon }\right), \ \epsilon \leq \epsilon _0.$$
\end{proposition}
\begin{proof}
	Fix $\phi  \in H$, $\delta >0$, $\gamma >0$ and $T>0$. We assume that $U^\alpha(\phi ) < \infty$ otherwise the result would be trivial. We set $\left\{v^x\right\}_{x \in H} \subset L^\infty ([0,T];H)$ to be a family of paths satisfying 
	$$\sup_{x \in H}\left\| v^x(T) -\phi \right\Vert _H <\delta /2.$$ Recalling the discussions in Section \ref{trace class}, we have that
	$$\sup_{\epsilon >0}\int_{H} \| x\Vert _H^2 \,d\nu_\epsilon  (x) < \infty . $$ Hence, by Chebyshev's inequality, there exists a constant $R$ being sufficiently large and independent of $\epsilon$ such that $$\nu _\epsilon \left(B_H(0,R)\right) \geq \frac{1}{2}, \ \text{for all}\ \epsilon >0.$$
	Moreover, as we discussed in Section \ref{trace class}, $\left\{\mathcal{L} (u_\epsilon^x)\right\}_{\epsilon > 0}$  satisfies the Freidlin-Wentzell uniform large deviations principle in $L^\infty([0,T];H)$. So, for every $r_0,\gamma >0$ there exists $\epsilon _0>0$ such that $\gamma /2 + \epsilon_0 \ln \frac{1}{2} >0$ and for any $v^x \in L^\infty([0,T];H)$ with $I_T^x(v^x) \leq r_0$,
	$$\inf _{x \in B_H(0,R)}\mathbb{P} \left(\left\| u_\epsilon ^x -v^x\right\Vert _{L^\infty([0,T];H)} < \delta /2 \right) \geq \inf _{x \in B_H(0,R)}\exp\left(-\frac{1}{\epsilon }\left[I_T^x(v^x) + \gamma /2 + \epsilon \ln \frac{1}{2}\right]\right)$$
	for every $\epsilon \leq \epsilon _0$. 
	Since the family  $\left\{\nu _\epsilon \right\}_{\epsilon >0}$ is the invariant measures of Eq. (\ref{sbe2}), we have 
	\begin{align*}
		\nu _\epsilon \left(B_H(\phi ,\delta )\right) &= \int_{H} \mathbb{P} \left(\| u_\epsilon ^x(T) -\phi \Vert _H <\delta \right) \,d\nu _\epsilon (x ) \\ &\geq \int_{H} \mathbb{P} \left(\left\| u_\epsilon ^x -v^x\right\Vert _{L^\infty([0,T];H)} < \delta /2 \right) \,d\nu _\epsilon (x )\\ &\geq \int_{B_H(0,R)} \mathbb{P} \left(\left\| u_\epsilon ^x -v^x\right\Vert _{L^\infty([0,T];H)} < \delta /2 \right) \,d\nu _\epsilon (x )\\ &\geq \nu _\epsilon \left(B_H(0,R)\right) \inf _{x \in B_H(0,R)}\mathbb{P} \left(\left\| u_\epsilon ^x -v^x\right\Vert _{L^\infty([0,T];H)} < \delta /2 \right) \\ &\geq \exp\left(\frac{\epsilon \ln\frac{1}{2}}{\epsilon }\right)\inf _{x \in B_H(0,R)}\exp\left(-\frac{1}{\epsilon }\left[I_T^x(v^x) + \gamma /2 + \epsilon \ln \frac{1}{2}\right]\right) \\ &= \inf _{x \in B_H(0,R)}\exp\left(-\frac{1}{\epsilon }\left[I_T^x(v^x) + \gamma /2 \right]\right).
	\end{align*}
	for every $\epsilon \leq \epsilon _0$. Therefore, to complete the proof, it suffices to find a sufficiently large $T$ such that for each $x \in B_H(0,R)$, there exists a path $v^x \in L^\infty([0,T];H)$ with $v^x(0) =x$ that satisfies $$I_T(v^x) \leq U^\alpha(\phi) + \gamma /2$$ and $$\| v^x(T) - \phi \Vert _H < \delta /2.$$ 
	For this, denote the solution of skeleton function Eq. (\ref{skt eq 2}) as $u^x_f$. By the definition of $U^\alpha(\phi )$, there exist $T_2 > 0$, $u_{\bar{f}}^0 \in C([0,T];H^{-\alpha})$ and $\bar{f} \in L^2([0,T];H)$ such that $u^0_{\bar{f}}(0)=0$, $u^0_{\bar{f}}(T_2)=\phi $ and $$\frac{1}{2}\int_{0}^{T_2} \| \bar{f}(t)\Vert_H^2  \,dt = I_{T_2}( u^0_{\bar{f}}) \leq U^\alpha(\phi ) + \gamma /2.$$ Meanwhile, by (\ref{u Gronwall 1}) when $f =0$, we have that $$ \| u^x_0(t)\Vert _H^2 \leq \| x\Vert ^2 \exp(-t).$$ Thus, for any $\lambda $, we can take $T_1= T_1(\lambda )$ large enough such that 
	$$ \sup_{x \in B_H(0,R)}\left\| u^x_0(T_1)\right\Vert _H \leq \lambda.$$ 
	We choose the paths $v^x=u^x_f$ with  $f \in L^2([0,T];H)$ defined by
	$$f(t):=\begin{cases}
		0, &t \in [0 ,T_1], \\ 
		\bar{f} (t-T_1), &t \in [T_1,T_1+T_2],
	\end{cases}$$
	and $T= T_1 + T_2$. By (\ref{u Gronwall 1 b}), $v^x \in L^\infty([0,T];H)$. Clearly, we have \begin{align*}
		I_T(v^x) &= \frac{1}{2}\int_{0}^{T} \| f(t)\Vert_H^2  \,dt =\frac{1}{2}\int_{0}^{T_2} \| \bar{f}(t)\Vert_H^2  \,dt  \leq U^\alpha(\phi ) + \gamma /2.
	\end{align*}
	Moreover, from the above construction and Proposition \ref{proposition x-y} we have that
	\begin{align*}
		\left\| u^x_f(T) - \phi \right\Vert _H &\leq \sup_{z \in B_H(0,\lambda ) }\left\| u^z_{\bar{f}}(T_2) - u^0_{\bar{f}}(T_2)\right\Vert_H  \\ &\leq \sup_{z \in B_H(0,\lambda )}\| z\Vert _H \exp\left(C\| z\Vert^2  + C\| \bar{f}\Vert _{L^2([0,T_2];H)}^2\right).
	\end{align*}
	Therefore, we can choose $\lambda $ small enough such that $$\| u^x_f(T) -\phi \Vert_H < \delta /2.$$ In conclusion, the paths $v^x(t) = u^x_f(t)$ satisfies the two desired conditions. The proof is completed.
\end{proof}

\subsection{Exponential estimates}\label{sec:exp esti}
In the proof of upper bound for large deviations principle we need to prove the exponential tightness for the invariant measure $\left\{\nu _\epsilon \right\}_{\epsilon >0}$ on a compact subspace of $H$. Here, we choose the compact subspace to be $H^{2\sigma }$, for some small $\sigma>0$. To achieve that, we need the following lemmas.


\begin{lemma}\label{exp estimate Z}
	Under Assumption \ref{assumption 2}, for some $0<\sigma <1/4 -\alpha/2$ and any $r>0$ there exists $\epsilon _r >0$ and $R_r >0$ such that for every $t \geq 0$
	$$ \mathbb{P} \left(\left\| Z_\epsilon (t)\right\Vert_{H^{2\sigma }} \geq R_r \right) \leq \exp\left(-\frac{r}{\epsilon }\right), \ \epsilon  \leq \epsilon _r.$$
\end{lemma}
\begin{proof}
	 Set a new operator $\tilde{Q}_\epsilon $ as $$\sqrt{\tilde{Q}_\epsilon}e_k:= (k\pi)^\alpha\tilde{\sigma}_{\delta (\epsilon ),k} e_k := \frac{(k\pi)^\alpha\sigma _{\delta (\epsilon) ,k}}{k \pi}e_k = \frac{\sigma _{\delta (\epsilon) ,k}}{(k \pi)^{1-\alpha -2\sigma }}l_k, \ k \in \mathbb{N}, $$ where $l_k := \frac{1}{(k \pi)^{2\sigma}}e_k$ is standard basis of the space $H^{2\sigma}$, and a martingale $\tilde{Z}_\epsilon $ as 
	$$\tilde{Z}_\epsilon (t) := \sqrt{\epsilon }\int_{0}^{t} \sqrt{\tilde{Q}_\epsilon} \,dW(s).$$ It follows from the Itô isometry that
	\begin{equation}
		\label{Zprime 1}
		\mathbb{E} \| \tilde{Z}_\epsilon(1)\Vert _{H^{2\sigma }}^2 = \epsilon \int_{0}^{1} \sum_{k \in \mathbb{N} }\left(\frac{\sigma _{\delta (\epsilon) ,k}}{(k \pi)^{1-\alpha-2\sigma }}\right)^2 \,ds = \epsilon \sum_{k \in \mathbb{N} }\frac{\sigma^2 _{\delta (\epsilon) ,k}}{(k \pi)^{2-2\alpha-4\sigma }} < \infty.
	\end{equation}
Again by the Itô isometry we have for all $t \geq 0$
\begin{align*}
	\mathbb{E} \| Z_\epsilon (t)\Vert _{H^{2\sigma }}^2 &\leq \epsilon \sum_{k \in \mathbb{N} }(k \pi)^{4\sigma }\int_{0}^{T} e^{-2(k \pi)^2s} (k\pi)^{2\alpha}\sigma _{\delta (\epsilon), k}^2\,ds  \\ &\leq \epsilon \sum_{k \in \mathbb{N}}\frac{\sigma^2 _{\delta (\epsilon) ,k}}{2(k \pi)^{2-2 \alpha-4\sigma }} < \infty 
\end{align*}
which implies 
$$
	\mathbb{E} \| Z_\epsilon (t)\Vert _{H^{2\sigma }}^2 \leq \mathbb{E} \| \tilde{Z}_\epsilon(1)\Vert _{H^{2\sigma }}^2  , \ t \geq 0, \ \epsilon >0.
$$
For every $x \in H^{2\sigma }$ and $\epsilon >0$, define 
$$f(x) = \left(1+ \| x\Vert^2 _{H^{2\sigma }}\right)^{1/2}$$
and
$$F_\epsilon(x) = \exp\left(\frac{f(x)}{\epsilon }\right).$$
By the Gaussian property of $Z_\epsilon (t)$ and $\tilde{Z}_\epsilon  (1)$ we can rewrite them as 
$$Z_\epsilon (t) = \sum_{k \in \mathbb{N} }\lambda _{\epsilon ,k} (t)\beta _k l_k, \ \tilde{Z}_\epsilon  (1) = \sum_{k \in \mathbb{N} }\tilde{\lambda} _{\epsilon, k} \beta _k l_k ,$$
where $\beta _k$ is standard Gaussian random variables and $\lambda _{\epsilon ,k} (t) \leq \tilde{\lambda} _{\epsilon, k} = \frac{\sigma _{\delta(\epsilon) ,k}}{(k \pi)^{1-\alpha-2\sigma}}$ for any $t \geq 0,\ k \in \mathbb{N}$. Since 
\begin{align*}
	\mathbb{E} \left[F_\epsilon(X)\right] &= \int_{0}^{+ \infty} \mathbb{P} (F_\epsilon(X) \geq x) \,dx \\ &= \exp\left(\frac{1}{\epsilon }\right)+ \int_{\exp\left(1/\epsilon \right)}^{+ \infty} \mathbb{P} \left(\| X\Vert _{H^{2\sigma}} \geq \sqrt{(\epsilon\ln x)^2-1}  \right) \,dx
\end{align*}
and $$\mathbb{P} \left(\| Z_\epsilon (t)\Vert _{H^{2\sigma}} \geq C \right) \leq \mathbb{P} \left(\| \tilde{Z}_\epsilon  (1)\Vert _{H^{2\sigma}} \geq C \right)$$ for all constant $C$ and $t \geq 0$ due to $\lambda _{\epsilon ,k} (t) \leq \tilde{\lambda} _{\epsilon, k}$, it follows that
$$\mathbb{E} \left[F_\epsilon \left(Z_\epsilon (t)\right)\right] \leq \mathbb{E} \left[F_\epsilon \left(\tilde{Z}_\epsilon(1)\right)\right], \ t \geq 0, \ \epsilon >0.$$ Moreover, by taking the derivative in $H^{2\sigma} $ space, we have that
$$\left| D_{H^{2\sigma} }f(x)\right\vert_{L(H^{2\sigma })}  \leq 1 , \ \left| D_{H^{2\sigma} }^2f(x) \right\vert_{L(H^{2\sigma }\otimes H^{2\sigma } ; \mathbb{R} )}  \leq 2$$ and $$D_{H^{2\sigma} }^2F_\epsilon (x) = \frac{1}{\epsilon ^2}F_\epsilon (x)D_{H^{2\sigma} }f(x) \otimes  D_{H^{2\sigma} }f(x) + \frac{1}{\epsilon }F_\epsilon (x)D_{H^{2\sigma} }^2f(x)I .$$
By Itô's formula, we get 
\begin{align*}
	\mathbb{E} \left[F_\epsilon(\tilde{Z}_\epsilon(t))\right] &= \exp(\epsilon ^{-1}) + \mathbb{E} \int_{0}^{t} \left[\frac{\epsilon }{2}\sum_{k \in \mathbb{N}}D^2_{H^{2\sigma} } F_\epsilon (\tilde{Z}_\epsilon(s))\left( \frac{\sigma _{\epsilon ,k}}{(k \pi)^{1-\alpha -2\sigma }}l_k,\frac{\sigma _{\epsilon ,k}}{(k \pi)^{1-\alpha-2\sigma }}l_k\right)\right] \,ds  \\ &\leq \exp(\epsilon ^{-1}) + \frac{\epsilon }{2}\int_{0}^{t} \mathbb{E}\left[ F_\epsilon (\tilde{Z}_\epsilon(s)) \left(\frac{1}{\epsilon ^2}+ \frac{2}{\epsilon }\right) \sum_{k \in \mathbb{N}}\frac{\sigma^2 _{\epsilon ,k}}{(k \pi)^{2-2\alpha -4 \sigma }}\right]\,ds \\ & \leq \exp(\epsilon ^{-1}) + C_Q^\alpha \left(\frac{1}{2\epsilon }+ 1\right) \int_{0}^{t} \mathbb{E} \left[F_\epsilon (\tilde{Z}_\epsilon(s))\right] \,ds,
\end{align*}
where $$C_Q^\alpha  := \sup_{\epsilon > 0} \sum_{k \in \mathbb{N}}\frac{\sigma^2 _{\delta (\epsilon) ,k}}{(k \pi)^{2-2\alpha -4\sigma }} < \infty.$$
It follows from the Gronwall inequality that
$$\mathbb{E} \left[F_\epsilon(\tilde{Z}_\epsilon(1))\right] \leq \exp(\epsilon ^{-1})\exp\left(\frac{C_Q^\alpha (1/2 +\epsilon )}{\epsilon }\right) = \exp \left(\frac{C_Q^\alpha (1/2 +\epsilon ) + 1}{\epsilon }\right).
$$
Thus, for any $r>0$ and $\epsilon _r >0$ such that for any $\epsilon \leq \epsilon _r$
\begin{align*}
	\mathbb{P} \left(\left\| Z_\epsilon (t)\right\Vert_{H^{2\sigma }} \geq R_r \right) &=\mathbb{P} \left(\exp\left(\frac{\left\| Z_\epsilon (t)\right\Vert_{H^{2\sigma }}}{\epsilon }\right) \geq \exp\left(\frac{R_r}{\epsilon }\right) \right) \\ &\leq \exp\left(-\frac{R_r}{\epsilon }\right)\mathbb{E} \left[\exp\left(\frac{\left\| Z_\epsilon (t)\right\Vert_{H^{2\sigma}}}{\epsilon }\right)\right] \\ & \leq \exp\left(-\frac{R_r}{\epsilon }\right)\mathbb{E} \left[F_\epsilon \left(Z_\epsilon (t)\right)\right]\\ &\leq \exp\left(-\frac{R_r}{\epsilon }\right)\mathbb{E} \left[F_\epsilon \left(\tilde{Z}_\epsilon (1)\right)\right]\\ &\leq \exp \left(-\frac{R_r-C_Q^\alpha (1/2 +\epsilon ) - 1}{\epsilon }\right) \leq \exp\left(-\frac{r}{\epsilon }\right). 
\end{align*}
Here we only need to take $R_r = r + C_Q^\alpha (1/2 +\epsilon ) +1$.
\end{proof}

\begin{lemma}
	\label{exponential estimate 1}
	Under Assumption \ref{assumption 2}, for any $r> 0$ there exist $\epsilon _r >0$ and $R_r >0$ such that for every $T \geq 0$
	$$ \mathbb{P} \left(\left  \| Z_\epsilon \right\Vert_{L^4([T,T+1];L^4)} \geq R_r \right) \leq \exp\left(-\frac{r}{\epsilon }\right), \ \epsilon  \leq \epsilon _r.$$
\end{lemma}
\begin{proof}
First, we want to prove that $\{\mathcal{L}(Z_\epsilon^x)\}_{\epsilon>0}$ satisfies FWULDP in $L^4([0,1]; L^4)$, uniformly with respect to $x \in H$. To achieve that, by Theorem 2.10 and Theorem 2.13 of \cite{salins2019equivalences}, it is enough to show for any $N >0$ and $\delta >0$
	$$\lim_{\epsilon  \rightarrow 0}\sup_{x \in H}\sup_{f \in \mathcal{P} _2^N}\mathbb{P} \left(\|Z_{\epsilon ,f}^x - Z_f^x\Vert _{L^4([0,1];L^4)} > \delta \right) =0.$$
 Notice that 
  $$Z_{\epsilon,f}^x(t) - Z_{f}^x(t) = \sqrt{\epsilon }\int_{0}^{t} S(t-s)\sqrt{Q_\epsilon} \,dW(s) + \int_{0}^{t} S(t-s)\left[\sqrt{Q_\epsilon}f(s) - \sqrt{Q_0}f(s)\right] \,ds.$$ Then, it follows from the same discussions as in Lemma \ref{LDPZ} that,
 $$\lim_{\epsilon  \rightarrow 0}\sup_{x \in H}\sup_{f \in \mathcal{P} _2^N}\mathbb{P} \left(\|Z_{\epsilon ,f}^x - Z_f^x\Vert _{L^4([0,1];L^4)} > \delta \right) \leq \lim_{\epsilon  \rightarrow 0}\sup_{x \in H}\sup_{f \in \mathcal{P} _2^N}\mathbb{P} \left(\|Z_{\epsilon ,f}^x - Z_f^x\Vert _{C([0,1];L^4)} > \delta \right) =0.$$ 
 Therefore, $\{\mathcal{L}(Z_\epsilon^x)\}_{\epsilon>0}$ satisfies FWULDP in $L^4([0,1];L^4)$ with good rate function $J^x$ uniform in $H$, where
 $$J^x(z) = \inf \left\{\frac{1}{2}\int_{0}^{1}\|f(t)\Vert_H^2 \,dt : f \in L^2([0,1];H), z = Z^x_f \right\}$$ Define $\Phi^x(r):= \{z \in L^4([0,1];L^4): J^x(z)\leq r \}$, by discussions in Proposition \ref{compact1}, 
 \begin{align*}
     \sup_{z \in \Phi^x(r)}\|z\Vert_{L^4([0,1];L^4)} &\leq \left(\int_{0}^{1}\|S(t)x\Vert_{L^4}^4 \,dt \right)^{1/4} + \sup_{f \in \mathcal{P}_2^{2r}}\sup_{t \in [0,1]}\left\|\int_{0}^{t}S(t-s)\sqrt{Q_0}f(s) \,ds \right\Vert_{L^4} \\
     &\leq \left(\int_{0}^{1}\|S(t)x\Vert_{H^{1/4}}^4 \,dt\right)^{1/4} + C\sqrt{r} \\
     &=   \left(\int_{0}^{1}\|A^{1/8}S(t)x\Vert_{H}^4 \,dt\right)^{1/4} + C\sqrt{r}\\
     &\leq  \left(\int_{0}^{1}\frac{1}{t^{1/2}}\|x\Vert_{H}^4 \,dt\right)^{1/4} + C\sqrt{r} \leq C(\|x\Vert_H + \sqrt{r})
     \end{align*}
Then, as a consequence of FWULDP, using similar arguments in Theorem \ref{uniform contraction principle}, for any $\delta>0$ and $\gamma>0$, there exits $\epsilon_0 >0$ such that for all $x \in H$,  
$$\mathbb{P}\left(\|Z_\epsilon^x\Vert_{L^4([0,1];L^4)} \geq C(\|x\Vert_{H} + \sqrt{r+\gamma})+ \delta \right) \leq \exp{\left(-\frac{r}{\epsilon} \right)}.$$ For the convenience of computations, we choose $\delta=\gamma=1$. Therefore, for every $T \geq 1$,
$$\mathbb{P}\left(\|Z_\epsilon\Vert_{L^4([T,T+1];L^4)} \geq C\left(\|Z_\epsilon(T)\Vert_{H} + \sqrt{r+1}+ 1 \right) \right) \leq \exp{\left(-\frac{r}{\epsilon} \right)}.$$ Moreover, notice that $$\|Z_\epsilon(T)\Vert_H < \bar{R}_r \  \text{and} \  \|Z_\epsilon\Vert_{L^4([T,T+1];L^4)} < C\left(\|Z_\epsilon(T)\Vert_{H} + \sqrt{r+1}+ 1 \right) $$
implies
$$\|Z_\epsilon\Vert_{L^4([T,T+1];L^4)} < C\left( \bar{R}_r + \sqrt{r+1}+ 1 \right).$$
  It then follows that \begin{align*}
      \mathbb{P}\left(\|Z_\epsilon\Vert_{L^4([T,T+1];L^4)} \geq C\left(\bar{R}_r + \sqrt{r+1}+ 1 \right) \right) & \leq \mathbb{P}\left(\|Z_\epsilon\Vert_{L^4([T,T+1];L^4)} \geq C\left(\|Z_\epsilon(T)\Vert_{H} + \sqrt{r+1}+ 1 \right) \right)\\  &\ \ \ \ + \mathbb{P}\left(\|Z_\epsilon(T)\Vert_H \geq \bar{R}_r \right).      
      \end{align*}
      Thanks to Lemma \ref{exp estimate Z}, there exist $\epsilon_r^1$ and $\bar{R_r}$ such that for every $T\geq 0$,
      $$\mathbb{P}\left(\|Z_\epsilon(T)\Vert_H \geq \bar{R}_r \right) \leq \exp{\left(-\frac{r}{\epsilon} \right)}, \ \epsilon \leq \epsilon_r^1.$$ We set $\epsilon_r = \epsilon_0 \wedge \epsilon_r^1$, the desired result follows.

\end{proof}

\begin{lemma}\label{exp esti u H}
	Under Assumption \ref{assumption 2}, for any $r>0$ there exists $\epsilon _r >0$ and $R_r >0$ such that 
	$$ \lim_{T\rightarrow\infty}\frac{1}{T}\int_{0}^{T}  \mathbb{P} \left(\left\| u_\epsilon^0 (t)\right\Vert_H \geq R_r \right) \,dt \leq \exp\left(-\frac{r}{\epsilon }\right), \ \epsilon  \leq \epsilon _r.$$
\end{lemma}
\begin{proof}
	By the ergodicity of $\nu _\epsilon $ we know the limit exists. Define a functional $F_\epsilon: \mathbb{R} \times H \rightarrow \mathbb{R} $  $$F_\epsilon (t,x) := \exp\left(t + \frac{\| x\Vert _H^2}{2\epsilon }\right).$$ Then we have 
	$$D_tF_\epsilon (t,x) = F_\epsilon (t,x), \ D_xF_\epsilon (t,x) = \frac{1}{\epsilon }F_\epsilon (t,x)x$$ and 
	$$D_x^2F_\epsilon (t,x) = \frac{1}{\epsilon ^2}F_\epsilon (t,x)x\otimes  x + \frac{1}{\epsilon }F_\epsilon (t,x)I .$$

	As a consequence of Itô's formula, we have \begin{align*}
		&\ \ \ \ \mathbb{E} F_\epsilon (t,u_\epsilon ^y(t)) \\&= F_\epsilon (0,y) + \mathbb{E} \int_{0}^{t} \bigg[ D_t F_\epsilon (s,u_\epsilon ^y(s)) + \left\langle D_xF_\epsilon (s,u_\epsilon ^y(s)), -Au_\epsilon ^y(s) + D\left(u_\epsilon ^y(s)^2\right)\right\rangle_H\\ 
  &\ \ \ \ +\frac{\epsilon }{2}\sum_{k \in \mathbb{N} }\left\langle D_x^2F_\epsilon (s,u_\epsilon ^y(s)) A^{\alpha}Q_{\delta (\epsilon )}e_k, e_k \right\rangle _H \bigg] \,ds \\ 
  & = F_\epsilon (0,y) + \mathbb{E} \int_{0}^{t} F_\epsilon (s,u_\epsilon ^y(s)) \bigg[1 - \frac{1}{\epsilon }\| u_\epsilon ^y(s)\Vert _{H^1}^2 + \frac{\epsilon }{2}\sum_{k \in \mathbb{N} }\bigg(\frac{1}{\epsilon ^2}\left| \langle u_\epsilon ^y(s), (k\pi)^\alpha \sigma _{\delta (\epsilon ),k}e_k\rangle _H \right\vert ^2 \\ 
  &\ \ \ \ + \frac{1}{\epsilon }\langle (k\pi)^{2\alpha}\sigma _{\delta (\epsilon ),k}^2e_k, e_k \rangle _H\bigg)\bigg] \,ds \\ 
  &\leq F_\epsilon (0,y) + \mathbb{E} \int_{0}^{t} F_\epsilon (s,u_\epsilon ^y(s)) \left[1 - \frac{1}{\epsilon }\| u_\epsilon ^y(s)\Vert _{H^1}^2 + \frac{1}{2\epsilon }\| u_\epsilon^y(s)\Vert _{H^\alpha}^2 + \frac{1}{2 }\sum_{k \in \mathbb{N} }(k\pi)^{2\alpha}\sigma _{\delta (\epsilon ),k}^2\right] \,ds \\ 
  &\leq F_\epsilon (0,y) + \mathbb{E} \int_{0}^{t} F_\epsilon (s,u_\epsilon ^y(s)) \left(1 - \frac{1}{2\epsilon }\| u_\epsilon ^y(s)\Vert _H^2 +  \frac{1}{2 }C_\beta \delta (\epsilon )^{-\frac{1+2\alpha}{2\beta}}\right) \,ds,
\end{align*}
	where in the second equality the nonlinear term was disposed due to the boundary condition; and in the last inequality $\sum_{k \in \mathbb{N} }(k\pi)^{2\alpha}\sigma _{\delta (\epsilon ),k}^2 \leq C_\beta \delta (\epsilon )^{-\frac{1+2\alpha}{2\beta}}$ and the Poincar\'e inequality were used. Hence, by the fact that $e^x(a-x) \leq e^{a-1}$ for every $a>1$ and $x \geq 0$, we have 
	\begin{align*}
		\mathbb{E} F_\epsilon (t,u_\epsilon ^y(t)) &\leq \exp\left(\frac{\| y\Vert _H^2}{2\epsilon }\right) + \mathbb{E} \int_{0}^{t} \exp(s) \exp\left(\frac{\| u_\epsilon ^y(s)\Vert _H^2}{2\epsilon }\right)\left(1 - \frac{1}{2\epsilon }\| u_\epsilon ^y(s)\Vert _H^2 +  \frac{1}{2 }C_\beta \delta (\epsilon )^{-\frac{1+2\alpha}{2\beta}}\right) \,ds \\ &\leq \exp\left(\frac{\| y\Vert _H^2}{2\epsilon }\right) + \mathbb{E} \int_{0}^{t} \exp(s) \exp\left( \frac{1}{2 }C_\beta \delta (\epsilon )^{-\frac{1+2\alpha}{2\beta}}\right) \,ds,
	\end{align*}
	which implies
	$$\mathbb{E}\exp\left(\frac{\| u_\epsilon ^0(t)\Vert _H^2}{2\epsilon }\right) \leq \exp(-t) + \exp\left( \frac{1}{2 }C_\beta \delta (\epsilon )^{-\frac{1+2\alpha}{2\beta}}\right).$$
	Since $\lim_{\epsilon \rightarrow 0}\epsilon \delta (\epsilon )^{-\frac{1+2\alpha}{2\beta}} =0$, then for any $r>0$ there exists $R_r>0$ and $\epsilon _r >0$ such that
	\begin{align*}
		\lim_{T\rightarrow\infty}\frac{1}{T}\int_{0}^{T}  \mathbb{P} \left(\left\| u_\epsilon^0 (t)\right\Vert_H \geq R_r \right) \,dt &\leq \exp\left(-\frac{R_r^2}{2\epsilon }\right)\frac{1}{T}\limsup_{T \rightarrow \infty}\int_{0}^{T} \mathbb{E}\exp\left(\frac{\| u_\epsilon ^0(t)\Vert _H^2}{2\epsilon }\right) \,dt \\ &\leq \exp\left(-\frac{R_r^2/2 - \frac{1}{2 }\epsilon C_\beta \delta (\epsilon )^{-\frac{1+2\alpha}{2\beta}}}{\epsilon }\right) \\ &\leq \exp(-\frac{r}{\epsilon }), \ \epsilon \leq \epsilon _r.
	\end{align*}
\end{proof}

\begin{lemma}\label{exponential estimate 2}
	Under Assumption \ref{assumption 2}, for some $0<\sigma <1/4 -\alpha/2$ and any $r>0$ there exist $\epsilon _r >0$ and $R_r >0$ such that 
	$$ \nu _\epsilon  \left(B_{H^{2\sigma} }^c(0,R_r) \right) \leq \exp\left(-\frac{r}{\epsilon }\right), \ \epsilon  \leq \epsilon _r.$$
\end{lemma}
\begin{proof}
	By the definition of mild solution, for arbitrary $ t \geq 0$
	\begin{align*}
		u_\epsilon ^0(t+1) &= \int_{0}^{t+1} S(t+1-s)D\left(u_\epsilon ^0(s)^2\right) \,ds + \sqrt{\epsilon }\int_{0}^{t+1} S(t+1-s)\sqrt{Q_{\epsilon} }\,dW(s) \\  &= S(1)u_\epsilon ^0(t) + \int_{t}^{t+1} S(t+1-s)D\left(u_\epsilon ^0(s)^2\right) \,ds + \sqrt{\epsilon }\int_{t}^{t+1} S(t+1-s)\sqrt{Q_{\epsilon } }\,dW(s).
	\end{align*}
	Thus, by Lemma 14.2.1 of \cite{da1996ergodicity}, for $2/p + \sigma < 1/4$ and positive constants $M_2$, $M_3$, $\bar{C}$
	\begin{align*}
		&\ \ \ \ \left\| A^\sigma  \int_{t}^{t+1} S(t+1-s)D\left(u_\epsilon ^0(s)^2\right) \,ds \right\Vert _H \\& \leq \bar{C}  \int_{t}^{t+1} (t+1-s)^{-(3/4 + \sigma )}\| u_\epsilon ^0(s)^2\Vert _{L^1} \,ds \\& \leq  \bar{C}\int_{t}^{t+1} (t+1-s)^{-(3/4 + \sigma) }\| 2Y_\epsilon ^0(s)^2 + 2Z_\epsilon (s)^2\Vert _{L^1} \,ds \\ & \leq 2\bar{C}\int_{t}^{t+1} (t+1-s)^{-(3/4 + \sigma) }\left(\| Y_\epsilon ^0(s)\Vert_H^2  + \| Z_\epsilon (s)\Vert _H^2 \right) \,ds \\ & \leq 2\bar{C}\int_{t}^{t+1} (t+1-s)^{-(3/4 + \sigma) } \,ds \sup_{0\leq s \leq 1}\|Y_\epsilon ^0(t+s)\Vert_H^2 \\&\ \ + 2\bar{C}\left[\int_{t}^{t+1} (t+1-s)^{-\frac{p}{p-2}(3/4 + \sigma) } \,ds \right]^{\frac{p-2}{p}} \left(\int_{t}^{t+1} \|Z_\epsilon (s)\Vert_H^p \,ds \right)^{2/p}\\ &\leq M_2 \sup_{0 \leq s \leq 1}\| Y(t+s)\Vert _H^2 + M_3 \left(\int_{t}^{t+1} \| Z_\epsilon (s)\Vert _H^p \,ds \right)^{2/p}.
	\end{align*}
	Therefore,
	\begin{align*}
		\left\| A^\sigma u_\epsilon ^0(t+1)\right\Vert _H & \leq \left\| A^\sigma S(1) u_\epsilon ^0(t)\right\Vert _H + \left\| A^\sigma  \int_{t}^{t+1} S(t+1-s)D\left(u_\epsilon ^0(s)^2\right) \,ds \right\Vert _H \\&\ \ \ \ + \left\| A^\sigma \sqrt{\epsilon }\int_{t}^{t+1} S(t+1-s)\sqrt{Q_{\epsilon } }\,dW(s) \right\Vert _H \\ &\leq M_1 \| u_\epsilon ^0(t)\Vert _H + M_2 \sup_{0 \leq s \leq 1}\| Y(t+s)\Vert _H^2 + M_3 \left(\int_{t}^{t+1} \| Z_\epsilon (s)\Vert _H^p \,ds \right)^{2/p} \\ &\ \ \ \ + \left\| Z_\epsilon (t+1) - S(1)Z_\epsilon (t)\right\Vert _{H^{2\sigma} },
	\end{align*}
	where $M_1$ is some positive constant.
	Then, we get 
	\begin{align*}
		\mathbb{P} \left(\left\| A^\sigma u_\epsilon ^0(t+1)\right\Vert _H \geq R_r\right) &\leq \mathbb{P} \left(M_1\| u_\epsilon ^0(t)\Vert _H \geq R_r/4\right) + \mathbb{P} \left(M_2\sup_{0 \leq s \leq 1}\| Y_\epsilon ^0(t+s)\Vert _H^2 \geq R_r/4\right) \\ &\ \ \ \ + \mathbb{P} \left(\|Z_\epsilon (t+1)\Vert _{H^{2\sigma }} \geq R_r/8 \right) + \mathbb{P} \left(\|Z_\epsilon (t)\Vert _{H^{2\sigma }} \geq R_r/8 \right)\\ &\ \ \ \ + \mathbb{P} \left(M_3 \left(\int_{t}^{t+1} \| Z_\epsilon (s)\Vert _H^p \,ds \right)^{2/p} \geq R_r/4 \right) \\ &=: I_1(t) + I_2(t) + I_3(t) + I_4(t) + I_5(t).
	\end{align*}
	According to Lemma \ref{exp esti u H}, $$\lim_{T \rightarrow \infty}\frac{1}{T}\int_{0}^{T} I_1(t) \,dt \leq \exp\left(-\frac{r}{\epsilon }\right).$$
	Moreover, by (\ref{Y estimate 1}), there exists $C>0$ such that
	\begin{align*}
		\sup_{0\leq s \leq 1} \| Y_\epsilon ^0(t+s)\Vert _H^2 &\leq \| Y_\epsilon^0 (t)\Vert _H^2\exp\left(2C\int_{t}^{t+1} \| Z_\epsilon (s) \Vert _{L^4}^{8/3} \,ds \right)\\ &\ \ \ \ + \frac{1}{2} \exp\left(2C\int_{t}^{t+1} \| Z_\epsilon (s) \Vert _{L^4}^{8/3} \,ds \right) \int_{t}^{t+1} \| Z_\epsilon (s) \Vert _{L^4}^4 \,ds.
	\end{align*}
	Then, 
	\begin{align*}
		I_2(t) &\leq 2\mathbb{P} \left(M_2\exp\left(2C\int_{t}^{t+1} \| Z_\epsilon (s) \Vert _{L^4}^{8/3} \,ds\right) \geq \sqrt{R_r/8}\right) + \mathbb{P} \left(\| Y_\epsilon ^0(t)\Vert _H^2 \geq \sqrt{R_r/8}\right) \\  &\ \ \ \ + \mathbb{P} \left(\frac{1}{2}\int_{t}^{t+1} \| Z_\epsilon (s) \Vert _{L^4}^4 \,ds \geq \sqrt{R_r/8}\right) \\ &=: 2I_{21}(t)+ I_{22}(t) + I_{23}(t).
	\end{align*}
	Since $Y_\epsilon ^0(t) = u_\epsilon ^0(t) - Z_\epsilon (t)$, we have 
\begin{align*}
	I_{22}(t) &= \mathbb{P} \left(\| Y_\epsilon ^0(t)\Vert _H^2 \geq \sqrt{R_r/8}\right) \\ &= \mathbb{P} \left(\| Y_\epsilon ^0(t)\Vert _H \geq (R_r/8)^{1/4}\right) \\ &\leq \mathbb{P} \left(\| u_\epsilon ^0(t)\Vert _H \geq \frac{(R_r/8)^{1/4}}{2}\right) + \mathbb{P} \left(\| Z_\epsilon (t)\Vert _H \geq \frac{(R_r/8)^{1/4}}{2}\right).
\end{align*}
According to Lemma \ref{exp estimate Z}, 
\begin{align*}
	\mathbb{P} \left(\| Z_\epsilon (t)\Vert _H \geq \frac{(R_r/8)^{1/4}}{2}\right) &\leq \mathbb{P} \left(\| Z_\epsilon (t)\Vert _{H^{2\sigma }} \geq \frac{(R_r/8)^{1/4}}{2}\right) \\ &\leq \exp\left(-\frac{r}{\epsilon }\right).
\end{align*}
Therefore,
$$\lim_{T \rightarrow \infty}\frac{1}{T}\int_{0}^{T} I_{22}(t) \,dt \leq \exp\left(-\frac{r}{\epsilon }\right).$$
Similarly by Lemma \ref{exp estimate Z}, 
$$I_3(t), I_4(t) \leq \exp\left(-\frac{r}{\epsilon }\right), \ t \geq 0.$$

Next, we will estimate $I_{21}$ and $I_{23}$. Since trivially $1+\left\| Z_\epsilon (s)\right\Vert _{L^4}^4 \geq \left\| Z_\epsilon (s)\right\Vert _{L^4}^{8/3}$ for any $s \geq 0$, we have
\begin{align*}
	I_{21}(t) &= \mathbb{P} \left(M_2\exp\left(2C\int_{t}^{t+1} \| Z_\epsilon (s) \Vert _{L^4}^{8/3} \,ds\right) \geq \sqrt{R_r/8}\right)\\ &= \mathbb{P} \left(\int_{t}^{t+1} \| Z_\epsilon (s) \Vert _{L^4}^{8/3} \,ds \geq \frac{\frac{1}{2}\ln R_r - \frac{1}{2}\ln 8 - \ln M_2}{2C}\right) \\ &\leq \mathbb{P} \left(\int_{t}^{t+1} \| Z_\epsilon (s) \Vert _{L^4}^4 \,ds \geq \frac{\frac{1}{2}\ln R_r - \frac{1}{2}\ln 8 - \ln M_2}{2C} -1\right).
\end{align*}
By Lemma \ref{exponential estimate 1}, there exist $\epsilon_r$ and $R_r^*>0$ such that for every $t\geq 0$
\begin{align*}
	\mathbb{P} \left(\int_{t}^{t+1} \| Z_\epsilon (s) \Vert _{L^4}^4 \,ds \geq {R_r^*}^4 \right)  &\leq \exp \left(-\frac{r}{\epsilon }\right), \ \epsilon \leq \epsilon_r.
\end{align*}
Thus, there exists $R_r>0$ such that for any $\epsilon \leq \epsilon_r$
$$I_{21}(t), I_{23}(t) \leq \exp\left(-\frac{r}{\epsilon }\right), \ t \geq 0.$$

Next, we need to show that $I_5(t) \leq \exp\left(-\frac{r}{\epsilon }\right)$. Since $\sigma < 1/4$, we can choose $p >8$ such that $2/p + \sigma < 1/4$. Then \begin{align*}
	I_5(t)= \mathbb{P} \left(\int_{t}^{t+1} \| Z_\epsilon (s)\Vert _H^p \,ds  \geq \left(\frac{R_r}{4M_3}\right)^{p/2} \right).
\end{align*}
Since the function 
$$ [0, +\infty )\ni R \longmapsto h(R) := \exp\left(\frac{(1+R)^{1/p}}{\epsilon }\right) \in [1, +\infty), $$
is convex when $\epsilon \leq 1/(p-1)$, and increasing, then for any $R_r >0$ we have by Jensen's inequality
\begin{align*}
	I_5(t) &= \mathbb{P} \left(h\left(\int_{t}^{t+1} \| Z_\epsilon (s)\Vert _H^p \,ds \right) \geq h\left(\left(\frac{R_r}{4M_3}\right)^{p/2}\right) \right) \\ &\leq \mathbb{P} \left(\int_{t}^{t+1} \exp\left(\frac{\left(1 + \| Z_\epsilon (s)\Vert _H^p\right)^{1/p}}{\epsilon }\right) \,ds \geq \exp\left(\frac{\left(1+ \left(\frac{R_r}{4M_3}\right)^{p/2}\right)^{1/p}}{\epsilon }\right)\right)\\ &\leq \exp\left(-\frac{\left(1+ \left(\frac{R_r}{4M_3}\right)^{p/2}\right)^{1/p}}{\epsilon }\right)\int_{t}^{t+1} \mathbb{E} F_\epsilon ^p\left(Z_\epsilon (s)\right) \,ds ,
\end{align*}
where $F_\epsilon ^p(x) = \exp\left(\frac{f^p(x)}{\epsilon }\right)$ and $f^p(x) = \left(1+ \| x\Vert _H^p\right)^{1/p}$. The function $f: H \rightarrow \mathbb{R} $ is twice differentiable and 
$$D_xf^p(x) = \frac{\| x\Vert_H ^{p-2}x}{\left(1+\| x\Vert_H^p \right)^{\frac{p-1}{p}}}, \ D_x^2f^p(x) = \frac{(p-2)\| x\Vert_H ^{p-4}x \otimes x + \| x\Vert _H^{p-2}I}{\left(1+\| x\Vert_H^p \right)^{\frac{p-1}{p}}} + \frac{(1-p)\| x\Vert_H ^{2p-4}x \otimes x}{\left(1+\| x\Vert_H^p \right)^{\frac{2p-1}{p}}}.$$
Thus, $$\left| D_xf^p(x)\right\vert _{L(H)} \leq 1 , \  \left| D_x^2f^p(x) \right\vert _{L(H \otimes H ;\mathbb{R} )} \leq 2p-2.$$ Furthermore, we also have that $$D_xF_\epsilon ^p = \frac{1}{\epsilon }F_\epsilon ^p(x)D_xf^p(x), \ D_x^2 F_\epsilon ^p(x) = \frac{1}{\epsilon ^2}F_\epsilon ^p(x)D_xf^p(x) \otimes  D_xf^p(x) + \frac{1}{\epsilon }F_\epsilon ^p(x)D_x^2f^p(x)I .$$
Define $F_\epsilon ^p(t,x) = \exp(t)F_\epsilon ^p(x)$. Note that $\sum_{k \in \mathbb{N} }\langle (z\otimes z)A^{\alpha}Q_{\delta (\epsilon )}e_k,e_k\rangle \leq \| z\Vert _{H^\alpha}^2$. By Itô's formula and integration by parts, it follows that
\begin{align*}
	\mathbb{E} F_\epsilon ^p\left(t, Z_\epsilon (t)\right) &= F_\epsilon^p (0,Z_\epsilon(0)) + \mathbb{E} \int_{0}^{t} \Bigg[D_t F_\epsilon^p (s,Z_\epsilon(s)) + \left\langle D_xF_\epsilon^p (s,Z_\epsilon (s)), \Delta Z_\epsilon (s) \right\rangle_H\\ &\ \ \ \ +\frac{\epsilon }{2}\sum_{k \in \mathbb{N} }\left\langle D_x^2F_\epsilon (s,Z_\epsilon (s))A^{\alpha}Q_{\delta (\epsilon )}e_k, e_k \right\rangle _H \Bigg] \,ds \\ &\leq \exp(\epsilon ^{-1}) + \mathbb{E} \int_{0}^{t} F_\epsilon ^p(s, Z_\epsilon (s))\left(\frac{\| Z_\epsilon (s)\Vert _H^{2p-4} \| Z_\epsilon (s)\Vert _{H^{\alpha}}^2 }{2\epsilon \left(1+ \| Z_\epsilon (s)\Vert _H^p\right)^{\frac{2p-2}{p}}} - \frac{\| Z_\epsilon (s)\Vert _H^{p-2}\| Z_\epsilon (s)\Vert _{H^1}^2}{\epsilon \left(1+\| Z_\epsilon (s)\Vert _H^p\right)^{\frac{p-1}{p}}} +1 \right)\\ &\ \ \ \ + F_\epsilon ^p(s,Z_\epsilon (s))(p-1)\sum_{k \in \mathbb{N} }(k\pi)^{2\alpha}\sigma _{\delta (\epsilon ),k}^2\,ds .
\end{align*}
Now using the Poincar\'e inequality and a trivial inequality $\frac{a^{p-2}}{(1+a^p)^{\frac{p-1}{p}}} \leq 1 $ for any $a \geq 0$, we have 
$$\frac{\| Z_\epsilon (s)\Vert _H^{2p-4} \| Z_\epsilon (s)\Vert _{H^{\alpha} }^2}{ \left(1+ \| Z_\epsilon (s)\Vert _H^p\right)^{\frac{2p-2}{p}}} \leq \frac{\| Z_\epsilon (s)\Vert _H^{p-2}\| Z_\epsilon (s)\Vert _{H^1}^2}{ \left(1+\| Z_\epsilon (s)\Vert _H^p\right)^{\frac{p-1}{p}}}.$$ Moreover, $\sum_{k \in \mathbb{N} }\sigma _{\delta (\epsilon ),k}^2 \leq C_\beta \delta (\epsilon )^{-\frac{1+2\alpha}{2\beta} },$ so we have 
$$\mathbb{E} F_\epsilon ^p\left(t, Z_\epsilon (t)\right) \leq \exp(\epsilon ^{-1}) + \mathbb{E} \int_{0}^{t} F_\epsilon ^p(s, Z_\epsilon (s))\left[(p-1)C_\beta \delta (\epsilon )^{-\frac{1+2\alpha}{2\beta} } - \frac{\| Z_\epsilon (s)\Vert _H^p}{2\epsilon\left(1+\| Z_\epsilon (s)\Vert _H^p\right)^{\frac{p-1}{p}} } + 1\right] \,ds.$$
By 
$$- \frac{\| Z_\epsilon (s)\Vert _H^p}{\left(1+\| Z_\epsilon (s)\Vert _H^p\right)^{\frac{p-1}{p}} } = \left(1+ \| Z_\epsilon \Vert _H^p\right)^{\frac{1-p}{p}} - \left(1+ \| Z_\epsilon \Vert _H^p\right)^{1/p} \leq 1 - f^p(Z_\epsilon (s)),$$ it follows that \begin{align*}
	\mathbb{E} F_\epsilon ^p\left(t, Z_\epsilon (t)\right) &\leq \exp(\epsilon ^{-1}) + \frac{1}{2}\mathbb{E} \int_{0}^{t} F_\epsilon ^p(s, Z_\epsilon (s))\left[-\frac{f^p(Z_\epsilon (s))}{\epsilon} + \frac{1}{\epsilon} + 2(p-1)C_\beta \delta (\epsilon )^{-\frac{1+2\alpha}{2\beta} } + 2\right] \,ds\\ &\leq \exp(\epsilon ^{-1}) + \frac{1}{2}\int_{0}^{t} \exp\left(\frac{1 + 2\epsilon (p-1)C_\beta \delta (\epsilon )^{-\frac{1+2\alpha}{2\beta} } + \epsilon }{\epsilon }\right)\exp(s) \,ds, 
\end{align*}
where the second inequality is due to the fact that $e^x(a-x) \leq \exp(a-1)$ for every $a>1$ and $x \geq 0$.
Thus,
$$\mathbb{E} F_\epsilon ^p\left(Z_\epsilon (t)\right) \leq \exp(-t + \epsilon ^{-1}) + \frac{1}{2}\exp\left(\frac{1 + 2\epsilon (p-1)C_\beta \delta (\epsilon )^{-\frac{1+2\alpha}{2\beta} } + \epsilon }{\epsilon }\right)$$ and 
\begin{align*}
    \int_{t}^{t+1} \mathbb{E} F_\epsilon ^p\left(Z_\epsilon (s)\right) \,ds &\leq \exp(\epsilon ^{-1}) + \frac{1}{2}\exp\left(\frac{1 + 2\epsilon (p-1)C_\beta \delta (\epsilon )^{-\frac{1+2\alpha}{2\beta} } + \epsilon }{\epsilon }\right) \\ &\leq \exp\left(\frac{1 + 2 (p-1)C_\beta \epsilon\delta (\epsilon )^{-\frac{1+2\alpha}{2\beta} } }{\epsilon }\right).
    \end{align*}

Therefore,$$I_5(t) \leq \exp\left(- \frac{r}{\epsilon }\right).$$

Finally, thanks to the ergodicity of $\nu _\epsilon $, we have \begin{align*}
	\nu _\epsilon  \left(B_{H^{2\sigma} }^c(0,R_r) \right) &= \lim_{T \rightarrow \infty } \frac{1}{T}\int_{0}^{T} \mathbb{P} \left(\left\| A^\sigma u_\epsilon ^0(t+1)\right\Vert _H \geq R_r\right) \,dt \\ &\leq \lim_{T \rightarrow \infty } \frac{1}{T}\int_{0}^{T} I_1(t) + I_{21}(t) + I_{22}(t) + I_{23}(t) + I_3(t) + I_4(t) + I_5(t)\,dt \\ &\leq 7\exp\left(-\frac{r}{\epsilon }\right),
\end{align*}
which implies the result.
\end{proof}

\subsection{Upper bound}
Let us recall that by Theorem \ref{U compactness}, the level set $\Phi^\alpha (r) = \left\{\phi  \in H: U^\alpha(\phi )\leq r\right\}$ is compact in H. To prove the upper bound for LDP, we need the following two key lemmas.
\begin{lemma}\label{upper bound lemma 1}
	Suppose $0 \leq \alpha< 1/2$, then for any $\delta >0$ and $r>0$, there exist $\lambda >0$ and $T>0$ such that for any $t \geq T$ and $z \in L^\infty([0,t];H)$,
	$$\| z(0)\Vert _H <\lambda , \ I_{t}(z) \leq r \Rightarrow dist_H \left(z(t), \Phi^\alpha (r)\right) <\delta.$$
\end{lemma}
\begin{proof}
Assume the claim does not hold. Then, there exist $\delta>0$ and $r>0$ such that for every $n \in \mathbb{N} $ there exists a function $z\in L^\infty([0,T_n];H)$ with $T_n \nearrow +\infty$, $\lambda_n := \|z_n(0)\Vert_H\searrow 0$ and
\begin{equation}\label{ineq 1 upper lemma 1}
    I_{T_n}(z_n) \leq r \ , dist_H(z_n(T_n), \Phi^\alpha(r)) \geq \delta.
\end{equation}
By the arguments in Proposition \ref{u infty estimate}, for any $\eta >0$, there exists $N \in \mathbb{N}$ such that for any $n \geq N$ we can find $t_n \geq 0$ with $T_n - t_n$ sufficiently large and satisfies $\|z_n(t_n)\Vert_{H^{1-\alpha}} < \eta$. According to arguments in Lemma \ref{lemma 2} and Lemma \ref{lemma 3}, this implies that for any $\varepsilon >0$, there exits $N \in \mathbb{N}$ such that for any $n \geq N$,
$$U^\alpha(z_n(T_n)) \leq I_{T_n}(z_n)+ \varepsilon \leq r + \varepsilon.$$
Since by Theorem \ref{U compactness}, level sets of $U^\alpha$ is compact in $H$, there would exists a subsequence $(n_k)_{k \in \mathbb{N}}$ and $\zeta \in H $ such that
$$\lim_{k \rightarrow \infty}\|z_{n_k}(T_{n_k})-\zeta\Vert_H = 0.$$
Then, by the lower semi-continuity of $U^\alpha$,
$$U^\alpha(\zeta) \leq  \liminf_{k \rightarrow \infty}U^\alpha(z_{n_k}(T_{n_k})) \leq r,$$ which contradicts (\ref{ineq 1 upper lemma 1}).
\end{proof}

\begin{lemma}\label{upper bound lemma 2}
	For any $r >0$, $\delta >0$ and $\rho >0$, let $\lambda $ be as in Lemma \ref{upper bound lemma 1}. Then there exists $\bar{N} \in \mathbb{N} $ large enough such that 
	$$u \in H_{\rho ,r,\delta }(\bar{N})\  \text{implies}\  I_{\bar{N}}(u) \geq r,$$
	where the set $H_{\rho ,r,\delta }(\bar{N})$ is defined for $\bar{N} \in \mathbb{N}$ by
	$$H_{\rho ,r,\delta }(\bar{N}):= \left\{u \in L^\infty([0,\bar{N}];H), \| u(0)\Vert_H \leq \rho , \| u(j)\Vert _H \geq \lambda , j =1,...,\bar{N}\right\}.$$
\end{lemma}
\begin{proof}
	According to (\ref{u Gronwall 1}) in Proposition \ref{skt estimate} we have $$\frac{d}{d t}\| u(t)\Vert _H^2  \leq -\| u(t)\Vert _H^2+\| f(t)\Vert _H^2.$$By the Gronwall inequality
 \begin{equation}\label{upper bound: gronwall}
     \| u(T)\Vert _H^2 \leq \int_{t}^{T} \| f(s)\Vert _H^2 \,ds + e^{-(T-t)}\| u(t)\Vert _H^2. \end{equation}
	Denote
	$$a_i = \| u(i)\Vert _H^2, \ b_i = \int_{i}^{i+1} \| f(s)\Vert _H^2 \,ds, i =0,1,...,\bar{N}-1.$$ Then by (\ref{upper bound: gronwall}), we have 
	$$e^{-1}a_i+b_i \geq a_{i+1}, \ i =0,1,...,\bar{N}-1.$$ Thus, $$e^{-1}\sum_{i =0}^{\bar{N}-1}a_i + \sum_{i =0}^{\bar{N}-1}b_i \geq \sum_{i =1}^{\bar{N}}a_i.$$ Notice that $2 I_{\bar{N}}(u) = \sum_{i =0}^{\bar{N}-1}b_i$. Therefore, if $u \in H_{\rho ,r,\delta }(\bar{N})$, then 
 \begin{align*}
		 2I_{\bar{N}}(u) &\geq \sum_{i =1}^{\bar{N}}a_i - e^{-1}\sum_{i =0}^{\bar{N}-1}a_i \geq (1-e^{-1})(\bar{N}-1)\lambda ^2 + \lambda ^2 - e^{-1}\rho ^2.
	\end{align*}
	Now for any $r,\delta >0$, we can choose $\bar{N}$ large enough such that
	\begin{align*}
		\frac{1}{2}\left[(1-e^{-1})(\bar{N}-1)\lambda ^2 + \lambda ^2 - e^{-1}\rho ^2\right] \geq r,
	\end{align*} where $\lambda $ depends on $r$ and $\delta $. The result is proved.
\end{proof}

With the help of Lemma \ref{exponential estimate 2}, Lemma \ref{upper bound lemma 1} and Lemma \ref{upper bound lemma 2}, we obtain the upper bound. 
\begin{proposition}
	Under Assumption \ref{assumption 2}, the family of the invariant measures $\left\{\nu _\epsilon \right\}_{\epsilon >0}$ of Eq. (\ref{sbe2})) satisfies a large deviations principle upper bound in $H$ with the rate function $U^\alpha$. That is, for any $r \geq 0$ and $\delta , \gamma  >0$, there exists $\epsilon _0>0$ such that 
$$\nu _\epsilon \left(\left\{h \in H: dist_H(h,\Phi^\alpha (r)) >\delta \right\}\right) \leq \exp\left(-\frac{r-\gamma }{\epsilon }\right), \ \epsilon \leq \epsilon _0.$$
\end{proposition}
\begin{proof}
	For any fixed $r>0$, $\delta >0$ and $\gamma >0$ and let $R_r$ and $\epsilon _r$ be as in Lemma \ref{exponential estimate 2}. Since $\nu _\epsilon $ is the invariant measure for Eq. (\ref{sbe2}), for $t \geq  N$ we have 
	\begin{align*}
		\nu _\epsilon \left(\left\{h \in H: dist_H(h,\Phi (r)) \geq \delta \right\}\right) &= \int_{H}\mathbb{P} \left(dist_H(u_\epsilon ^y(t),\Phi (r)) \geq \delta \right)  \,d\nu _\epsilon (y) \\ &= \int_{B_{H^{2\sigma} }^c(0,R_r)}\mathbb{P} \left(dist_H(u_\epsilon ^y(t),\Phi (r)) \geq \delta \right)  \,d\nu _\epsilon (y) \\ &\ \ \ \ + \int_{B_{H^{2\sigma} }(0,R_r)}\mathbb{P} \left(dist_H(u_\epsilon ^y(t),\Phi (r)) \geq \delta, u_\epsilon ^y \in H_{R_r,r,\delta }(N) \right)  \,d\nu _\epsilon (y) \\  &\ \ \ \ + \int_{B_{H^{2\sigma} }(0,R_r)}\mathbb{P} \left(dist_H(u_\epsilon ^y(t),\Phi (r)) \geq \delta, u_\epsilon ^y \notin H_{R_r,r,\delta }(N) \right)  \,d\nu _\epsilon (y) \\ &=: K_1 +K_2 + K_3,
	\end{align*}
	where $N$ is chosen as in Lemma \ref{upper bound lemma 2} such that $$u \in H_{R_r,r,\delta }(N) \ \text{implies} \ I_N(u) \geq r.$$
	Now, thanks to Lemma \ref{exponential estimate 2} we have 
	$$K_1 \leq \nu _\epsilon \left(B^c_{H^{2\sigma} }(0,R_r)\right) \leq \exp(-\frac{r}{\epsilon }), \ \epsilon \leq \epsilon _r.$$ Notice that $\overline{B_{H^{2\sigma} }(0,R_r) }$ is a compact subset of H. Since $H_{R_r,r,\delta }(N)$ is a closed set in $C([0,N];H)$, and since by Corollary \ref{corollary DZ} $\left\{\mathcal{L} (u_\epsilon^x)\right\}_{\epsilon > 0}$ satisfies a Dembo-Zeitouni uniform large deviations principle over compact sets, we infer that there exists $\epsilon _1 >0$ such that \begin{align*}
		K_2 &\leq \sup_{y \in \overline{B_{H^{2\sigma} }(0,R_r) }}\mathbb{P} \left( u_\epsilon ^y \in H_{R_r,r,\delta }(N) \right) \\ &\leq \exp\left(-\frac{1}{\epsilon }\left[\inf_{z \in \overline{B_{H^{2\sigma} }(0,R_r) }}\inf_{h \in H_{R_r,r,\delta }(N)}I_N^z(h) - \gamma \right]\right)\\ &\leq \exp\left(-\frac{r-\gamma }{\epsilon }\right),
	\end{align*}
	for any $\epsilon \leq \epsilon _1$. Here, the third inequality is due to Lemma \ref{upper bound lemma 2}. 

	Finally, let us deal with the third term $K_3$. By the definition of $H_{R_r,r,\delta }(N)$ and the Markov property of $u_\epsilon $, we have 
	\begin{align*}
		K_3 &= \int_{B_{H^{2\sigma} }(0,R_r)}\mathbb{P} \left(\bigcup _{j=1}^N\left\{\| u_\epsilon ^y(j)\Vert _H < \lambda \right\}\bigcap \left\{dist_H(u_\epsilon^y(t), \Phi^\alpha (r)) \geq \delta \right\}\right) \,d\nu _\epsilon (y) \\ &\leq \sum_{j =1}^N \int_{B_{H^{2\sigma} }(0,R_r)}\mathbb{P} \left(\left\{\| u_\epsilon ^y(j)\Vert _H < \lambda \right\}\bigcap \left\{dist_H(u_\epsilon^y(t), \Phi^\alpha (r)) \geq \delta \right\}\right) \,d\nu _\epsilon (y) \\ &\leq \sum_{j=1}^N \sup_{y \in B_H(0,\lambda )}\mathbb{P} \left(dist_H(u_\epsilon ^y(t-j),\Phi^\alpha (r)) \geq \delta \right).
	\end{align*}

	In order to use the ULDP, we need to transfer the event at $t-j$ to an event in $C([0,t-j];H)$. To achieve that, we choose $t \geq T + N$, we then have that for $y \in B_H(0,\lambda )$, if
	$$dist_{L^\infty([0,t-j];H)}\left(u_\epsilon^y, \Psi ^y(r)\right) < \frac{\delta }{2}$$ then $$ \inf\left\{\| u_\epsilon^y-v\Vert _{L^\infty([0,t-j];H)}: v \in L^\infty([0,t-j];H),\  \| v(0)\Vert _H <\lambda , \ I_{t-j}(v) \leq r \right\} < \frac{\delta }{2},$$ where $$\Psi ^y(r) := \left\{v \in L^\infty([0,t-j];H): v(0) =y, I_{t-j}(v) \leq r\right\}.$$ It turns out by applying Lemma \ref{upper bound lemma 1} for $\frac{\delta}{2}$ that $$ dist_H(u_\epsilon^y(t-j), \Phi^\alpha (r)) <\delta.$$ 
Thus, by Theorem \ref{LDPu} and our discussions on the Assumption \ref{assumption 1} about the trace class case in Section \ref{trace class}, there exists $\epsilon _{0,j}$ such that for any $\epsilon \leq \epsilon _{0,j}$,
\begin{align*}
	\sup_{y \in B_H(0,\lambda )}\mathbb{P} \left(dist_H(u_\epsilon ^y(t-j),\Phi^\alpha (r)) \geq \delta \right) & \leq \sup_{y \in B_H(0,\lambda )}\mathbb{P} \left(dist_{L^\infty([0,t-j];H)}\left(u_\epsilon^y, \Psi ^y(r)\right) \geq \frac{\delta }{2}\right) \\ &\leq \exp\left(-\frac{r-\gamma }{\epsilon }\right).
\end{align*}
Therefore, we choose $\epsilon_0= \min(\epsilon _r,\epsilon _1,\epsilon _{0,1},...,\epsilon _{0,N})$, it follows that for any $\epsilon \leq \epsilon _0$
$$K_3 \leq N \exp\left(-\frac{r-\gamma }{\epsilon }\right),$$
Combining the estimates of $K_1$, $K_2$, $K_3$, we complete the proof of this proposition.
\end{proof}

\begin{appendices}
    \section{Besov Space}
    Let us recall the definition of Besov space. Set $D = (0,1)$, for any $l \in \mathbb{N}$ and tempered distribution $u \in \mathcal{S}^\prime(D) $, we define $$\Delta_l u = \sum_{2^{l-1}<|k| \leq 2^l} \langle u, e^{i2k\pi\xi}\rangle e^{i2k\pi\xi}.$$ For $\alpha \in \mathbb{R}$, $p \geq 1$ and $q \geq 1$, we define 
    $$B_{p,q}^\alpha = \left\{u \in \mathcal{S}^\prime(D):\|u\Vert_{B_{p,q}^\alpha} := \left(\sum_{l \in \mathbb{N}}\left( 2^{l\alpha} \| \Delta_l u\Vert_{L^p(D)}\right)^q \right)^{1/q} < + \infty\right\},$$
    with the usual interpretation as $l^\infty$ norm in case $q = \infty$. Then, $B_{p,q}^\alpha$ is a Banach space and $\|u\Vert_{B_{2,2}^\alpha} \simeq \|u\Vert_{H^\alpha}$. Moreover, for $\alpha \in (0,1)$ we have $B_{\infty,\infty}^\alpha = C^\alpha$, the $\alpha$-Holder space, and $\|u\Vert_{B_{\infty,\infty}^\alpha} = \|u\Vert_{C^\alpha}$. We recall the following Besov embedding theorem (see Proposition 2.71, \cite{bahouri2011fourier} ).
    \begin{lemma}\label{lemma:A1}
        Let $1 \leq p_1 \leq p_2 \leq \infty$ and $1\leq q_1 \leq q_2 \leq \infty $. Then, for any $\alpha \in \mathbb{R}$, the space $B_{p_1,q_1}^\alpha$ is continuous embedded into $B_{p_2,q_2}^{\alpha-d\left(\frac{1}{p_1}-\frac{1}{p_2}\right )}$.
    \end{lemma}
    By Chapter 4 of \cite{triebel1995interpolation}, one can extend the multiplication on suitable Besov space and also have the duality properties of Besov space.
    \begin{lemma}\label{lemma:A2} (1) Let $\alpha, \beta \in \mathbb{R}$ and 
    $p,p_1,p_2,q \in [0, \infty]$ satisfy
    $$\frac{1}{p_1} + \frac{1}{p_2}=\frac{1}{p} .$$ Then, the bilinear map $(u;v) \mapsto uv$  extends to a continuous map from $B_{p_1,q}^\alpha \times B_{p_2,q}^\beta $ to $B_{p,q}^ {\alpha\wedge \beta}$, for any $\alpha + \beta >0$.

    (2) Let $\alpha \in (0,1)$, $p,q \in [1,\infty]$, $p^\prime$ and $q^\prime $ be their conjugate exponents, respectively. Then, the mapping $(u;v) \mapsto \int uv \,d \xi $ extends to a continuous bilinear form on $B_{p,q}^\alpha \times B_{p^\prime ,q^\prime }^{-\alpha}$.        
    \end{lemma} 
    
    \end{appendices}

\appendix
\renewcommand\thesection{\normalsize Acknowledgements}
\section{}
We acknowledge the financial supports of EPSRC grant ref EP/S005293/2 and Royal Society Newton Fund grant (ref. NIF\textbackslash R1\textbackslash 221003).

\

\bibliographystyle{siam}
\footnotesize

\end{document}